\documentclass[a4paper, 12pt]{article}

\usepackage[utf8]{inputenc}
\usepackage[russian, english]{babel}
\usepackage[T1]{fontenc}
\usepackage[a4paper, margin=2cm]{geometry}
\usepackage{needspace}

\usepackage{libertine} 
\usepackage{inconsolata} 

\usepackage{amsmath, amsthm, amssymb}
\numberwithin{equation}{section}
\usepackage{graphicx}
\usepackage{enumerate}
\usepackage{authblk}
\usepackage{thmtools}
\usepackage{thm-restate}
\usepackage{float}
\usepackage{tablefootnote}
\usepackage{booktabs}
\usepackage{microtype}
\usepackage{mathtools}
\usepackage[nocompress]{cite}
\usepackage[colorlinks=true, citecolor=red]{hyperref}

\declaretheorem[name=Theorem, numberwithin=section]{theorem}
\declaretheorem[name=Lemma, sibling=theorem]{lemma}

\declaretheorem[name=Conjecture, sibling=theorem]{conjecture}

\DeclareMathOperator{\Cay}{Cay}
\DeclareMathOperator{\CP}{CP}
\DeclareMathOperator{\COP}{COP}
\DeclareMathOperator{\BQC}{BQC}
\DeclareMathOperator{\tr}{tr}

\let\le\leqslant
\let\ge\geqslant

\newcommand{\T}{\mathbb T}
\newcommand{\Z}{\mathbb Z}
\newcommand{\R}{\mathbb R}
\newcommand{\Q}{\mathbb Q}

\title{Measurable obstructions for unmeasurable colourings}

\author{James Davies\thanks{
The author was supported by the Alexander von Humboldt Foundation in the framework of the Alexander von Humboldt Professorship of Daniel Kráľ endowed by the Federal Ministry of Education and Research.}}

\affil{Leipzig University, Germany}
\date{}

\begin{document}

\maketitle

\begin{abstract}
Due to the availability of powerful analytic techniques, vastly superior lower bounds are known for the measurable chromatic number of Euclidean spaces compared to their ordinary chromatic number.
Indeed, even the breakthrough lower bound of 5 for the famous Hadwiger-Nelson problem lagged over 35 years behind that of the measurable setting.
This raises the fundamental question of whether the measurable and ordinary chromatic number of Euclidean spaces differ as conjectured by Székely in 1984 and since reiterated numerous times.

Our main result is that $\overline\alpha(\R^4)=m_1(\R^4)$ and $\chi(\R^4)=\chi^{(m)}(\R^4)$, and for $d\ge 5$ that
\[
\overline\alpha(\Q^d) =
\overline\alpha(\R^d)=m_1(\R^d)
\qquad\text{and}\qquad
\chi(\Q^d) =
\chi(\R^d)=\chi^{(m)}(\R^d).
\]
We remark that the $d\ge 5$ condition is necessary for the extra conclusion on $\Q^d$.
Our theorem also holds for multiple forbidden distances $D=\{d_1, \ldots , d_t\}$ provided that $d_1^2, \ldots , d_t^2 \in \Q$.
As a consequence, we immediately lift numerous measurable chromatic number results into the ordinary chromatic number setting.
We also take the opportunity to further optimize the new bounds.

For multiple distances $D$ with $|D|\le k$, Erd\H{o}s asked for estimates on the chromatic number $\chi_D(\R^d)$ of $\R^d$ with forbidden distances $D$, and
in particular, whether $\sup_{|D|=k}\chi_D(\R^d)$ grows exponentially in $k$.
By a theorem of Bukh, we obtain for $d \ge 4$ that
\[
\inf_{|D|=k}\overline\alpha_D(\R^d) = m_1(\R^d)^k
\qquad
\text{and thus}
\qquad
\sup_{|D|=k}\chi_D(\R^d) \ge m_1(\R^d)^{-k}.
\]
This also implies that for any finite colouring of $\R^4$ there exists some rational $R$ such that for every rational $r\ge R$ there is a monochromatic pair of points whose distance is exactly $\sqrt{r}$.

We also make progress on another problem of Erd\H{o}s by proving that
\[
(2+o(1))^d
\le
\chi(\R^d)
\le
\left(\frac{3\sqrt{3}}{4} + o(1) \right)^d.
\]
The upper bound is given by a surprisingly simple probabilistic argument.
The lower bound argument is also used to vastly improve the bounds for $\chi(\R^d)$ for small $d\ge 5$.

We expect that our techniques could be developed much further. This includes the possibility of extending our main theorem that $\chi(\R^d)=\chi^{(m)}(\R^d)$ for $d\ge 4$ to $d=3$ or possibly even to $d=2$ to tackle the Hadwiger-Nelson problem.
\end{abstract}

\pagebreak

\section{Introduction}

The Hadwiger-Nelson problem asks for the chromatic number of the plane:
the minimum number of colours needed to colour $\R^2$ so that points at
distance $1$ receive different colours. See Soifer~\cite{soifer2009coloring} for its rich history.
The Moser spindle~\cite{mosermoser1961problem10} gives the classical lower bound $4$, and de Grey's breakthrough~\cite{degrey2018plane} raised this to $5$ in 2018. 
A hexagonal tiling gives the upper bound $7$
\cite{soifer2009coloring}, so
\[
 5\le\chi(\R^2)\le7.
\]
More generally, one may ask for the chromatic number $\chi(\R^d)$ of the
unit-distance graph in $d$-dimensional Euclidean space.  Frankl and
Wilson~\cite{franklwilson1981intersection} proved Erd\H{o}s's conjecture
that this number grows exponentially with the dimension.
Erd\H{o}s~\cite{erdos1985combinatorialgeometry} asked more precisely whether
$\lim_{d\to\infty}\chi(\R^d)^{1/d}$ exists and, if so, what its value is.
Subsequent refinements of Raigorodskii~\cite{raigorodskii2000chromatic}, together with the upper bound of Larman and Rogers~\cite{larmanrogers1972}, gave
\[
 (1.239\ldots+o(1))^d\le\chi(\R^d)\le(3+o(1))^d.
\]

Requiring the colour classes to be Lebesgue measurable gives the
measurable chromatic number $\chi^{(m)}(\R^d)$.
This significant additional assumption makes powerful analytic methods available, and the resulting lower bounds have
historically been much stronger. 
Falconer~\cite{falconer1981realization}
proved $\chi^{(m)}(\R^2)\ge5$ in 1981, more than 35 years before the same
lower bound was known for arbitrary colourings.  In dimension three,
Nechushtan~\cite{nechushtan2002space} proved $\chi(\R^3)\ge6$, whereas
de Oliveira Filho and Vallentin~\cite{oliveiravallentin2010fourier} proved
$\chi^{(m)}(\R^3)\ge7$.  For other small dimensions, the measurable
lower bounds have been increasingly stronger than the ordinary ones; see
Table~\ref{tab:previous-new-lowdim}.  The disparity also persists
asymptotically: Bachoc, Passuello and Thiery~\cite{bachoc2015density}
proved that $\chi^{(m)}(\R^d) \ge (1.268\ldots+o(1))^d$.

The distinction is even more pronounced when several distances are forbidden.
For $D\subseteq(0,\infty)$ and $X\subseteq\R^d$, write $\chi_D(X)$ for
the chromatic number of the graph whose edges join points at distances
in $D$, and write $\chi_D^{(m)}(\R^d)$ for its measurable chromatic
number.  In dimensions at least two, positive-density distance theorems
imply that $\chi_D^{(m)}(\R^d)=\infty$ whenever $D$ is unbounded.
Furstenberg, Katznelson and Weiss~\cite{furstenbergkatznelsonweiss1990}
gave an ergodic-theoretic proof; Bourgain~\cite{bourgain1986szemeredi}
used Fourier analysis, and Falconer and
Marstrand~\cite{falconermarstrand1986distances} gave a direct geometric
argument.  The same conclusion
follows from Steinhaus's theorem~\cite{steinhaus1920distances} when
$\inf D=0$.  For arbitrary colourings, even the case of odd distances $D=\{1,3,\ldots\}$ as in Rosenfeld's odd-distance
problem~\cite{rosenfeld2008problems} remained open for nearly three
decades. This was proven by the
author~\cite{davies2023odd}, and Davies, McCarty and
Pilipczuk~\cite{dmp2024} subsequently obtained analogous results for prime
and polynomial distances.

Alongside colouring, we consider the density of distance-avoiding sets.
A set is $D$-avoiding if it contains no pair of points at a distance in
$D$.  Write $m_D(\R^d)$ for the supremum of the upper densities of
measurable $D$-avoiding sets, and $\overline\alpha_D(X)$ for the infimum
of the independence ratios of the nonempty finite induced distance
graphs on $X$.  We write $m_1$ and omit the subscript on
$\overline\alpha$ when $D=\{1\}$.

For $k\ge1$, put
\[
 F_d(k):=\sup_{|D|=k} \chi_D(\R^d).
\]
Erd\H{o}s~\cite{erdos1985combinatorialgeometry,croftfalconerguy1991}
asked how $F_d(k)$ grows with $k$, in particular whether it is exponential.  A square-lattice construction gives
$F_2(k)\ge c k\sqrt{\log k}$ for an absolute constant $c>0$, whereas
taking the product of $k$ rescaled unit-distance colourings gives
$F_2(k)\le7^k$.
Bukh~\cite{bukh2008measurable} proved that, for $d\ge2$, the infimum
of $m_D(\R^d)$ over $k$-element distance sets is $m_1(\R^d)^k$,
approached by sets of distances whose consecutive distances become widely separated. In particular, the maximum measurable chromatic number grows exponentially in $k$ since $\chi_D^{(m)} \ge 1/m_D(\R^d)$.

Sz\'ekely~\cite{szekely1984measurable} conjectured that the ordinary and
measurable chromatic numbers of Euclidean spaces can differ.  Croft,
Falconer and Guy~\cite[Problem~G10]{croftfalconerguy1991} discuss this, noting Sz\'ekely's work as evidence against equality.
The question and related measurability phenomena recur throughout the
literature; see
\cite{szekelywormald1989,soifer2009coloring,shelahsoifer2003choice,soifershelah2004examples,soifer2005choice,payne2007ambiguous,bachocnebeoliveiravallentin2009,oliveiravallentin2010fourier,bachocdecorteoliveiravallentin2014,keleti2016better,ambrus2024density,kechrismarks2020}.
Our main result proves that both the density and chromatic parameters coincide with
their measurable counterparts in dimension at least four.
In dimensions at least five, they are
even already determined by the rational points.
Furthermore, this result holds both for unit distance graphs and multiple distance graphs provided that the squares of these distances are rational.

\Needspace{12\baselineskip}
\begin{theorem}\label{thm:main}
Let $D=\{d_1,\ldots,d_t\}\subset (0,\infty)$ be finite and
nonempty with $d_1^2, \ldots ,d_t^2\in \Q$.  Then
\[
 \overline\alpha_D(\R^4)=m_D(\R^4)
 \qquad\text{and}\qquad
 \chi_D(\R^4)=\chi_D^{(m)}(\R^4),
\]
furthermore, for $d\ge 5$,
\[
\overline\alpha_D(\Q^d)=
\overline\alpha_D(\R^d)=m_D(\R^d)
 \qquad\text{and}\qquad
 \chi_D(\Q^d)= \chi_D(\R^d)=\chi_D^{(m)}(\R^d).
\]
\end{theorem}

Unit-distance colouring of rational spaces $\Q^d$ also has
substantial literature, including arithmetic constructions and
small-dimensional bounds
\cite{woodall1973distances,chilakamarri1988rational,chilakamarri1990five,bendaperles2000colorings,cibulka2008chromatic}
and asymptotic estimates
\cite{raigorodskii2001borsuk,raigorodskiishitova2008multiple,ponomarenkoraigorodskii2013rational,ponomarenkoraigorodskii2015rational}.
In particular, $\chi(\Q^2)=\chi(\Q^3)=2$ and $\chi(\Q^4)=4$
\cite{woodall1973distances,bendaperles2000colorings}, so the restriction
$d\ge5$ in our rational-point conclusion is necessary.
Our four-dimensional construction instead uses a quadratic extension;
in the unit-distance case, there exists a positive nonsquare integer $n$
such that $\chi(\R^4)=\chi(\Q(\sqrt n)^4)$.
We remark that a conjecture of Bukh \cite{bukh2008measurable} suggests that some kind of extra algebraic condition on the multiple distances $D$ as in Theorem \ref{thm:main} may be necessary.

Theorem \ref{thm:main} has numerous applications lifting bounds from the easier measurable setting into the ordinary density or chromatic number settings.

For the multiple-distance problem in higher dimensions, Raigorodskii
\cite{raigorodskii2001borsuk} obtained lower bounds of the form
$(c_1k)^{c_2d}$ with $c_1,c_2>0$; Berdnikov
\cite{berdnikov2016several} showed that any $c_2<1/2$ is admissible.
Naslund~\cite{naslund2023multiple} made this dependence quantitative:
for each fixed $k$,
\[
 F_d(k)\ge\bigl(\Gamma_\chi\sqrt{k+1}+o(1)\bigr)^d,
 \qquad \Gamma_\chi=0.79983\ldots,
 \qquad(d\to\infty).
\]
These earlier bounds did not settle whether the dependence on the
number of distances is polynomial or exponential in each fixed dimension.

Combining Bukh's \cite{bukh2008measurable} measurable result with Theorem~\ref{thm:main} gives the following exponential bound.

\begin{theorem}\label{thm:intro-multiple-distances}
For every $d\ge 4$, $k\ge 1$, and $\epsilon>0$, there exists a real $R$ such that if $D=\{d_1,\ldots,d_k\}$ with $d_2/d_1,\ldots ,d_k/d_{k-1} \ge R$ and $d_1^2,\ldots ,d_k^2 \in \Q$, then
\[
 \overline\alpha_D(\R^d)
 \le
 (m_1(\R^d)+\epsilon)^k,
\]
and in particular
\[
\chi_D(\R^d) \ge (m_1(\R^d) + \epsilon)^{-k}.
\]
\end{theorem}

An immediate corollary of this is the following theorem. For dimensions at least four this significantly extends the results for odd \cite{davies2023odd} and for prime and polynomial \cite{dmp2024} distances in the plane.

\begin{theorem}\label{thm:intro-large-distances}
For every finite colouring of $\R^4$, there exists $R>0$ such that, for
every $r>0$ with $r^2\in\Q$ and $r^2\ge R$, there is a monochromatic
pair of points at distance $r$.
\end{theorem}

For unit distances,
Theorem~\ref{thm:main} of course makes the measurable lower bound available for
arbitrary colourings in dimensions at least four.
In particular, by Bachoc, Passuello and Thiery's~\cite{bachoc2015density} theorem, we already have that $\chi(\R^d) \ge (1.268\ldots+o(1))^d$, improving on Raigorodskii's~\cite{raigorodskii2000chromatic} bound.
We shall also take the opportunity to improve these analytic
bounds themselves.
We prove the following

\begin{theorem}\label{thm:intro-asymptotic}
\[
m_1(\R^d) \le (2+o(1))^{-d}
\qquad \text{and therefore} \qquad
 \chi(\R^d)\ge(2+o(1))^d.
\]
\end{theorem}

For small dimensions, Section~\ref{sec:mone-bounds} gives new certified
density bounds for $5\le d\le24$, and hence chromatic lower bounds by
Theorem~\ref{thm:main}.  For example, we obtain
$\chi(\R^5)\ge24$, $\chi(\R^{10})\ge729$, and
$\chi(\R^{24})\ge10{,}515{,}016$.
Table~\ref{tab:previous-new-lowdim} compares all these bounds with the
previous best known ones.  In dimension four, the new columns record the transfer
of the previous measurable bounds of DeCorte, de Oliveira Filho and
Vallentin~\cite{decorte2022complete}.

\begin{table}[H]
\centering
\small
\renewcommand{\arraystretch}{1.08}
\begin{tabular*}{\textwidth}{@{\extracolsep{\fill}}
  c
  r@{\,}l
  r@{\,}l
  r@{\,}l
  r@{\,}l
  r
  r
  @{}}
\toprule
& \multicolumn{8}{c}{\textbf{Previous bounds}}
& \multicolumn{2}{c}{\textbf{New bounds}}\\
\cmidrule(lr){2-9}\cmidrule(l){10-11}
$d$
& \multicolumn{2}{c}{$\overline\alpha(\R^d)$}
& \multicolumn{2}{c}{$m_1(\R^d)$}
& \multicolumn{2}{c}{$\chi(\R^d)$}
& \multicolumn{2}{c}{$\chi^{(m)}(\R^d)$}
& {$\overline\alpha(\R^d)=m_1(\R^d)$}
& {$\chi(\R^d)=\chi^{(m)}(\R^d)$}\\
\midrule
4  & 0.16666667 & {\scriptsize\cite{balko2020almost}}          & 0.09857010 & {\scriptsize\cite{decorte2022complete}} & 10  & {\scriptsize\cite{voronovpersonalcommunication}} & 11   & {\scriptsize\cite{decorte2022complete}} & 0.09857010 & 11 \\
5  & 0.12500000 & {\scriptsize\cite{balko2020almost}}          & 0.06244850 & {\scriptsize\cite{decorte2022complete}} & 10  & {\scriptsize\cite{voronovpersonalcommunication}}      & 17   & {\scriptsize\cite{decorte2022complete}} & 0.042234949 & 24 \\
6  & 0.07659575 & {\scriptsize\cite{parts2026polymath16}}      & 0.04503250 & {\scriptsize\cite{decorte2022complete}} & 14  & {\scriptsize\cite{parts2026polymath16}}      & 23   & {\scriptsize\cite{decorte2022complete}} & 0.019296807 & 52 \\
7  & 0.05987395 & {\scriptsize\cite{parts2026polymath16}}      & 0.02607820 & {\scriptsize\cite{decorte2022complete}} & 18  & {\scriptsize\cite{parts2026polymath16}}      & 39   & {\scriptsize\cite{decorte2022complete}} & 0.010261216 & 98 \\
8  & 0.04033215 & {\scriptsize\cite{parts2026r8}}              & 0.01909450 & {\scriptsize\cite{decorte2022complete}} & 25  & {\scriptsize\cite{parts2026r8}}              & 53   & {\scriptsize\cite{decorte2022complete}} & 0.0057762907 & 174 \\
9  & 0.04033215 & {\scriptsize\cite{parts2026r8}}              & 0.01465770 & {\scriptsize\cite{bachoc2015density}}   & 30  & {\scriptsize\cite{parts2026polymath16}}      & 69   & {\scriptsize\cite{bachoc2015density}}   & 0.0028018974 & 357 \\
10 & 0.03333334 & {\scriptsize\cite{cherkashin2018small}}      & 0.01116210 & {\scriptsize\cite{bachoc2015density}}   & 33  & {\scriptsize\cite{parts2026polymath16}}      & 90   & {\scriptsize\cite{bachoc2015density}}   & 0.0013718654 & 729 \\
11 & 0.02878788 & {\scriptsize\cite{cherkashin2018small}}      & 0.00862918 & {\scriptsize\cite{bachoc2015density}}   & 35  & {\scriptsize\cite{cherkashin2018small}}      & 116  & {\scriptsize\cite{bachoc2015density}}   & $7.2303784\times10^{-4}$ & 1,384 \\
12 & 0.02680653 & {\scriptsize\cite{exooismailescu2014small}}  & 0.00611112 & {\scriptsize\cite{bachoc2015density}}   & 37  & {\scriptsize\cite{cherkashin2018small}}      & 164  & {\scriptsize\cite{bachoc2015density}}   & $3.9758787\times10^{-4}$ & 2,516 \\
13 & 0.02185315 & {\scriptsize\cite{cherkashin2018small}}      & 0.00394335 & {\scriptsize\cite{bachoc2015density}}   & 46  & {\scriptsize\cite{cherkashin2018small}}      & 254  & {\scriptsize\cite{bachoc2015density}}   & $1.7808728\times10^{-4}$ & 5,616 \\
14 & 0.01923077 & {\scriptsize\cite{cherkashin2018small}}      & 0.00300288 & {\scriptsize\cite{bachoc2015density}}   & 52  & {\scriptsize\cite{cherkashin2018small}}      & 334  & {\scriptsize\cite{bachoc2015density}}   & $7.8820187\times10^{-5}$ & 12,688 \\
15 & 0.01703297 & {\scriptsize\cite{cherkashin2018small}}      & 0.00242258 & {\scriptsize\cite{bachoc2015density}}   & 59  & {\scriptsize\cite{cherkashin2018small}}      & 413  & {\scriptsize\cite{bachoc2015density}}   & $4.2750273\times10^{-5}$ & 23,392 \\
16 & 0.01517858 & {\scriptsize\cite{cherkashin2018small}}      & 0.00161646 & {\scriptsize\cite{bachoc2015density}}   & 66  & {\scriptsize\cite{cherkashin2018small}}      & 619  & {\scriptsize\cite{bachoc2015density}}   & $2.4066521\times10^{-5}$ & 41,552 \\
17 & 0.01360295 & {\scriptsize\cite{cherkashin2018small}}      & 0.00110487 & {\scriptsize\cite{bachoc2015density}}   & 74  & {\scriptsize\cite{cherkashin2018small}}      & 906  & {\scriptsize\cite{bachoc2015density}}   & $1.1553067\times10^{-5}$ & 86,558 \\
18 & 0.01225491 & {\scriptsize\cite{cherkashin2018small}}      & 0.00084949 & {\scriptsize\cite{bachoc2015density}}   & 82  & {\scriptsize\cite{cherkashin2018small}}      & 1178 & {\scriptsize\cite{bachoc2015density}}   & $5.6103486\times10^{-6}$ & 178,243 \\
19 & 0.01109392 & {\scriptsize\cite{cherkashin2018small}}      & 0.00074601 & {\scriptsize\cite{bachoc2015density}}   & 91  & {\scriptsize\cite{cherkashin2018small}}      & 1341 & {\scriptsize\cite{bachoc2015density}}   & $2.9406514\times10^{-6}$ & 340,061 \\
20 & 0.01008772 & {\scriptsize\cite{cherkashin2018small}}      & 0.00046910 & {\scriptsize\cite{bachoc2015density}}   & 100 & {\scriptsize\cite{cherkashin2018small}}      & 2132 & {\scriptsize\cite{bachoc2015density}}   & $1.5393277\times10^{-6}$ & 649,635 \\
21 & 0.00877193 & {\scriptsize\cite{bogolubsky2019remark}}     & 0.00031431 & {\scriptsize\cite{bachoc2015density}}   & 114 & {\scriptsize\cite{bogolubsky2019remark}}     & 3182 & {\scriptsize\cite{bachoc2015density}}   & $6.7317355\times10^{-7}$ & 1,485,502 \\
22 & 0.00722395 & {\scriptsize\cite{bogolubsky2019remark}}     & 0.00024622 & {\scriptsize\cite{bachoc2015density}}   & 139 & {\scriptsize\cite{bogolubsky2019remark}}     & 4062 & {\scriptsize\cite{bachoc2015density}}   & $3.3140690\times10^{-7}$ & 3,017,439 \\
23 & 0.00584796 & {\scriptsize\cite{bogolubsky2019remark}}     & 0.00021227 & {\scriptsize\cite{bachoc2015density}}   & 171 & {\scriptsize\cite{bogolubsky2019remark}}     & 4712 & {\scriptsize\cite{bachoc2015density}}   & $1.7655897\times10^{-7}$ & 5,663,831 \\
24 & 0.00478469 & {\scriptsize\cite{bogolubsky2019remark}}     & 0.00018437 & {\scriptsize\cite{bachoc2015density}}   & 209 & {\scriptsize\cite{bogolubsky2019remark}}     & 5424 & {\scriptsize\cite{bachoc2015density}}   & $9.5102091\times10^{-8}$ & 10,515,016 \\
\bottomrule
\end{tabular*}
\caption{Previous and new bounds in dimensions $4$--$24$.
The $\overline\alpha$ and $m_1$ columns are upper bounds; the $\chi$
and $\chi^{(m)}$ columns are lower bounds.  The new columns use
Theorem~\ref{thm:main} to transfer the previous measurable bounds when
$d=4$, and the certificates of Section~\ref{sec:mone-bounds} when
$5\le d\le24$.  The latter density bounds are rounded upward to eight
significant figures, and their integer bounds use the exact rational
certificate parameters in the computational supplement~\cite{MoneData}.}
\label{tab:previous-new-lowdim}
\end{table}

For upper bounds, we also give a surprisingly simple probabilistic
argument reducing the previous
upper-bound base $3$ of Larman and
Rogers~\cite{larmanrogers1972} to $3\sqrt3/2 = 2.5981\ldots$.

\begin{theorem}\label{thm:intro-upper-bound}
	\[
	\chi(\R^d)\le\left(\frac{3\sqrt3}{2}+o(1)\right)^d.
	\]
\end{theorem}

Together with Theorems~\ref{thm:intro-asymptotic},
this gives the following bounds for the two classical problems:
\[
(2+o(1))^d\le\chi(\R^d)
\le\left(\frac{3\sqrt3}{2}+o(1)\right)^d
\]
and
\[
\left(\frac{3\sqrt3}{2}+o(1)\right)^{-d}
\le m_1(\R^d)\le(2+o(1))^{-d}.
\]
We boldly conjecture that one of these two bases is tight and cowardly do not specify which one.

\medskip

Our starting point for the proof of Theorem~\ref{thm:main} is the completely-positive framework of DeCorte,
de Oliveira Filho and Vallentin~\cite{decorte2022complete}.  Their
characterization of measurable independence density gives the finite
certificates in Theorem~\ref{prop:scalar-certificate}.
Theorem~\ref{prop:chromatic-certificate} is a chromatic analogue built
on the same framework.  These certificates express density upper bounds
and failure of measurable colourability through radial inequalities and
finite rational configuration constraints.  We prove the corresponding bounds
for lattice Cayley graphs of $\Z^r$, then construct lattice embeddings that
approximate the radial inequalities while preserving the configurations.
The four-dimensional embedding uses a quadratic extension; in
dimensions at least five it lies in $\Q^d$.  These different arithmetic
constructions account for the two conclusions of Theorem~\ref{thm:main}.

We hope that it will be possible to extend our techniques to dimensions $2$ and $3$.
Besides the obvious application to lower bounds, a particularly tempting aspect of Theorem \ref{thm:main} is for $d\ge 5$ it reduces the problem of colouring $\R^d$ to that of colouring $\Q^d$.
This opens up the possibility of using arithmetic techniques to obtain improved upper bounds on $\chi(\R^d)$.
For this reason, it would be particularly interesting to have such a proof that $\chi(\Q^d)$ is finite, even if it does not improve the known bounds.
Although we cannot reduce colouring $\R^d$ to colouring $\Q^d$ for $d=2,3,4$, there is still hope of reducing the problem colouring a simpler subset of $\R^d$. For instance, in the four dimensional case the proof of Theorem \ref{thm:main} shows that there exists some positive integer $n$
such that $\chi(\R^4)=\chi(\Q(\sqrt n)^4)$. We conjecture that this also holds in the three dimensional case.
In the plane we suspect that a more complicated extension of $\Q^2$ might be required.
Nevertheless, the problem of determining the chromatic number of $\Q(\sqrt{n})^2$ has been well studied \cite{Johnson1987,Fischer1990,payne2007ambiguous,Moorhouse2010}.
In particular, it is known that $\Q(\sqrt{n})^2$ is four-colourable for $n \not\equiv 47, 143, 167 \pmod{168}$.
It would be particularly interesting to determine if the remaining cases are six-colourable and if any of them contain a five chromatic unit distance graph.

Given this new approach to improving both the upper and lower bounds for the Hadwiger-Nelson problem, we boldly conjecture the following.
\begin{conjecture}
	$\chi(\R^2)=6$.
\end{conjecture}

Section~\ref{sec:overview} gives the preliminary definitions and
certificate inputs.  Section~\ref{sec:lattice-lp} proves the lattice
density and colouring bounds, and
Section~\ref{sec:lattice-sphere-estimates} supplies the
four-dimensional lattice-sphere estimates, including their quaternionic
form.  Section~\ref{sec:four-dimensional-arithmetic} combines these
estimates with a planar filter to transfer radial inequalities in
dimension four.  Section~\ref{sec:higher-dimensional-arithmetic}
uses the same filter with integer Euclidean similitudes and rational
four-planes to obtain the rational embedding in dimensions at least
five.  Section~\ref{sec:main-theorem} applies these transfer results
to the certificates and completes the proof of Theorem~\ref{thm:main}.
Sections~\ref{sec:high-dimensional} and~\ref{sec:mone-bounds} give,
respectively, the asymptotic lower bounds and numerical certificates
for dimensions $5$--$24$.
Section~\ref{sec:euclidean-upper-bound} proves
Theorem~\ref{thm:intro-upper-bound}
.
Appendix~\ref{app:scalar-certificate} gives the finite-point
extraction for density certificates, and
Appendix~\ref{app:chromatic-certificate} proves the chromatic converse
through completely-positive extraction and separation.

\section{Preliminaries and certificate inputs}\label{sec:overview}

We collect the notation and graph parameters used in the proof, then
state the measurable certificate theorems that supply the inputs for
Sections~\ref{sec:lattice-lp}--\ref{sec:main-theorem}.

We write $[n]:=\{1,\ldots,n\}$ and $\T:=\R/\Z$.  All vectors are column vectors.  For $x,y\in\R^d$ we write
$x\cdot y$ for the Euclidean inner product, $\|x\|$ for the Euclidean norm,
and $\|x\|_\infty:=\max_i|x_i|$.  Put
$S^{d-1}:=\{x\in\R^d:\|x\|=1\}$.  Distances on $\T^r$ are measured using the
quotient of the Euclidean metric on $\R^r$.

\Needspace{5\baselineskip}
For a ring $K$, let $M_{m\times n}(K)$ denote the set of $m\times n$ matrices
with entries in $K$, and put $M_n(K):=M_{n\times n}(K)$.  We write $A^T$ for
the transpose of $A$, $\tr(A)$ for its trace, and $\|A\|$ for its
Euclidean operator norm.  The Kronecker delta $\delta_{ij}$ is $1$ if
$i=j$ and $0$ otherwise.  We use
$I_d$ for the $d\times d$ identity matrix, $e_j$ for the $j$th standard basis
vector, and $A\oplus B$ for block-diagonal direct sum.  Also
\[
 SO(d):=\{A\in M_d(\R):A^TA=I_d,\ \det A=1\},
 \qquad SO(d,\Q):=SO(d)\cap M_d(\Q).
\]
A matrix $B$ is a Euclidean similitude of ratio $\lambda>0$ if
$B^TB=\lambda^2I$, equivalently $\|Bv\|=\lambda\|v\|$ for every $v$.

For a graph $G$, $\chi(G)$ denotes its chromatic number, and
$\overline\alpha(G)$ is the infimum of the independence ratios of its
nonempty finite induced subgraphs.

If $\Gamma$ is an abelian group and
$C\subseteq\Gamma\setminus\{0\}$ satisfies $C=-C$, then
$\Cay(\Gamma,C)$ is the graph in which $x$ and $y$ are adjacent when
$y-x\in C$.  We write $\Z_n:=\Z/n\Z$.
For a set $E$, $\mathbf1_E(x)$ is $1$ if $x\in E$ and $0$ otherwise;
on a finite index set we also view $\mathbf1_E$ as a vector.  For a
statement $P$, $\mathbf1_P$ is $1$ if $P$ holds and $0$ otherwise.
We use $\mathbf1$ for the all-ones vector, with its index set understood.

For finite nonempty $D\subseteq(0,\infty)$, the parameters defined in the
introduction satisfy the elementary comparisons
\begin{equation}\label{eq:mlealpha-base}
 m_D(\R^d)\le \overline\alpha_D(\R^d),
 \qquad \chi_D^{(m)}(\R^d)\ge\frac1{m_D(\R^d)}.
\end{equation}

Restriction to $\Q^d$ also gives
\[
 \overline\alpha_D(\Q^d)\ge\overline\alpha_D(\R^d)\ge m_D(\R^d),
 \qquad
 \chi_D(\Q^d)\le\chi_D(\R^d)\le\chi_D^{(m)}(\R^d).
\]
The main task in Theorem~\ref{thm:main} is to prove the reverse
comparisons.  We do this by transferring finite certificates for the
measurable problems to unrestricted distance graphs.  The
completely-positive framework of DeCorte, de Oliveira Filho and
Vallentin~\cite{decorte2022complete} supplies the density certificates
in Theorem~\ref{prop:scalar-certificate} and underlies their chromatic
analogue, Theorem~\ref{prop:chromatic-certificate}.  Both combine radial
inequalities with matrix constraints on finite configurations.

To describe the radial inequalities, let $\sigma_d$ be normalized
surface measure on $S^{d-1}$.  We use the normalized spherical transform
\[
 \Omega_d(t):=\int_{S^{d-1}}e^{2\pi i t\omega_1}\,d\sigma_d(\omega).
\]
Rotation invariance gives the useful vector form
\[
 \int_{S^{d-1}}e^{2\pi i u\cdot\omega}\,d\sigma_d(\omega)
 =\Omega_d(\|u\|)\qquad(u\in\R^d).
\]
It is real and continuous, $\Omega_d(0)=1$, $|\Omega_d(t)|\le1$, and for
$d\ge2$, $\Omega_d(t)\to0$ as $t\to\infty$.

The matrix constraints use the following cones.  For a finite set $I$, put
\[
 \COP(I):=\{Z\in\R^{I\times I}:Z=Z^T,\ x^TZx\ge0
              \text{ for every }x\in\R_{\ge0}^I\},
\]
\[
 \BQC^*(I):=\{Z\in\R^{I\times I}:Z=Z^T,\ x^TZx\ge0
              \text{ for every }x\in\{0,1\}^I\}.
\]
We write $J_I$ for the all-ones matrix indexed by $I$ and use the
Frobenius pairing
\[
 \langle A,B\rangle:=\sum_{x,y\in I}A(x,y)B(x,y).
\]
Completely positive matrices are finite sums $bb^T$ with
$b\in\R_{\ge0}^I$; we denote their cone by $\CP(I)$.
The cones $\CP(I)$ and $\COP(I)$ are dual under this pairing.
Similarly, $\BQC^*(I)$ is dual to the cone generated by $xx^T$ with
$x\in\{0,1\}^I$.  Since Boolean vectors are nonnegative,
$\COP(I)\subseteq\BQC^*(I)$.
Validity below only uses the Boolean tests, whereas the converses supply
copositive witnesses.  Both matrix classes are preserved by relabelling
their indices and by positive scaling.

For finite graphs, exact copositive formulations are known for the
independence number, by De Klerk and Pasechnik~\cite{deklerkpasechnik2002},
and for the chromatic number, by Gvozdenovi{\'c} and
Laurent~\cite{gvozdenoviclaurent2008psi}.  The first measurable input below
is a finite-point form of the completely-positive characterization of
independence density due to DeCorte, de Oliveira Filho and
Vallentin~\cite{decorte2022complete}.\footnote{Their work does not quite prove the full finite-point converse
stated below.  Appendix~\ref{app:scalar-certificate} fills the small gap
by replacing thick copositive constraints uniformly by finitely many
rational point constraints, and includes the extension to finite $D$.}

\begin{theorem}[DeCorte, de Oliveira Filho,
Vallentin~\cite{decorte2022complete}]\label{prop:scalar-certificate}
Let $D=\{d_1,\ldots,d_t\}\subseteq(0,\infty)$ be finite and nonempty and let
$d\ge2$.  Suppose finite nonempty sets $U_1,\ldots,U_\ell\subseteq\R^d$,
matrices $Z_q\in\BQC^*(U_q)$, and real numbers
$z_0,a_1,\ldots,a_t$ satisfy
\begin{align}
 z_0+\sum_{s=1}^ta_s-\sum_{q=1}^\ell\langle Z_q,J_{U_q}\rangle&\ge1,
 \label{eq:scalar-cert-normalization}\\
 z_0-\sum_{q=1}^\ell\tr (Z_q)+\sum_{s=1}^ta_s\Omega_d(d_s\tau)
 -\sum_{q=1}^\ell \sum_{\substack{u,v\in U_q\\u\ne v}}Z_q(u,v)\Omega_d(\|u-v\|\tau)
 &\ge0
 \qquad(\tau\ge0).\label{eq:scalar-cert-radial}
\end{align}
Then
\[
 m_D(\R^d)\le z_0.
\]
Conversely, for every $\varepsilon>0$ there is such a certificate in which
$U_q\subseteq\Q^d$, $Z_q\in\COP(U_q)\subseteq\BQC^*(U_q)$, and
\[
 z_0<m_D(\R^d)+\varepsilon.
\]
\end{theorem}

Writing $F(\tau)$ for the left-hand side of
\eqref{eq:scalar-cert-radial}, the certificate requires $F(\tau)\ge0$
for all $\tau\ge0$ and $F(0)\ge1$, since $\Omega_d(0)=1$.
The converse identifies $m_D(\R^d)$ with the infimum of $z_0$ over
certificates with rational point sets and copositive matrices.

For the chromatic problem, we use a colour-labelled version of the same
completely-positive approach.  Simplex inner products record whether
two colours agree.  Translation and rotation averaging gives radial
functions, with labelled Boolean-quadratic inequalities providing
finite-configuration constraints.
The matrix indices now record both a point and a colour.
For $k\ge2$, finite $U\subseteq\R^d$, and a symmetric matrix $Z$ indexed by
$U\times[k]$, write $Z_{ui,vj}:=Z((u,i),(v,j))$ and define
\begin{equation}\label{eq:AZ}
 A_Z:=\frac1{k^2}\sum_{u,v\in U}\sum_{i,j=1}^k Z_{ui,vj},
\end{equation}
\begin{equation}\label{eq:SZ}
 S_Z(u,v):=\frac1{k^2}\sum_{i,j=1}^k
                 Z_{ui,vj}(k\delta_{ij}-1).
\end{equation}
$A_Z$ records the constant contribution from averaging colour labels,
and the symmetric matrix $S_Z$ supplies the pair coefficients.
The finite-configuration inequality comes from the Boolean-quadratic
condition on $Z$; no positivity condition is needed on $S_Z$ itself.

\begin{theorem}\label{prop:chromatic-certificate}
Let $D=\{d_1,\ldots,d_t\}\subseteq(0,\infty)$ be finite and nonempty, let
$d\ge2$, and let $k\ge2$.  Suppose finite nonempty sets
$U_1,\ldots,U_\ell\subseteq\R^d$, matrices
$Z_q\in\BQC^*(U_q\times[k])$, and real numbers $z,a_1,\ldots,a_t$ satisfy
\begin{align}
z-\frac1{k-1}\sum_{s=1}^ta_s+\sum_{q=1}^\ell A_{Z_q}&<0,
 \label{eq:chromatic-cert-gap}\\
 z-\sum_{q=1}^\ell\tr (S_{Z_q})+\sum_{s=1}^ta_s\Omega_d(d_s\tau)
 -\sum_{q=1}^\ell
 \sum_{\substack{u,v\in U_q\\u\ne v}}S_{Z_q}(u,v)\Omega_d(\|u-v\|\tau)
 &\ge0
 \qquad(\tau\ge0).\label{eq:chromatic-cert-radial}
\end{align}
Then
\[
\chi_D^{(m)}(\R^d)>k.
\]
Conversely, if $\chi_D^{(m)}(\R^d)>k$, there is such a
certificate in which $U_q\subseteq\Q^d$ and
$Z_q\in\COP(U_q\times[k])\subseteq\BQC^*(U_q\times[k])$.
\end{theorem}

For each fixed $k$, Theorem~\ref{prop:chromatic-certificate} characterizes
failure of measurable $k$-colourability by a finite certificate.
The strict negative gap in \eqref{eq:chromatic-cert-gap} excludes a
$k$-colouring and allows a small increase of $z$ to absorb the
transference error in Section~\ref{sec:main-theorem}.

Only the converse direction of Theorem~\ref{prop:chromatic-certificate}
is used in this paper.  We include the sufficient direction for
motivation and do not give a proof of it here.
A proof can be obtained by averaging the simplex correlations of a
measurable colouring over translations and rotations, and representing
the resulting radial positive-definite function by a nonnegative
spectral measure, as in
\cite[Section~3.1]{bachocdecorteoliveiravallentin2014}.
Integrating \eqref{eq:chromatic-cert-radial} against this measure and
adding the colour-permutation inequalities of
Lemma~\ref{lem:simplex-bqc} forces the left-hand side of
\eqref{eq:chromatic-cert-gap} to be nonnegative, a contradiction.
Appendix~\ref{app:chromatic-certificate} proves the converse by a
completely-positive extraction with fixed colour-block sums, using the
finite-point reduction from Appendix~\ref{app:scalar-certificate}.

The two certificate inputs leave a common task for the rest of the proof.
Their radial expressions are finite linear combinations of spherical
transforms.
After the common rescaling in Section~\ref{sec:main-theorem}, each has
the form
\[
 F(\tau)=c_\infty+\sum_{w\in W}c_w\Omega_d(\|w\|\tau)\ge0
 \qquad(\tau\ge0),
\]
where $W\subseteq\Z^d\setminus\{0\}$ is finite.
The rational point sets furnished by the converses, together with the
rational squared forbidden distances, make this rescaling possible for
$d\ge4$ by Lagrange's four-square theorem.
We seek a lattice trigonometric polynomial obtained by replacing the
spherical terms with averages of characters
$e^{2\pi i v\cdot\theta}$, while allowing an arbitrarily small increase
of the constant term.

The construction must transport each whole configuration $U_q$ together,
preserving its pairwise distances: its matrix constraint concerns all
those points at once.  Relabelling the matrices along these copies and
averaging them preserves their Boolean-quadratic conditions and the
normalization quantities.  The density bound then changes by an
arbitrarily small amount, and the strict chromatic gap survives if the
error is chosen sufficiently small.

\section{Cayley graphs of $\Z^r$}\label{sec:lattice-lp}

In this section we prove analogues of the (positive halves) of Theorem~\ref{prop:scalar-certificate} and Theorem~\ref{prop:chromatic-certificate} for Cayley graphs of $\Z^r$.
By using suitable embedding of such Cayley graphs in $\R^d$, transferring measurable certificates onto such Cayley graphs will provide certificates in the non-measurable setting.

We begin with the analogue of Theorem~\ref{prop:scalar-certificate} (see Theorem~\ref{thm:lattice-bound}).
The recent resolution of Rosenfeld's odd distance problem \cite{davies2023odd} used the following weakening of Theorem~\ref{thm:lattice-bound}. It was also used in its extension to prime and polynomial distances \cite{dmp2024} as well as in studying the chromatic number of unit distance graphs in pseudo-Euclidean spaces \cite{davies2025spacetime}.

\begin{theorem}[Davies~\cite{davies2023odd}]\label{thm:odd-distances-ratio}
Let $C\subseteq\Z^r\setminus\{0\}$ be finite with $C=-C$.  Suppose $z_0\in\R$ and real coefficients
$a_s=a_{-s}$ $(s\in C)$ satisfy
\begin{align*}
 z_0+\sum_{s\in C}a_s&\ge1,\\
 z_0+\sum_{s\in C}a_se^{2\pi i s\cdot\theta}&\ge0
 \qquad(\theta\in\T^r).
\end{align*}
Then
\[
 \overline\alpha(\Cay(\Z^r,C))\le z_0.
\]
\end{theorem}

The above is essentially a weighted Hoffman's ratio bound
\cite{hoffman1970eigenvalues,elzingagregory2010weighted} or Lov{\'a}sz theta bound
\cite{lovasz1979shannon} for Cayley graphs of $\Z^r$.
De Klerk and Pasechnik~\cite{deklerkpasechnik2002} gave an exact copositive formulation of the independence number of a finite graph which is where our required strengthening of Theorem~\ref{thm:odd-distances-ratio} comes from.

For $A\subseteq\Z^r$, define its upper density by
\[
 \overline\delta(A):=\limsup_{N\to\infty}
 \frac{|A\cap[-N,N]^r|}{(2N+1)^r}.
\]
For finite $C$, the parameter $\overline\alpha(\Cay(\Z^r,C))$ equals
the supremum of these upper densities over independent sets.
Translation averaging gives one inequality.  For the other, repeat
maximum independent sets from large boxes, leaving fixed gaps wider
than $\|s\|_\infty$ for every $s\in C$ to prevent edges between copies.
The gaps have negligible density as the boxes grow.

As in \cite{davies2023odd}, we will prove the result for finite Cayley graphs of $\Z^r_n$ and obtain the result for $\Z^r$ by taking a limit.
For this we require the following lemma providing the eigenvalues of weighted adjacency matrices of such Cayley graphs.

\begin{lemma}[Lov{\'a}sz~\cite{lovasz1975spectra}]\label{eigenvalues}
Let $w:\Z_n^r\to\R$ be centrally symmetric and let $B$ be the symmetric matrix
with $B_{u,v}=w(v-u)$.
Then the eigenvalues of $B$ are
\[
 \lambda_z:=\sum_{x\in\Z_n^r}w(x)e^{2\pi i z\cdot x/n}
 \qquad(z\in\Z_n^r).
\]
\end{lemma}

We are now ready to prove the following extension of Theorem~\ref{thm:odd-distances-ratio} (note that taking $\ell=0$ in Theorem~\ref{thm:lattice-bound} recovers
Theorem~\ref{thm:odd-distances-ratio}), which is the aforementioned analogue of Theorem~\ref{prop:scalar-certificate}.

\begin{theorem}\label{thm:lattice-bound}
Let $C\subseteq\Z^r\setminus\{0\}$ be finite with $C=-C$.  For $1\le q\le\ell$, let $U_q\subseteq\Z^r$ be finite and
nonempty and let $Z_q\in\BQC^*(U_q)$.  Suppose $z_0\in\R$ and real
coefficients $a_s=a_{-s}$ $(s\in C)$ satisfy
\begin{align}
 z_0+\sum_{s\in C}a_s-
 \sum_{q=1}^\ell\langle Z_q,J_{U_q}\rangle&\ge1,
 \label{eq:lattice-normalization}\\
 z_0-\sum_{q=1}^\ell\tr (Z_q)+\sum_{s\in C}a_se^{2\pi i s\cdot\theta}
 -\sum_{q=1}^\ell\sum_{\substack{u,v\in U_q\\u\ne v}}Z_q(u,v)
       e^{2\pi i(v-u)\cdot\theta}&\ge0 \qquad(\theta\in\T^r).
 \label{eq:lattice-spectral}
\end{align}
Then
\[
 \overline\alpha(\Cay(\Z^r,C))\le z_0.
\]
\end{theorem}

\begin{proof}
Fix an independent set $A\subseteq\Z^r$ and let
$\delta:=\overline\delta(A)$.  It is enough to show that $\delta\le z_0$.  Choose an integer
\[
 \rho>\max\{\|x\|_\infty:x\in C\cup U_1\cup\cdots\cup U_\ell\}.
\]
For each $N$, put $n=2N+1+2\rho$ and let $A_N=A\cap[-N,N]^r$.
The choice of $\rho$
ensures that $A_N$ is independent in
$\Cay(\Z_n^r,C)$.
Since $\lim_{N\to \infty} n/(2N+1) =1$,
it is therefore
enough to prove that
\[
 \frac{|A_N|}{n^r}\le z_0.
\]

For $x\in\Z_n^r$, let
\[
 K(x)=z_0\mathbf1_{x=0}+\sum_{s\in C}a_s\mathbf1_{x=s}
 -\sum_{q=1}^\ell\sum_{u,v\in U_q}Z_q(u,v)\mathbf1_{x=v-u},
\]
and let $B$ be the matrix indexed by $\Z_n^r$ with
$B_{u,v}=K(v-u)$. Note that $B$ is symmetric since $a_s=a_{-s}$ for each $s\in C$ and each $Z_q$ is symmetric.  By Lemma~\ref{eigenvalues}, the eigenvalues of $B$ are
\[
 z_0+\sum_{s\in C}a_se^{2\pi i s\cdot w/n}
 -\sum_{q=1}^\ell\sum_{u,v\in U_q}Z_q(u,v)
       e^{2\pi i(v-u)\cdot w/n},
 \qquad w\in\Z_n^r.
\]
Separating the diagonal terms, these eigenvalues are nonnegative by \eqref{eq:lattice-spectral}, and hence $B$ is positive semidefinite.

The sum of every row of $B$ is 
\[
 R=z_0+\sum_{s\in C}a_s-
 \sum_{q=1}^\ell\langle Z_q,J_{U_q}\rangle\ge1.
\]
Here \eqref{eq:lattice-normalization} gives $R\ge1$.
Let $h=\mathbf1_{A_N}$.
Writing
$h=(|A_N|/n^r)\mathbf1+w$ with $w$ orthogonal to $\mathbf1$ gives
\begin{align}
    h^TBh=\frac{|A_N|^2}{n^r}R+w^TBw\ge\frac{|A_N|^2}{n^r}. \label{eq:lattice1}
\end{align}

We now evaluate $h^TBh = \sum_{x,y\in A_N} K(y-x)$ directly.  The diagonal term contributes
$z_0|A_N|$, while the generator terms vanish because $A_N$ is independent. 
For each \(q\), set
\[
S_q:=
\sum_{x,y\in A_N}\sum_{u,v\in U_q}
Z_q(u,v)\,\mathbf1_{\{y-x=v-u\}}.
\]
Then,
\[
h^TBh = z_0|A_N| - \sum_{q=1}^\ell S_q.
\]

For
fixed $q$ and $t\in\Z_n^r$, put
\[
 f_t(u)=\mathbf1_{A_N}(t+u)\qquad(u\in U_q).
\]
Since $Z_q\in\BQC^*(U_q)$, $f_t^TZ_qf_t\ge 0$.
Summing over \(t\) and using \(x=t+u\), \(y=t+v\), we get
\[
S_q = \sum_{t\in \Z^r_n} f_t^TZ_qf_t\ge 0.
\]
Hence 
\begin{align}
h^TBh \le z_0|A_N|.     \label{eq:lattice2}
\end{align}

It now follows from \eqref{eq:lattice1} and \eqref{eq:lattice2} that $|A_N|/n^r\le z_0$, as desired.
\end{proof}

We now turn to the chromatic analogue.
Bilu~\cite{bilu2006tales} showed that Hoffman's spectral bound extends to the vector chromatic number.
Fix $k\ge2$ and unit vectors $\xi_1,\ldots,\xi_k\in\R^{k-1}$ forming a regular simplex, so
\begin{equation}\label{eq:simplex-inner-products}
 \xi_i\cdot\xi_j=\frac{k\delta_{ij}-1}{k-1}.
\end{equation}
A proper $k$-colouring may therefore be represented by assigning these vectors to the colours, so that distinct colours have inner product $-1/(k-1)$.
The additional Boolean-quadratic constraints are handled by the following elementary inequality.

\begin{lemma}\label{lem:simplex-bqc}
Let $k\ge2$, let $U$ be finite, let $Z\in\BQC^*(U\times[k])$, and let
$c:U\to[k]$ be arbitrary.  Then
\[
 A_Z+\sum_{u,v\in U}S_Z(u,v)
 \xi_{c(u)}\cdot\xi_{c(v)}\ge0.
\]
\end{lemma}

\begin{proof}
For a permutation $\pi$ of $[k]$, put
$f_\pi(u,i)=\mathbf1_{i=\pi(c(u))}$.  Since $f_\pi\in\{0,1\}^{U\times[k]}$, the Boolean-quadratic condition
gives $f_\pi^TZf_\pi\ge0$.  Averaging over $\pi$ and using
\eqref{eq:simplex-inner-products}, together with
\eqref{eq:AZ}--\eqref{eq:SZ}, gives exactly the displayed inequality.
\end{proof}

We combine this inequality with the vector-colouring viewpoint and the approach of Theorem~\ref{thm:lattice-bound}.

\begin{theorem}\label{thm:lattice-colouring}
Let $C\subseteq\Z^r\setminus\{0\}$ be finite with $C=-C$.  Let $k\ge2$.  For $1\le q\le\ell$, let
$U_q\subseteq\Z^r$ be finite and nonempty and let
$Z_q\in\BQC^*(U_q\times[k])$.  Suppose $z\in\R$ and real coefficients
$a_s=a_{-s}$ $(s\in C)$ satisfy
\begin{align}
 z-\frac1{k-1}\sum_{s\in C}a_s+\sum_{q=1}^\ell A_{Z_q}&<0,
 \label{eq:lattice-colouring-gap}\\
 z-\sum_{q=1}^\ell\tr (S_{Z_q})+\sum_{s\in C}a_se^{2\pi i s\cdot\theta}
 -\sum_{q=1}^\ell\sum_{\substack{u,v\in U_q\\u\ne v}}S_{Z_q}(u,v)
       e^{2\pi i(v-u)\cdot\theta}&\ge0 \qquad(\theta\in\T^r).
 \label{eq:lattice-colouring-spectral}
\end{align}
Then
\[
\chi(\Cay(\Z^r,C)) > k.
\]
\end{theorem}

\begin{proof}
Suppose that $c:\Z^r\to[k]$ is a proper colouring of $\Cay(\Z^r,C)$. Choose an integer
\[
 \rho>\max\{\|x\|_\infty:x\in C\cup U_1\cup\cdots\cup U_\ell\}.
\]
For each $N$, put $n=2N+1+2\rho$ and let
$c_N:\Z_n^r\to[k]$ agree with $c$ on $[-N,N]^r$, extending it
arbitrarily to the remaining vertices of $\Z_n^r$. The choice of $\rho$
ensures that $C$ and the $U_q$ may be identified with their reductions modulo $n$.
For each $s\in C$ all but $O(n^{r-1})$ values of $x\in\Z_n^r$ satisfy
\[
 c_N(x)\ne c_N(x+s).
\]

For $x\in\Z_n^r$, let
\[
 K(x)=z\mathbf1_{x=0}+\sum_{s\in C}a_s\mathbf1_{x=s}
 -\sum_{q=1}^\ell\sum_{u,v\in U_q}
 S_{Z_q}(u,v)\mathbf1_{x=v-u},
\]
and let $B$ be the matrix indexed by $\Z_n^r$ with
$B_{u,v}=K(v-u)$. Note that $B$ is symmetric since $a_s=a_{-s}$
for each $s\in C$ and each $S_{Z_q}$ is symmetric. By
Lemma~\ref{eigenvalues}, the eigenvalues of $B$ are
\[
 z+\sum_{s\in C}a_se^{2\pi i s\cdot w/n}
 -\sum_{q=1}^\ell\sum_{u,v\in U_q}S_{Z_q}(u,v)
 e^{2\pi i(v-u)\cdot w/n},
 \qquad w\in\Z_n^r.
\]
Separating the diagonal terms, these eigenvalues are nonnegative by
\eqref{eq:lattice-colouring-spectral}, and hence $B$ is positive
semidefinite.

The sum of every row of $B$ is
\[
 R=z+\sum_{s\in C}a_s
 -\sum_{q=1}^\ell\sum_{u,v\in U_q}S_{Z_q}(u,v)\ge0.
\]
Taking $\theta=0$ in \eqref{eq:lattice-colouring-spectral} gives $R\ge0$.
Let $X$ be the $n^r\times(k-1)$ matrix whose row indexed by
$x\in\Z_n^r$ is $\xi_{c_N(x)}^T$, and put
\[
 \overline\xi=\frac1{n^r}\sum_{x\in\Z_n^r}\xi_{c_N(x)}.
\]
Writing $X=\mathbf1\overline\xi^T+W$ with $W^T\mathbf1=0$ gives
\begin{align}
 \tr(X^TBX)
 =n^rR\|\overline\xi\|^2+\tr(W^TBW)\ge0.
 \label{eq:lattice-colouring1}
\end{align}

We now evaluate
\[
 \tr(X^TBX)
 =\sum_{x,y\in\Z_n^r}K(y-x)
 \xi_{c_N(x)}\cdot\xi_{c_N(y)}
\]
directly. The diagonal term contributes $zn^r$. For each $s\in C$,
properness of $c$ and \eqref{eq:simplex-inner-products} give
\[
 \xi_{c_N(x)}\cdot\xi_{c_N(x+s)}=-\frac1{k-1}
\]
for all but $O(n^{r-1})$ values of $x$. Since all simplex inner
products have absolute value at most $1$,
\[
 \sum_{x\in\Z_n^r}
 \xi_{c_N(x)}\cdot\xi_{c_N(x+s)}
 =-\frac{n^r}{k-1}+O(n^{r-1}).
\]
For each $q$, set
\[
 S_q:=
 \sum_{x,y\in\Z_n^r}\sum_{u,v\in U_q}
 S_{Z_q}(u,v)\mathbf1_{\{y-x=v-u\}}
 \xi_{c_N(x)}\cdot\xi_{c_N(y)}.
\]
Then
\[
 \tr(X^TBX)
 =zn^r-\frac{n^r}{k-1}\sum_{s\in C}a_s
 -\sum_{q=1}^\ell S_q+O(n^{r-1}).
\]

For fixed $q$ and $t\in\Z_n^r$, put
\[
 c_t(u)=c_N(t+u)\qquad(u\in U_q).
\]
By Lemma~\ref{lem:simplex-bqc},
\[
 A_{Z_q}+\sum_{u,v\in U_q}S_{Z_q}(u,v)
 \xi_{c_t(u)}\cdot\xi_{c_t(v)}\ge0.
\]
Summing over $t$ and using $x=t+u$, $y=t+v$, we get
\[
 n^rA_{Z_q}+S_q
 =
 \sum_{t\in\Z_n^r}
 \left(
 A_{Z_q}+\sum_{u,v\in U_q}S_{Z_q}(u,v)
 \xi_{c_t(u)}\cdot\xi_{c_t(v)}
 \right)\ge0.
\]
Hence
\begin{align}
 \tr(X^TBX)
 \le n^r\left(
 z-\frac1{k-1}\sum_{s\in C}a_s
 +\sum_{q=1}^\ell A_{Z_q}
 \right)+O(n^{r-1}).
 \label{eq:lattice-colouring2}
\end{align}

It now follows from \eqref{eq:lattice-colouring1} and
\eqref{eq:lattice-colouring2} that
\[
 0\le n^r\left(
 z-\frac1{k-1}\sum_{s\in C}a_s+\sum_{q=1}^\ell A_{Z_q}
 \right)+O(n^{r-1}),
\]
contradicting \eqref{eq:lattice-colouring-gap} for all sufficiently
large $N$.
\end{proof}

\section{Exponential sum estimates on 3-spheres}\label{sec:lattice-sphere-estimates}

In this section we gather the 4-dimensional exponential sum estimates we require for points of $\Z^4$ on a $3$-sphere.
We require decay when bounded away from bounded denominator rational points and approximation by $\Omega_4$ near $0$.

Whenever $p$ appears in this section, it denotes an odd prime, and limits as
$p\to\infty$ are taken through the odd primes.  Put
\[
 X_p:=\{x\in\Z^4:\|x\|^2=p^2\}.
\]
Jacobi's four-square formula~\cite{grosswald1985sumsquares} gives $|X_p|=8(1+p+p^2)$.  Define
\[
 M_p(\theta):=\frac1{|X_p|}\sum_{x\in X_p}e^{2\pi i x\cdot\theta},
 \qquad \theta\in\T^4.
\]

For integers $r,q\ge1$, write
\[
 \T^r[q]:=q^{-1}\Z^r/\Z^r\subseteq\T^r
\]
for the $q$-torsion subgroup.  For $r,Q\ge1$ let
\[
 \mathfrak R_Q^r:=\bigcup_{1\le q\le Q}\T^r[q]\subseteq\T^r.
\]
Thus $\mathfrak R_Q^r$ is the set of torus points of denominator at most $Q$.

For $q\ge1$, let $v_2(q)$ be the exponent of $2$ in $q$ and define
\[
 \rho_\lambda(q):=2^{v_2(q)}
 \gcd\!\left(\lambda,\frac{q}{2^{v_2(q)}}\right).
\]

To obtain decay when bounded away from bounded denominator rational points, we simply apply an exponential sum estimate of Magyar~\cite{magyar2007distribution} for points of $\Z^4$ on a 3-sphere of radius $\sqrt{\lambda}$.
Having such an estimate is crucial to our arguments.
We remark that the corresponding theorem for $d\ge 5$ was proven by Magyar, Stein and Wainger~\cite{magyarsteinwainger2002spherical}.
The following four-dimensional estimate is the specialization of Lemma~1 of Magyar~\cite{magyar2007distribution}; we take $\epsilon = 1/8$, input the given further inequalities in the lemma, and suppress the bounded cutoff factor, whose absolute value is at most $1$.

\begin{theorem}[Magyar~\cite{magyar2007distribution}]\label{lem:magyar}
There is a constant $C>0$ such that, for every integer $\lambda\ge1$ and every
$\theta\in\T^4$,
\[
 \left|\sum_{\substack{x\in\Z^4\\ \|x\|^2=\lambda}}e^{2\pi i x\cdot\theta}\right|
 \le C\left(\lambda\sum_{q\le\sqrt\lambda}q^{-11/8}\rho_\lambda(q)^{1/2}
 \bigl(1+\sqrt\lambda\,\operatorname{dist}(\theta,\T^4[q])\bigr)^{-3/2}+\lambda^{7/8}\right).
\]
\end{theorem}

For $1\le Q\le p$ define
\[
 A_p(Q):=\sup\left\{
 |M_p(\theta)|:\theta\in\T^4,
 \operatorname{dist}(\theta,\mathfrak R_Q^4)\ge\frac Qp
 \right\}.
\]

We now deduce from Theorem~\ref{lem:magyar} uniform decay away from rational points with bounded denominator.

\begin{lemma}\label{lem:minor}
\[
 \lim_{Q\to\infty}\limsup_{p\to\infty}A_p(Q)=0.
\]
\end{lemma}

\begin{proof}
Apply Theorem~\ref{lem:magyar} with $\lambda=p^2$.  Since
$|X_p|=8(1+p+p^2)\ge8p^2$, division by $|X_p|$ gives, after changing the
absolute constant,
\begin{equation}\label{eq:four-dimensional-magyar}
 |M_p(\theta)|\le C\left(
 \sum_{q\le p}q^{-11/8}\rho_{p^2}(q)^{1/2}
 \bigl(1+p\operatorname{dist}(\theta,\T^4[q])\bigr)^{-3/2}
 +p^{-1/4}\right).
\end{equation}
Writing $q=2^j m$ with $m$ odd gives
\[
 \sum_{q\ge1}q^{-11/8}2^{v_2(q)/2}
 =\sum_{j\ge0}2^{-7j/8}\sum_{\substack{m\ge1\\m\text{ odd}}}m^{-11/8}<\infty.
\]
For the tail, first suppose $2^j\le Q$.  Since $q=2^jm>Q$ implies
$m>Q/2^j$, we have
\[
 2^{-7j/8}\sum_{\substack{m>Q/2^j\\m\text{ odd}}}m^{-11/8}
 \le C Q^{-3/8}2^{-j/2}.
\]
Summing over $2^j\le Q$ gives $O(Q^{-3/8})$.  If $2^j>Q$, then
\[
 \sum_{2^j>Q}2^{-7j/8}
 \sum_{\substack{m\ge1\\m\text{ odd}}}m^{-11/8}
 =O(Q^{-7/8}).
\]
Consequently
\[
 \sum_{q>Q}q^{-11/8}2^{v_2(q)/2}\le C Q^{-3/8}.
\]

Fix $Q$ and take $p>Q$.  Suppose
$\operatorname{dist}(\theta,\mathfrak R_Q^4)\ge Q/p$.  If $q\le Q$, then
$\T^4[q]\subseteq\mathfrak R_Q^4$, so
$\operatorname{dist}(\theta,\T^4[q])\ge Q/p$; also $p\nmid q$, and hence
$\rho_{p^2}(q)^{1/2}=2^{v_2(q)/2}$.  The total contribution of these moduli
in \eqref{eq:four-dimensional-magyar} is therefore at most $C Q^{-3/2}$.
The terms with $q>Q$ and $p\nmid q$ contribute at most $C Q^{-3/8}$ by the
tail estimate above.  Finally, if $p\mid q$ and $q\le p$, then $q=p$, and
its contribution is
\[
 p^{-11/8}\rho_{p^2}(p)^{1/2}=p^{-7/8}.
\]
Hence
\[
 A_p(Q)\le C\bigl(Q^{-3/2}+Q^{-3/8}+p^{-7/8}+p^{-1/4}\bigr).
\]
Taking first $p\to\infty$ and then $Q\to\infty$ proves the lemma.
\end{proof}

Next we show that near 0 the exponential sum can be approximated by $\Omega_4$. This is just an application of Malyshev's~\cite{malyshev1957sphere} equidistribution theorem for points of $\Z^4$ on a $3$-sphere.

\begin{lemma}\label{lem:major-zero}
For every $C>0$,
\[
 \lim_{p\to\infty}\sup_{\|u\|\le C}
 \left|M_p(u/p)-\Omega_4(\|u\|)\right|=0.
\]
\end{lemma}

\begin{proof}
For $x\in X_p$, put $\omega=x/p\in S^3$.  Then the corresponding summand in
$M_p(u/p)$ is $e^{2\pi i u\cdot\omega}$.

Malyshev~\cite{malyshev1957sphere} proved that, as $n\to\infty$ through
integers with $4\nmid n$, the normalized sets
\[
 n^{-1/2}\{x\in\Z^4:\|x\|^2=n\}\subseteq S^3
\]
become equidistributed with respect to spherical measure.  So, for odd primes $p$,
we may take $n=p^2$.
Thus $p^{-1}X_p$ becomes equidistributed on $S^3$ as
$p\to\infty$.  Consequently, for every fixed $u\in\R^4$,
\[
 M_p(u/p)\longrightarrow
 \int_{S^3}e^{2\pi i u\cdot\omega}\,d\sigma_4(\omega)
 =\Omega_4(\|u\|)
 \qquad\text{as }p\to\infty.
\]
For $u,v\in\R^4$ and $\omega\in S^3$,
\[
 |e^{2\pi i u\cdot\omega}-e^{2\pi i v\cdot\omega}|
 \le2\pi\|u-v\|.
\]
Consequently both $u\mapsto M_p(u/p)$ and
$u\mapsto\Omega_4(\|u\|)$ are $2\pi$-Lipschitz, uniformly in $p$.

Fix $\varepsilon>0$ and choose a finite $\varepsilon/(6\pi)$-net
$\{u_1,\ldots,u_m\}$ of the compact set $\{u:\|u\|\le C\}$. By pointwise convergence, for all
sufficiently large $p$,
\[
\left|M_p(u_j/p)-\Omega_4(\|u_j\|)\right|<\varepsilon/3
\qquad (1\le j\le m).
\]
Now, for any $\|u\|\le C$, choose $j$ with
$\|u-u_j\|<\varepsilon/(6\pi)$. Using the $2\pi$-Lipschitz bounds above,
\[
\begin{aligned}
\left|M_p(u/p)-\Omega_4(\|u\|)\right|
&\le \left|M_p(u/p)-M_p(u_j/p)\right| \\
&\quad +\left|M_p(u_j/p)-\Omega_4(\|u_j\|)\right| \\
&\quad +\left|\Omega_4(\|u_j\|)-\Omega_4(\|u\|)\right| \\
&<\varepsilon.
\end{aligned}
\]
Hence the convergence is uniform on $\{u:\|u\|\le C\}$, as required.
\end{proof}

Next we aim to combine Lemma~\ref{lem:minor} and Lemma~\ref{lem:major-zero} into a form suitable for our later applications.
In our application we will want to estimate over more than just individual points $a\in Z^r$ corresponding to individual points of $\R^d$, we further want to keep track of entire finite sets $W\subset \Z^r$ corresponding to isometric copies of a set in $\R^d$.
This actually makes $d=4$ an ideal dimension since we can identify $\R^4$ with the quaternion algebra $\mathbb H$ in the basis
$1,\mathbf i,\mathbf j,\mathbf k$.
The advantage of doing this is that quaternions form a normed division
algebra and the integer-coordinate quaternions
\[
 \mathbb H_{\mathbb Z}:=
 \Z+\Z\mathbf i+\Z\mathbf j+\Z\mathbf k
\]
form a ring.
Thus, for $y\in\R^4$, we can let $\mathsf L_y$ and
$\mathsf R_y$ denote the matrices for left and right multiplication by $y$ respectively.
Note that if $y$ is integral then so is $\mathsf L_y$ and
$\mathsf R_y$.
Also note that
\[
 \mathsf L_y^T\mathsf L_y=\mathsf R_y^T\mathsf R_y=\|y\|^2I_4,
 \qquad
 z\cdot\mathsf L_xa=x\cdot\mathsf R_a^Tz.
\]
These matrices will later allow us to transport entire congruent copies of $W$ at once while remaining within our lattice.
For $a\in\Z^4$ and
$\theta\in\T^4$, define
\[
 \mathcal K_{p,a}(\theta)
 :=\frac1{|X_p|}\sum_{x\in X_p}e^{2\pi i\theta\cdot\mathsf L_xa}
 =M_p(\mathsf R_a^T\theta).
\]

\begin{lemma}\label{lem:four-dimensional-kernel}
Let $W\subseteq\Z^4\setminus\{0\}$ be finite and let $0<\delta$.
There are $C_0>0$ and a positive integer $Q_0$ such that, for every sufficiently large odd
prime $p$, every $a\in W$, and every $\theta\in\T^4$ satisfying
\[
 \operatorname{dist}\!\left(\theta,
 \mathfrak R_{Q_0}^4\setminus\{0\}\right)\ge\frac{C_0}{p},
\]
we have
\[
 \left|\mathcal K_{p,a}(\theta)
 -\Omega_4\!\left(p\|a\|\operatorname{dist}(\theta,0)\right)\right|<\delta.
\]
\end{lemma}

\begin{proof}
The lemma trivially holds if $W=\varnothing$, so we may assume otherwise.
By Lemma~\ref{lem:minor}, choose an integer $Q$ such that
\[
 \limsup_{p\to\infty}A_p(Q)<\delta/2.
\]
So, $A_p(Q)<\delta/2$ for all sufficiently large $p$.
Using $\Omega_4(t)\to0$ as $t\to\infty$, choose $C_0\ge Q$ such that
\[
 |\Omega_4(t)|<\delta/2\qquad(t\ge C_0).
\]
Taking preimages under $\mathsf R_a^T$ can enlarge denominators: since $\mathsf R_a$ is integral and $\mathsf R_a\mathsf R_a^T=\|a\|^2I_4$, every preimage of $\T^4[s]$ under $\mathsf R_a^T$ lies in $\T^4[s\|a\|^2]$.
We therefore put
\[
 Q_0:=Q\max_{a\in W}\|a\|^2.
\]

Fix $a\in W$.
Suppose first that $\theta$ satisfies the hypothesis of the lemma with
$\operatorname{dist}(\theta,0)\le C_0/p$.
Choose the representative \(\theta_0\in\mathbb R^4\) of \(\theta\) with
$
\|\theta_0\|=\operatorname{dist}(\theta,0),
$
and put
$u=p\mathsf R_a^T\theta_0$.  Since $\mathsf R_a$ is a similitude of ratio
$\|a\|$,
\[
 \|u\|=p\|a\|\,\|\theta_0\|
 \le C_0\max_{b\in W}\|b\|.
\]
Thus, $\mathcal K_{p,a}(\theta) = M_p(u/p)$ and so Lemma~\ref{lem:major-zero} gives, uniformly in $a$ and $\theta$ in this case,
\[
 \left|\mathcal K_{p,a}(\theta)
 -\Omega_4\!\left(p\|a\|\operatorname{dist}(\theta,0)\right)\right|<\delta,
\]
for all sufficiently large odd primes $p$, as required.

We may therefore assume that
$\operatorname{dist}(\theta,0)>C_0/p$.  Together with the hypothesis of
the lemma, this gives
\[
 \operatorname{dist}(\theta,\mathfrak R_{Q_0}^4)\ge C_0/p.
\]
For $1\le s\le Q$, we have $s\|a\|^2\le Q_0$.  Using
$\operatorname{dist}(\theta,\T^4[ms])
=m^{-1}\operatorname{dist}(m\theta,\T^4[s])$, together with
$\mathsf R_a\T^4[s]\subseteq\T^4[s]$ and the fact that $\mathsf R_a$
is a similitude of ratio $\|a\|$, we obtain
\[
\begin{aligned}
 \|a\|\operatorname{dist}(\theta,\mathfrak R_{Q_0}^4)
 &\le
 \|a\|\operatorname{dist}\!\left(\theta,\T^4[s\|a\|^2]\right)\\
 &=
 \frac1{\|a\|}
 \operatorname{dist}\!\left(
 \mathsf R_a\mathsf R_a^T\theta,\T^4[s]\right)\\
 &\le
 \operatorname{dist}\!\left(
 \mathsf R_a^T\theta,\T^4[s]\right).
\end{aligned}
\]
Taking the minimum over $1\le s\le Q$ gives
\[
 \operatorname{dist}(\mathsf R_a^T\theta,\mathfrak R_Q^4)
 \ge \|a\|\operatorname{dist}(\theta,\mathfrak R_{Q_0}^4)
 \ge \frac{C_0\|a\|}{p}
 \ge\frac Qp.
\]
Thus, by the definition of $A_p(Q)$,
\[
 |\mathcal K_{p,a}(\theta)|
 =|M_p(\mathsf R_a^T\theta)|
 \le A_p(Q)<\delta/2.
\]
Also $p\|a\|\operatorname{dist}(\theta,0)>C_0$, since $\|a\|\ge1$, and
hence the corresponding $\Omega_4$ term has absolute value less than
$\delta/2$.  The triangle inequality then completes the proof.
\end{proof}

\section{Four-dimensional transference}\label{sec:four-dimensional-arithmetic}

In this section we prove Theorem~\ref{thm:radial-transfer-four}, which is our transference theorem in the four dimensional case, and which is sufficient to prove the four dimensional case of Theorem~\ref{thm:main}.
However, we remark that the lemma we prove in this section, Lemma~\ref{lem:dmp-planar}, is also required for our proof of the transference theorem in dimensions $d\ge 5$ (see Theorem~\ref{thm:radial-transfer}).

As in our solution to Rosenfeld's odd distance problem \cite{davies2023odd}, we shall handle the exponential sum estimates near rationals with bounded denominator by designing the embedded distance graph so that the exponential sum can be split into parts where only a small proportion can fall into the unwieldy small denominator major arc case.

For showing that only a small proportion can fall into this bad case we require the following lemma on particular rotations in $\R^2$. It is close to similar lemmas of the author, McCarty and Pilipczuk~\cite{dmp2024} that were used to extend the odd distance theorem to prime and polynomial distances.
However, we require some additional properties and also include the full proof since it is elementary and fundamental to our arguments.
We remark that the ``Moreover'' part of this lemma is not used until Section~\ref{sec:higher-dimensional-arithmetic} in the proof of Lemma~\ref{lem:primary-frames-high}.

\begin{lemma}\label{lem:dmp-planar}
Let $Q,K$ be positive integers.  There exist an irrational number $\vartheta>0$ with
$\vartheta^2\in\Z$, a positive integer $\kappa$, a number $\gamma>0$, and integer
pairs $(r_j,s_j)$, $0\le j<K$, such that
\[
 C_j:=\frac1\kappa
 \begin{pmatrix}r_j&-\vartheta s_j\\ \vartheta s_j&r_j\end{pmatrix}\in SO(2)
 \qquad(0\le j<K),
\]
and, writing $q=Q!$, for every $\omega=(\omega_1,\omega_2)\in\T^2$
at most two indices $j$ satisfy
\[
 \operatorname{dist}\!\left(r_j\omega_1+s_j\omega_2,
 \T^1[q]\setminus\{0\}\right)<\gamma.
\]
Moreover, the parameters may be chosen so that,
\[
 \gcd(\vartheta^2,q)=1,\qquad \frac{\kappa+r_j}{\vartheta^2}\in\Z
 \qquad(0\le j<K).
\]
\end{lemma}

\begin{proof}
For $0\le j<K$, let
\[
 m_j:=q^{K+j}+q^{K-j}+1
      =q^{K-j}(q^{2j}+1)+1.
\]
Set
\[
X:=\prod_{h=0}^{K-1}m_h, \qquad \vartheta:=\sqrt{4q^{2K}-1},\qquad \kappa:=2q^KX.
\]
Note that $\vartheta^2\in \Z$, $\gcd(\vartheta^2,q)=1$, and $\vartheta \not\in \Q$ since 
\[
 (2q^K-1)^2<\vartheta^2<(2q^K)^2.
\]
For each $0\le j < K$, set
\[
 s_j:=\frac{Xq^j}{m_j}(1+2q^{K-j}),\qquad
 r_j:=\kappa-\frac{Xq^j}{m_j}(1+2q^{K-j})^2.
\]
Since $m_j\mid X$, these are integers.

To check that the matrices $C_j$ are rotations, it is enough to check that $r_j^2 + \vartheta^2 s_j^2 = \kappa^2$.
Note that
\[
 \vartheta^2+(1+2q^{K-j})^2
 =4q^{K-j}m_j,
\]
and hence
\[
 2\kappa
 =\frac{Xq^j}{m_j}
   \bigl(\vartheta^2+(1+2q^{K-j})^2\bigr).
\]
Therefore
\[
 \kappa-r_j=\frac{Xq^j}{m_j}(1+2q^{K-j})^2,
 \qquad
 \kappa+r_j=\frac{Xq^j}{m_j}\vartheta^2.
\]
Since
$m_j\mid X$, the second identity gives $(\kappa+r_j)/\vartheta^2 \in\Z$.  Also
\[
 \kappa^2-r_j^2
 =(\kappa-r_j)(\kappa+r_j)
 =\vartheta^2\left(\frac{Xq^j}{m_j}(1+2q^{K-j})\right)^2
 =\vartheta^2s_j^2.
\]
Thus $r_j^2+\vartheta^2s_j^2=\kappa^2$, and hence $C_j\in SO(2)$.

It remains to prove the separation property.
Since any three of the vectors $(r_j,s_j)$ are linearly dependent over $\Q$,
the following number is well-defined:
\[
 M:=\max_{0\le i<j<\ell<K}\;
 \min\Bigl\{|a|+|b|+|c|:\ a,b,c\in\Z,\ \gcd(a,b,c)=1,
 \ a(r_i,s_i)+b(r_j,s_j)+c(r_\ell,s_\ell)=0\Bigr\}.
\]
Set
\[
 \gamma:=\frac{1}{q(M+1)}.
\]

Suppose for sake of contradiction, that for some $\omega=(\omega_1,\omega_2)\in\T^2$
there are indices $i<j<\ell$ such that for every $h\in 
\{i,j,\ell\}$ we have that 
\[
 \operatorname{dist}\!\left(r_h\omega_1+s_h\omega_2,
 \T^1[q]\setminus\{0\}\right)<\gamma.
\]
By the definition of $M$, there is a primitive relation
\begin{equation}
 a(r_i,s_i)+b(r_j,s_j)+c(r_\ell,s_\ell)=0
 \label{eq:dmp-relation}
\end{equation}
with $|a|+|b|+|c|\le M$.

Since every $m_u\equiv1\pmod q$, we have for every $h\in 
\{i,j,\ell\}$ that 
$X/m_h=\prod_{u\ne h}m_u\equiv1\pmod q$, and hence
\begin{equation}
 s_h=\frac{X}{m_h}q^h(1+2q^{K-h})
 \equiv q^h\pmod{q^{h+1}}.
 \label{eq:dmp-s-congruence}
\end{equation}
The second coordinate of \eqref{eq:dmp-relation}, together with
\eqref{eq:dmp-s-congruence}, gives
\[
 0=as_i+bs_j+cs_\ell\equiv aq^i\pmod{q^{i+1}},
\]
so $q\mid a$.  Also
\begin{equation}
 \frac{r_h+s_h}{2}
 =\frac{X}{m_h}q^{2K-h}(q^{2h}-1)
 \equiv-q^{2K-h}\pmod{q^{2K-h+1}}
 \qquad(h>0).
 \label{eq:dmp-rs-congruence}
\end{equation}
For $h=0$, the same formula gives $(r_0+s_0)/2=0$.  Adding the two coordinates
of \eqref{eq:dmp-relation}, dividing by $2$, and using
\eqref{eq:dmp-rs-congruence} gives
\[
 0=a\frac{r_i+s_i}{2}+b\frac{r_j+s_j}{2}
   +c\frac{r_\ell+s_\ell}{2}
 \equiv-cq^{2K-\ell}\pmod{q^{2K-\ell+1}},
\]
so $q\mid c$.  Since $\gcd(a,b,c)=1$, it follows that
\[
 \gcd(b,q)=1.
\]

Now for each $h\in\{i,j,\ell\}$, choose $t_h\in\T^1[q]\setminus\{0\}$ such that
\[
 \operatorname{dist}(r_h\omega_1+s_h\omega_2,t_h)<\gamma.
\]
Since $q=Q!$, we have $qt_h\equiv0\pmod1$ for each $h$.  Hence $q\mid a$
and $q\mid c$ 
gives
\[
 at_i\equiv ct_\ell\equiv0\pmod1.
\]
Since multiplication by $b$ is an automorphism of the finite group
$\T^1[q]$, the point $bt_j$ is nonzero.  Hence
\[
 \operatorname{dist}(bt_j,0)\ge\frac1q.
\]
Using \eqref{eq:dmp-relation} modulo $1$,
\[
\begin{aligned}
 bt_j
 &\equiv at_i+bt_j+ct_\ell\\
 &\equiv a\bigl(t_i-r_i\omega_1-s_i\omega_2\bigr)
   +b\bigl(t_j-r_j\omega_1-s_j\omega_2\bigr)
   +c\bigl(t_\ell-r_\ell\omega_1-s_\ell\omega_2\bigr)
   \pmod1.
\end{aligned}
\]
Therefore, by our choice of $t_i,t_j,t_\ell, M, \gamma$, we have that
\[
 \frac1q
 \le \operatorname{dist}(bt_j,0)
 <(|a|+|b|+|c|)\gamma
 \le M\gamma<\frac1q,
\]
a contradiction.
Hence at most two indices satisfy the inequality in the statement, as desired.
\end{proof}

We are now ready to prove our four dimensional transference theorem.
A finite family $\mathcal G$ of homomorphisms is treated as an indexed
family: sums over $\mathcal G$ and its size $|\mathcal G|$ count
repetitions with multiplicity.  It is symmetric if $\Gamma$ and
$-\Gamma$ occur with the same multiplicity for every $\Gamma$.

\begin{theorem}\label{thm:radial-transfer-four}
Let $W\subseteq\Z^4\setminus\{0\}$ be finite, let $c_\infty\in\R$, and let
$c_w\in\R$ for $w\in W$.  Suppose
\begin{equation}\label{eq:four-transfer-radial}
 F(\tau):=c_\infty+\sum_{w\in W}c_w\Omega_4(\|w\|\tau)\ge0
 \qquad(\tau\ge0).
\end{equation}
Then for every $\varepsilon>0$ there exist a positive nonsquare integer $n$,
an injective homomorphism
\[
 E:\Z^8\longrightarrow\Q(\sqrt n)^4\subseteq\R^4,
\]
and a finite symmetric family $\mathcal G$ of injective homomorphisms
$\Gamma:\Z^4\to\Z^8$, 
such that
\[
 \|E\Gamma w\|=\|w\|
 \qquad(w\in\Z^4,\ \Gamma\in\mathcal G),
\]
and
\begin{equation}\label{eq:four-transfer-positive}
 F_\varepsilon(\theta):=c_\infty+\varepsilon+
 \frac1{|\mathcal G|}\sum_{\Gamma\in\mathcal G}\sum_{w\in W}c_w
       e^{2\pi i\Gamma w\cdot\theta}>0
 \qquad(\theta\in\T^8).
\end{equation}
\end{theorem}

\begin{proof}
Since $\Omega_4(t)\to0$ as $t\to\infty$, letting $\tau\to\infty$ in
\eqref{eq:four-transfer-radial} gives $c_\infty\ge0$.  Put $\Lambda:=\sum_{w\in W}|c_w|$ and choose $0<\delta<1$ so
that $2\delta \Lambda<\varepsilon$.
Apply Lemma~\ref{lem:four-dimensional-kernel} to $W$ and $\delta$, obtaining
$C_0,Q_0$.

Choose a positive integer $K$ with $8/K<\delta$ and apply
Lemma~\ref{lem:dmp-planar} with $Q=Q_0$, obtaining
$\vartheta,\kappa,\gamma$, and for $0\le j<K$, the integer pairs $(r_j,s_j)$, and the rotations
$C_j$.
Put $n:=\vartheta^2$.  Since $\vartheta>0$ is
irrational and $\vartheta^2\in\Z$, the integer $n$ is positive and nonsquare.
Define
\[
 \Psi:=\frac{1}{\kappa}
\begin{pmatrix}
1 & 0 & 0 & 0 & \vartheta & 0 & 0 & 0 \\
0 & 1 & 0 & 0 & 0 & \vartheta & 0 & 0 \\
0 & 0 & 1 & 0 & 0 & 0 & \vartheta & 0 \\
0 & 0 & 0 & 1 & 0 & 0 & 0 & \vartheta
\end{pmatrix},
 \qquad
 B_j:=
 \begin{pmatrix}
 r_j&0&0&0\\
 0&r_j&0&0\\
 0&0&r_j&0\\
 0&0&0&r_j\\
 0&-s_j&0&0\\
 s_j&0&0&0\\
 0&0&0&-s_j\\
 0&0&s_j&0
 \end{pmatrix}
 \qquad(0\le j<K).
\]
Note that $\Psi$ is injective on $\Z^8$, and
\[
 \Psi B_j=C_j\oplus C_j\in SO(4).
\]

Choose a sufficiently large odd prime $p$ such that $C_0/p<\gamma$ and the
conclusion of Lemma~\ref{lem:four-dimensional-kernel} holds.  Let
$\mathcal G$ be the indexed family
\[
 \Gamma_{j,x}:=B_j\mathsf L_x
 \qquad(0\le j<K,\ x\in X_p).
\]
Since $X_p=-X_p$ and $\mathsf L_{-x}=-\mathsf L_x$, the indexed family
$\mathcal G$ is symmetric, and therefore $F_\varepsilon$ is real valued.
Put $E:=p^{-1}\Psi$.  Then
$E(\Z^8)\subseteq\Q(\vartheta)^4=\Q(\sqrt n)^4\subseteq\R^4$, and $E$ is
injective.
Since $\Psi B_j\in SO(4)$ and $\mathsf L_x$ is a similitude of ratio $p$,
every $\Gamma_{j,x}$ is injective and
\[
 \|E\Gamma_{j,x}w\|=\|w\|
 \qquad(w\in\Z^4).
\]

It remains to prove \eqref{eq:four-transfer-positive}.  
Fix
$\theta=(x_1,x_2,x_3,x_4,y_1,y_2,y_3,y_4)\in\T^8$, then,
\[
 B_j^T\theta=
 \begin{pmatrix}
 r_jx_1+s_jy_2\\
 r_jx_2-s_jy_1\\
 r_jx_3+s_jy_4\\
 r_jx_4-s_jy_3
 \end{pmatrix}.
\]
Since $\mathfrak R_{Q_0}^1\subseteq\T^1[Q_0!]$,
Lemma~\ref{lem:dmp-planar} shows that, for each $1\le i \le 4$, the $i$-th coordinate of $B_j^T\theta$ lies within $\gamma$ of a nonzero point of
$\mathfrak R_{Q_0}^1$ for at most two indices $j$.  A nonzero point of $\mathfrak R_{Q_0}^4$ has a nonzero coordinate, so
at most eight indices $j$ satisfy
$ \operatorname{dist}(B_j^T\theta,
 \mathfrak R_{Q_0}^4\setminus\{0\})<\gamma$.
In particular, all but a proportion less than $\delta$ of the indices $j$ satisfy
\[
 \operatorname{dist}(B_j^T\theta,
 \mathfrak R_{Q_0}^4\setminus\{0\})\ge\gamma>C_0/p.
\]

Define
\[
 L_j:=c_\infty+\sum_{w\in W}c_w\mathcal K_{p,w}(B_j^T\theta).
\]
For every $j$ with $\operatorname{dist}(B_j^T\theta,
 \mathfrak R_{Q_0}^4\setminus\{0\})>C_0/p$, Lemma~\ref{lem:four-dimensional-kernel} gives
$|L_j-F(p\operatorname{dist}(B_j^T\theta,0))|\le\delta \Lambda$, and so
\eqref{eq:four-transfer-radial} therefore gives $L_j\ge-\delta \Lambda$.
For the remaining indices, $|\mathcal K_{p,w}(B_j^T\theta)|\le1$ and
$c_\infty\ge0$, so $L_j\ge-\Lambda$. 
Since the proportion of such indices is at most $\delta$, we therefore have
\[
\frac{1}{K}\sum_{j=0}^{K-1}L_j \ge -(1-\delta)\delta\Lambda - \delta\Lambda  \ge  -2\delta \Lambda.
\]
For $0\le j<K$, $x\in X_p$, and $w\in W$, we have
$B_j^T\theta \cdot\mathsf L_xw = \theta \cdot B_j\mathsf L_xw  =\theta\cdot\Gamma_{j,x}w$, and
$|\mathcal G|=K|X_p|$.
Hence
\[
F_\varepsilon(\theta)
=
 c_\infty+ \varepsilon+
 \frac1{|\mathcal G|}\sum_{\Gamma\in\mathcal G}\sum_{w\in W}c_w
 e^{2\pi i\Gamma w\cdot\theta}
 =
 \varepsilon
 +
 \frac1K\sum_{j=0}^{K-1}L_j 
 \ge
 \varepsilon
 -2\delta\Lambda > 0,
\]
as desired.
\end{proof}

\section{Higher-dimensional transference}\label{sec:higher-dimensional-arithmetic}

In this section we prove Theorem~\ref{thm:radial-transfer}, which is our transference theorem in dimensions $d
\ge 5$.
In Theorem~\ref{thm:radial-transfer-four} we obtained an injective homomorphism
\[
 E:\Z^8\longrightarrow\Q(\sqrt n)^4\subseteq\R^4,
\]
for some positive nonsquare $n$.
In Theorem~\ref{thm:radial-transfer}, we shall further obtain a much simpler injective homomorphism
\[
 E= m^{-1} I_d:\Z^d\longrightarrow m^{-1} \Z^d \subseteq  \Q^d\subseteq\R^d,
\]
for some positive integer $m$. In particular, this will give us our rational strengthening in Theorem \ref{thm:main} of $\chi(\Q^d)=\chi(\R^d)$ for $d\ge 5$.
For this we require the following more refined consequence of Lemma~\ref{lem:dmp-planar}.
The condition $d\ge5$ ensures that the quaternionic four-dimensional block has a nonzero orthogonal complement.
In the following construction we use Lagrange's four-square theorem in that block and combine each of its basis directions with a direction in the complement to lift the planar rotations of Lemma~\ref{lem:dmp-planar}.

\begin{lemma}\label{lem:primary-frames-high}
Let $d\ge5$, and let $Q,K$ be positive integers.
Then there exist an integer
$\lambda\ge1$, a number $\gamma>0$, a finite index set $\Sigma$, and
matrices $B_\sigma\in M_d(\Z)$ $(\sigma\in\Sigma)$ such that
\[
 B_\sigma^TB_\sigma=\lambda^2I_d
 \qquad(\sigma\in\Sigma),
\]
and, for every $\theta\in\T^d$,
\begin{equation}\label{eq:6.1}
 \#\bigl\{\sigma\in\Sigma:
 \operatorname{dist}(B_\sigma^T\theta,
 \T^d[Q!]\setminus\{0\})<\gamma\bigr\}
 \le \frac{2d}{K}|\Sigma|.
\end{equation}
\end{lemma}

\begin{proof}
Put $q=Q!$ and apply Lemma~\ref{lem:dmp-planar} with $Q,K$, obtaining
$\vartheta,\kappa,\gamma'$, and for $0\le j<K$, pairs $(r_j,s_j)$ and matrices $C_j$ as in the lemma.
Set $\gamma:=\gamma'/\vartheta$.
Thus
\[
 r_j^2+\vartheta^2s_j^2=\kappa^2,
 \qquad \gcd(\vartheta^2,q)=1,
 \qquad \frac{\kappa+r_j}{\vartheta^2}\in\Z.
\]
The two additional arithmetic properties in Lemma~\ref{lem:dmp-planar} play different roles here.
The divisibility condition $(\kappa+r_j)/\vartheta^2\in\Z$ makes the matrices $C_{a,b}(j)$ defined below integral, while $\gcd(\vartheta^2,q)=1$ ensures that the quaternionic change of basis is invertible modulo $q$, so its coordinates detect every nonzero point of $\T^d[q]$.
By Lagrange's four-square theorem choose an integral quaternion
$w\in\Z^4$ with $\|w\|^2=\vartheta^2$.  Since
\[
 \mathsf R_w^T\mathsf R_w=\vartheta^2I_4,
\]
the columns $w_1,\ldots,w_4$ of $\mathsf R_w$ are pairwise orthogonal and
satisfy $\|w_a\|^2=\vartheta^2$.
Put $m=d-4$, and write $e_1,\ldots,e_m$ for the standard basis of $\R^m$,
viewing both families in $\R^d=\R^4\oplus\R^m$.
So,
the vectors \(\vartheta^{-1}w_1,\ldots,\vartheta^{-1}w_4,e_1,\ldots,e_m\) form an orthonormal basis of \(\mathbb R^d\).

For \(a\in[4]\), \(b\in[m]\), and \(0\le j<K\), define the matrix \(C_{a,b}(j)\)
to act as \(\kappa C_j^T\) on the ordered orthonormal basis
\((\vartheta^{-1}w_a,e_b)\) of \(\operatorname{span}\{\vartheta^{-1}w_a,e_b\}\), and as
\(-\kappa I\) on its orthogonal complement. Since \(C_j\in SO(2)\),
\[
C_{a,b}(j)^TC_{a,b}(j)=\kappa^2I_d.
\]

Since \(w_a\) and \(e_b\) are orthogonal and \(\|w_a\|^2=\vartheta^2\), for \(v\in\mathbb R^d\),
\[
\begin{aligned}
C_{a,b}(j)v
&=-\kappa\left(
v-\frac{v\cdot w_a}{\vartheta^2}w_a-(v\cdot e_b)e_b
\right)
+\frac{v\cdot w_a}{\vartheta^2}
\left(r_jw_a-\vartheta^2s_je_b\right)
+(v\cdot e_b)\left(s_jw_a+r_je_b\right)\\
&=-\kappa v
+\left(
\frac{\kappa+r_j}{\vartheta^2}(v\cdot w_a)
+s_j(v\cdot e_b)
\right)w_a
+\left(
(\kappa+r_j)(v\cdot e_b)-s_j(v\cdot w_a)
\right)e_b.
\end{aligned}
\]
Thus, since \(w_a,e_b\in\mathbb Z^d\) and
\[
\kappa,\ r_j,\ s_j,\ 
\frac{\kappa+r_j}{\vartheta^2}\in\mathbb Z,
\]
we have that \(C_{a,b}(j)\mathbb Z^d\subseteq\mathbb Z^d\), and hence
\[
C_{a,b}(j)\in M_d(\mathbb Z).
\]

Put $\ell:=m+3$ and set
\[
 (a_1,b_1),\ldots,(a_\ell,b_\ell)
 :=(1,1),(1,2),\ldots,(1,m),(2,1),(3,1),(4,1).
\]
Let
\[
 \Sigma:=\{0,\ldots,K-1\}^{\ell},
 \qquad
 B_\sigma:=C_{a_1,b_1}(j_1)\cdots C_{a_\ell,b_\ell}(j_\ell)
 \quad(\sigma=(j_1,\ldots,j_\ell)).
\]
Then $|\Sigma|=K^\ell$ and
$B_\sigma^TB_\sigma=\kappa^{2\ell}I_d$, so take $\lambda=\kappa^\ell$.
Clearly $B_\sigma \in M_d(\mathbb Z)$ for each $\sigma \in \Sigma$.

It remains to show \eqref{eq:6.1}.
Fix $\theta\in \T^d$.
Put
\[
v_h :=
\begin{cases}
w_h, & 1 \le h  \le 4,\\
e_{h-4}, & 5 \le h \le d.
\end{cases}
\]
Since $\mathsf R_w^T \mathsf R_w = \vartheta^2 I_4$ and
$\gcd(\vartheta^2,q)=1$, the matrix
\[
\mathsf R_w^T \oplus I_m
\]
is invertible modulo $q$. Hence, for every
$z \in \mathbb{T}^d[q] \setminus \{0\}$, there is some $h \in [d]$
with $v_h \cdot z \neq 0$. Moreover $\|v_h\| \le \vartheta$, and so,
since $\gamma=\gamma'/\vartheta$,
\[
\operatorname{dist}\bigl(x,\mathbb{T}^d[q]\setminus\{0\}\bigr)
<\gamma
\]
implies
\[
\operatorname{dist}\bigl(v_h\cdot x,
\mathbb{T}^1[q]\setminus\{0\}\bigr)
<\gamma'
\]
for some $h \in [d]$. By taking a union bound, to prove  \eqref{eq:6.1} it therefore suffices to show, for each
$h \in [d]$, that
\begin{equation}\label{eq:6.1.2}
\#\left\{
\sigma \in \Sigma :
\operatorname{dist}\bigl(v_h\cdot B_\sigma^T\theta,
\mathbb{T}^1[q]\setminus\{0\}\bigr)<\gamma'
\right\}
\le \frac{2}{K}|\Sigma|.
\end{equation}

Every basis direction $v_h$ occurs in at least one of the chosen rotation planes.
We vary the last such factor in the product $B_\sigma$, so that the factors to its right only rescale $v_h$; after fixing the other indices, this will leave a linear combination of $r_j$ and $s_j$ to which Lemma~\ref{lem:dmp-planar} applies.
Fix $h$, and let $t$ be largest such that $v_h$ is one of
$w_{a_t}$ or $e_{b_t}$. Fix all $j_s$, $s\neq t$, and for
$0\le j<K$ let $\sigma_j$ be obtained by setting $j_t=j$. Write
\[
B_{\sigma_j}=L C_{a_t,b_t}(j)R,
\]
where \(L\) and \(R\) are the products of the factors preceding and following \(C_{a_t,b_t}(j)\), respectively.
By the choice of \(t\), every factor in \(R\) acts as \(-\kappa I\) on \(v_h\), and hence
\[
Rv_h=(-\kappa)^{\ell-t}v_h.
\]
Thus, writing $x=L^T\theta$,
\[
v_h\cdot B_{\sigma_j}^T\theta
=
v_h\cdot R^T C_{a_t,b_t}(j)^T x
=
(Rv_h)\cdot C_{a_t,b_t}(j)^T x
=
(-\kappa)^{\ell-t}C_{a_t,b_t}(j)v_h \cdot x
=
(-\kappa)^{\ell-t}
x\cdot C_{a_t,b_t}(j)v_h.
\]
Now
\[
C_{a_t,b_t}(j)w_{a_t}
=
r_jw_{a_t}-\vartheta^2s_je_{b_t},
\qquad
C_{a_t,b_t}(j)e_{b_t}
=
s_jw_{a_t}+r_je_{b_t},
\]
and hence in either case
\[
v_h\cdot B_{\sigma_j}^T\theta
=
r_j\omega_1+s_j\omega_2
\]
for some fixed $\omega_1,\omega_2\in\mathbb{T}$, independent of $j$.
Lemma~\ref{lem:dmp-planar} therefore shows that at most two values of $j$ satisfy
\[
\operatorname{dist}\bigl(
v_h\cdot B_{\sigma_j}^T\theta,
\mathbb{T}^1[q]\setminus\{0\}
\bigr)
<\gamma'.
\]
There are $K^{\ell-1}$ choices for the remaining coordinates, so
\[
\#\left\{
\sigma \in \Sigma :
\operatorname{dist}\bigl(
v_h\cdot B_\sigma^T\theta,
\mathbb{T}^1[q]\setminus\{0\}
\bigr)<\gamma'
\right\}
\le
2K^{\ell-1}
=
\frac{2}{K}|\Sigma|.
\]
This proves \eqref{eq:6.1.2}, and therefore \eqref{eq:6.1}, as desired.
\end{proof}

We continue to use the decomposition \(\mathbb R^d=\mathbb R^4\oplus\mathbb R^m\), applying Lemma~\ref{lem:four-dimensional-kernel} on the first quaternionic block. The next lemma averages over finitely many rational rotations in order to recover the \(d\)-dimensional spherical transform.
We do this since Lemma~\ref{lem:four-dimensional-kernel} can only be applied to such a quaternionic block and more importantly
because Lemma~\ref{lem:four-dimensional-kernel} also does not generalize to general $d\ge 5$ due to its use of the integral quaternions \(\mathbb H_{\mathbb Z}\) to transport all of a fixed $W\subseteq \Z^4 \backslash 
\{0\}$ (although it does work for $d=8$ by instead using the integral octonions).

Let $\Pi$ and $\Pi^\perp$
denote the coordinate projections onto the first four and last $d-4$
coordinates, on both $\R^d$ and $\T^d$.  For an odd prime $p$ and $x\in X_p$, put
\[
 S_x:=\mathsf L_x\oplus pI_{d-4}.
\]
Then $S_x\in M_d(\Z)$ and $S_x^TS_x=p^2I_d$.  For $a\in\Z^d$ and
$\theta\in\T^d$, define
\[
 \mathcal K_{p,a}(\theta)
 :=\frac1{|X_p|}\sum_{x\in X_p}e^{2\pi i\theta\cdot S_xa}.
\]
This extends the notation of Section~\ref{sec:four-dimensional-arithmetic}; the quaternionic pairing identity gives
\begin{equation}\label{eq:four-plane-kernel-factorization}
 \mathcal K_{p,a}(\theta)
 =e^{2\pi ip\Pi^\perp(\theta)\cdot\Pi^\perp(a)}
 M_p\!\left(\mathsf R_{\Pi(a)}^T\Pi(\theta)\right).
\end{equation}
For comparison with the spherical transform, put
\[
 \begin{aligned}
 \Xi_a(u)
 &:=\int_{S^3}
 e^{2\pi i u\cdot(\mathsf R_{\Pi(a)}\omega,\Pi^\perp(a))}\,d\sigma_4(\omega)\\
 &=e^{2\pi i\Pi^\perp(u)\cdot\Pi^\perp(a)}
 \Omega_4\!\left(\|\Pi(a)\|\,\|\Pi(u)\|\right).
 \end{aligned}
\]
Since
$\|(\mathsf R_{\Pi(a)}\omega,\Pi^\perp(a))\|=\|a\|$, normalized Haar
measure $dO$ on $SO(d)$ gives
\[
 \int_{SO(d)}\Xi_a(O^Tu)\,dO
 =\Omega_d(\|a\|\,\|u\|)
 \qquad(u\in\R^d).
\]

We now prove the following, which will be our higher dimensional substitute for Lemma~\ref{lem:four-dimensional-kernel}.

\begin{lemma}\label{lem:four-plane-averaging}
Let $d\ge5$, let $W\subseteq\Z^d\setminus\{0\}$ be finite with
$\Pi(a)\ne0$ for every $a\in W$, and let $0<\delta<1$.  Then there exist
positive integers $J,h,Q_0$, a constant $C_0>0$, and matrices
$R_1,\ldots,R_J\in M_d(\Z)$ satisfying
\[
 R_j^TR_j=h^2I_d\qquad(1\le j\le J),
\]
such that
for every sufficiently large odd prime $p$, every $a\in W$, and every
$\theta\in\T^d$ with
\[
 \operatorname{dist}\!\left(\theta,
 \mathfrak R_{Q_0}^d\setminus\{0\}\right)\ge\frac{C_0}{p},
\]
we have
\[
 \left|\overline{\mathcal K}_{p,a}(\theta)
 -\Omega_d\!\left(hp\,\|a\|\operatorname{dist}(\theta,0)\right)\right|<\delta,
\]
where
\[
 \overline{\mathcal K}_{p,a}(\theta)
 :=\frac1J\sum_{j=1}^J\mathcal K_{p,a}(R_j^T\theta)
 \qquad(p\text{ odd prime},\ a\in\Z^d,\ \theta\in\T^d).
\]
\end{lemma}

\begin{proof}
We first choose rational rotations $O_\ell$ whose four-planes satisfy a
spanning property. This property will be used twice: to control small
projections at large frequencies, and to show that simultaneous rational
approximation of sufficiently many projections forces $\theta$ close to a
rational point. A second family $H_k$ will approximate Haar averaging on a
bounded region. Applying the same $H_k$ to all the four-planes in a block
preserves the spanning property. These two families will give a uniform
continuous approximation, which we then combine with
Lemma~\ref{lem:four-dimensional-kernel}.

Put $\varepsilon:=\delta/5$.
Put $r_0:=\lceil d/4\rceil$, and choose $L$ so large that
$(r_0-1)/L<\varepsilon$.  By genericity and the density of
$SO(d,\Q)$ in $SO(d)$ \cite[Proposition~4.1]{djokovic2009rational}, choose
rational rotations $O_1,\ldots,O_L$ so that every $r_0$ of the four-planes
$O_\ell(\R^4\times\{0\})$ span $\R^d$.
Then, by compactness there exists some $\rho>0$ such that for every
nonzero $v\in\R^d$, at most $r_0-1$ indices $\ell$ satisfy
\[
 \|\Pi(O_\ell^Tv)\|<\rho\|v\|.
\]
Choose $T\ge1$ so large that, for every $a\in W$ and $t\ge T$,
\[
 |\Omega_d(\|a\|t)|<\varepsilon,
 \qquad
 |\Omega_4(\|\Pi(a)\|\rho t)|<\varepsilon.
\]

We now average over a common rotation of these four-planes.
Right invariance of Haar measure gives
\[
 \int_{SO(d)}\frac1L\sum_{\ell=1}^L
 \Xi_a(O_\ell^TH^Tu)\,dH
 =\Omega_d(\|a\|\,\|u\|).
\]
By compactness and uniform continuity, the integral on the left can be
approximated uniformly for all $a\in W$ and $\|u\|\le T$ by a finite average.  Since $SO(d,\Q)$ is dense in $SO(d)$, the rotations
may be chosen rational.  Thus choose $H_1,\ldots,H_M\in SO(d,\Q)$ such that
\[
 \left|\frac1{ML}\sum_{k=1}^M\sum_{\ell=1}^L
 \Xi_a(O_\ell^TH_k^Tu)-\Omega_d(\|a\|\,\|u\|)\right|<\varepsilon
 \qquad(a\in W,\ \|u\|\le T).
\]

The spanning property of the $O_\ell$ extends this approximation from
$\|u\|\le T$ to all of $\R^d$.
Indeed, if $\|u\|>T$, then for
each fixed $k$ at most $r_0-1$ indices $\ell$ satisfy
$\|\Pi(O_\ell^TH_k^Tu)\|<\rho\|u\|$.  For every remaining $\ell$,
\[
 |\Xi_a(O_\ell^TH_k^Tu)|<\varepsilon.
\]
Thus
\[
 \left|\frac1{ML}\sum_{k=1}^M\sum_{\ell=1}^L
 \Xi_a(O_\ell^TH_k^Tu)\right|
 \le\frac{r_0-1}{L}+\varepsilon<2\varepsilon,
\]
while $|\Omega_d(\|a\|\,\|u\|)|<\varepsilon$.  Consequently
\begin{equation}\label{eq:6.2.ML}
 \left|\frac1{ML}\sum_{k=1}^M\sum_{\ell=1}^L
 \Xi_a(O_\ell^TH_k^Tu)-\Omega_d(\|a\|\,\|u\|)\right|<3\varepsilon
 \qquad(a\in W,\ u\in\R^d).
\end{equation}

Choose a common denominator $h$ for the finitely many products $H_kO_\ell$
and set
\[
 R_{k,\ell}:=hH_kO_\ell\in M_d(\Z).
\]
Then $R_{k,\ell}^TR_{k,\ell}=h^2I_d$.  Put $J:=ML$ and enumerate these
matrices as $R_1,\ldots,R_J$.

We next use the same spanning property to control projections near
bounded-denominator rational points. Apply Lemma~\ref{lem:four-dimensional-kernel} to
$\Pi(W)\subseteq\Z^4\setminus\{0\}$ with $\varepsilon$, obtaining
$C_1>0$ and a positive integer $Q_1$.  Increasing $Q_1$ if necessary,
assume that $Q_1\ge3$, and increasing $C_1$ if necessary, assume also that
\[
 |\Omega_4(t)|<\varepsilon\qquad(t\ge C_1).
\]

For an odd prime $p$ and
$\theta\in\T^d$, call $(k,\ell)$ major if
\[
 \operatorname{dist}\!\left(
 \Pi(R_{k,\ell}^T\theta),\mathfrak R_{Q_1}^4\right)
 <\frac{C_1}{p}.
\]
This definition includes proximity to zero. We will use local approximation
there, and bound the number of major projections when $\theta$ is away
from every bounded-denominator rational point.
Fix $k$ and $I=\{\ell_1,\ldots,\ell_{r_0}\}\subseteq[L]$, and put
\[
 A_{k,I}:=
 \begin{pmatrix}
  \Pi R_{k,\ell_1}^T\\
  \vdots\\
  \Pi R_{k,\ell_{r_0}}^T
 \end{pmatrix}
 \in M_{4r_0\times d}(\Z).
\]
If $A_{k,I}v=0$, then $H_k^Tv$ is orthogonal to each of the four-planes
\[
 O_{\ell_r}(\R^4\times\{0\}),\qquad 1\le r\le r_0.
\]
These planes span $\R^d$, so $v=0$.  Hence $A_{k,I}$ has rank $d$, and
we may choose $\Lambda_{k,I}\in M_{d\times4r_0}(\Q)$ such that
\[
 \Lambda_{k,I}A_{k,I}=I_d.
\]

Suppose that every $\ell\in I$ is major.  For any lift
$\widetilde\theta\in\R^d$ of $\theta$, there is a vector
$z_I\in\Q^{4r_0}$, each of whose four-dimensional blocks represents a
point of $\mathfrak R_{Q_1}^4$, such that, by the triangle inequality
\[
 \|A_{k,I}\widetilde\theta-z_I\|
 <r_0\frac{C_1}{p}.
\]
It follows that
\[
 \|\widetilde\theta-\Lambda_{k,I}z_I\|
 =
 \|\Lambda_{k,I}(A_{k,I}\widetilde\theta-z_I)\|
 <\|\Lambda_{k,I}\| r_0 \frac{C_1}{p}.
\]
Put $q:=Q_1!$ and choose a positive integer $D$ such that every
$D\Lambda_{k,I}$ is integral. Each block of $z_I$ lies in $q^{-1}\Z^4$,
so $\Lambda_{k,I}z_I\in(Dq)^{-1}\Z^d$. Thus we may take
$Q_0:=Dq\ge Q_1$, ensuring that all these images in $\T^d$ lie in
$\mathfrak R^d_{Q_0}$.
Choose \(C_0>C_1\) so large that
\[
\|\Lambda_{k,I}\|r_0C_1<C_0
\]
for every \(k,I\), and \(|\Omega_d(t)|<\varepsilon\) for \(t\ge C_0\). It follows that
\begin{equation}\label{eq:four-plane-return}
 \#\{\ell:(k,\ell)\text{ is major}\}\ge r_0
 \quad\Longrightarrow\quad
 \operatorname{dist}(\theta,\mathfrak R_{Q_0}^d)
 <\frac{C_0}{p}
 \qquad(1\le k\le M).
\end{equation}

Consequently, if $\operatorname{dist}(\theta,\mathfrak R_{Q_0}^d)\ge C_0/p$,
then each $k$-block contains at most $r_0-1$ major projections.
We now combine this bound with \eqref{eq:6.2.ML} to prove the desired
estimate. Near zero we use the continuous approximation; away from zero
both the averaged kernel and the spherical transform will be small.
Take $p$ sufficiently large that
Lemma~\ref{lem:four-dimensional-kernel} applies to $\Pi(W)$ with
$\varepsilon$ and
\[
 \frac{hC_0+C_1}{p}<\frac1{Q_1}.
\]
Let $a\in W$ and let $\theta$ satisfy the hypothesis of the lemma.

First suppose $\operatorname{dist}(\theta,0)\le C_0/p$.  
Choose the representative \(\theta_0\in\mathbb R^d\) of \(\theta\) with
\[
\|\theta_0\|=\operatorname{dist}(\theta,0).
\]
For each \(k,\ell\), since \(R_{k,\ell}^TR_{k,\ell}=h^2I_d\),
\[
\|R_{k,\ell}^T\theta_0\|
=h\|\theta_0\|
\le \frac{hC_0}{p}.
\]
Thus $\Pi(R_{k,\ell}^T\theta_0)$ is a lift of
$\Pi(R_{k,\ell}^T\theta)$ of norm at most $hC_0/p$.
Every nonzero point of $\mathfrak R^4_{Q_1}$ has distance at least
$1/Q_1$ from zero, so the projected point is at distance at least
$1/Q_1-hC_0/p>C_1/p$ from the nonzero rational points.
Moreover, the lift has norm less than $1/2$ and hence realizes the torus
distance to zero. Lemma~\ref{lem:four-dimensional-kernel} and
\eqref{eq:four-plane-kernel-factorization} therefore give
\[
\left|
\mathcal K_{p,a}(R_{k,\ell}^T\theta)
-
\Xi_a(pR_{k,\ell}^T\theta_0)
\right|
<\varepsilon.
\]
Since
\[
pR_{k,\ell}^T\theta_0
=
O_\ell^TH_k^T(hp\theta_0),
\]
averaging over \(k,\ell\) and applying \eqref{eq:6.2.ML} at \(hp\theta_0\) gives
\[
\left|
\overline{\mathcal K}_{p,a}(\theta)
-
\Omega_d\bigl(hp\|a\|\operatorname{dist}(\theta,0)\bigr)
\right|
<4\varepsilon<\delta.
\]

Suppose instead that $\operatorname{dist}(\theta,0)>C_0/p$.  Together with the
hypothesis of the lemma this gives
\[
 \operatorname{dist}(\theta,\mathfrak R_{Q_0}^d)\ge\frac{C_0}{p}.
\]
By \eqref{eq:four-plane-return}, for every fixed $k$, at most $r_0-1$
indices $\ell$ are major.  For every nonmajor pair,
$0\in\mathfrak R_{Q_1}^4$ gives
\[
p\|\Pi(a)\|\operatorname{dist}\!\left(\Pi(R_{k,\ell}^T\theta),0\right)
 \ge\|\Pi(a)\|C_1\ge C_1.
\]
Lemma~\ref{lem:four-dimensional-kernel} and the choice of $C_1$ therefore give
\[
 |\mathcal K_{p,a}(R_{k,\ell}^T\theta)|<2\varepsilon.
\]
The major terms have absolute value at most $1$, so
\[
 |\overline{\mathcal K}_{p,a}(\theta)|
 <2\varepsilon+\frac{r_0-1}{L}<3\varepsilon.
\]
Moreover
$hp\|a\|\operatorname{dist}(\theta,0)>hC_0\|a\|\ge C_0$, so the
corresponding $\Omega_d$ term has absolute value less than $\varepsilon$.  Hence the desired
difference is less than $4\varepsilon<\delta$.
\end{proof}

We can now combine Lemma~\ref{lem:primary-frames-high} and Lemma~\ref{lem:four-plane-averaging} to obtain our high dimensional transference theorem.

\begin{theorem}\label{thm:radial-transfer}
Let $d\ge5$, let $W\subseteq\Z^d\setminus\{0\}$ be finite, let
$c_\infty\in\R$, and let $c_w\in\R$ for $w\in W$.  Suppose
\begin{equation}\label{eq:high-transfer-radial}
 F(\tau):=c_\infty+\sum_{w\in W}c_w\Omega_d(\|w\|\tau)\ge0
 \qquad(\tau\ge0).
\end{equation}
Then for every $\varepsilon>0$ there exists a positive integer $m$, an injective homomorphism
\[
 E=m^{-1}I_d:\Z^d\longrightarrow\Q^d\subseteq\R^d
\]
and a finite symmetric family $\mathcal G$ of injective homomorphisms
$\Gamma:\Z^d\to\Z^d$, 
such that
\[
 \|E\Gamma w\|=\|w\|
 \qquad(w\in\Z^d,\ \Gamma\in\mathcal G),
\]
and
\begin{equation}\label{eq:high-transfer-positive}
 F_\varepsilon(\theta):=c_\infty+\varepsilon+
 \frac1{|\mathcal G|}\sum_{\Gamma\in\mathcal G}\sum_{w\in W}c_w
 e^{2\pi i\Gamma w\cdot\theta}>0
 \qquad(\theta\in\T^d).
\end{equation}
\end{theorem}

\begin{proof}
Since $\Omega_d(t)\to0$ as $t\to\infty$, letting $\tau\to\infty$ in
\eqref{eq:high-transfer-radial} gives $c_\infty\ge0$.  Discarding the terms with
$c_w=0$, we may assume that $c_w\ne0$ for every $w\in W$.
We may also assume that $W$ is non-empty, otherwise the theorem is trivial.
Put
$\Lambda:=\sum_{w\in W}|c_w|$ and choose $0<\delta<1$ so that
$2\delta \Lambda<\varepsilon$.

To apply Lemma~\ref{lem:four-plane-averaging}, we first put \(W\) in a rational position in which every vector has nonzero projection onto the first four coordinates.
By the density of $SO(d,\Q)$ in $SO(d)$
\cite[Proposition~4.1]{djokovic2009rational}, we can choose an integer $\ell\ge1$ and
$A\in M_d(\Z)$ such that
\[
 A^TA=\ell^2I_d,
 \qquad
 \Pi(Aw)\ne0\quad(w\in W).
\]
Indeed, one can do this by choosing a rational rotation outside the finitely many proper algebraic
sets $\{O\in SO(d):\Pi(Ow)=0\}$, $w\in W$, and then clear denominators.
Apply Lemma~\ref{lem:four-plane-averaging} to $AW$ and $\delta$, obtaining
$J,h,R_1,\ldots,R_J,C_0,Q_0$ as in the lemma.  Choose a positive integer $K$ with $2d/K<\delta$ and apply
Lemma~\ref{lem:primary-frames-high} with $Q_0,K$, obtaining
$\lambda,\gamma,\Sigma$ and matrices $B_\sigma$ $(\sigma\in\Sigma)$.  Finally choose a
sufficiently large odd prime $p$ such that $C_0/p<\gamma$ and the conclusion
of Lemma~\ref{lem:four-plane-averaging} holds.

Although $X_p=-X_p$, the last $d-4$ coordinates of $S_x$ are unchanged when $x$ is replaced by $-x$: $S_{-x}=\mathsf L_{-x}\oplus pI_{d-4}\ne -S_x$.
We therefore include a separate sign index.
Let $\mathcal G$ be the indexed family
\[
 \Gamma_{\sigma,j,x,\xi}:=\xi B_{\sigma}R_jS_xA
 \qquad(\sigma\in\Sigma,\ 1\le j\le J,\ x\in X_p,\ \xi\in\{-1,1\}).
\]
The $\xi$-index makes $\mathcal G$ symmetric and thus $F_\varepsilon(\theta)$ real valued.
Put $E:=(\ell\lambda hp)^{-1}I_d$ (so we may take $m= \ell\lambda hp$ in the theorem).  Using
$\xi^2=1$, $B_\sigma^TB_\sigma=\lambda^2I_d$, $R_j^TR_j=h^2I_d$, 
$S_x^TS_x=p^2I_d$, and $A^TA=\ell^2I_d$ we have
\[
 \Gamma_{\sigma,j,x,\xi}^T\Gamma_{\sigma,j,x,\xi}
 =A^TS_x^TR_j^TB_\sigma^TB_\sigma R_jS_xA
 =(\ell\lambda hp)^2I_d.
\]
Thus every $\Gamma_{\sigma,j,x,\xi}$ is injective and
\[
 \|E\Gamma_{\sigma,j,x,\xi}w\|=\|w\|
 \qquad(w\in\Z^d).
\]

It remains to prove \eqref{eq:high-transfer-positive}.
Fix $\theta\in\T^d$. For $\sigma\in \Sigma$, define
\[
 L_\sigma:=c_\infty+\sum_{w\in W}c_w
 \operatorname{Re}\!\left(\overline{\mathcal K}_{p,Aw}(B_\sigma^T\theta)\right),
\]
(for $\overline{\mathcal K}_{p,Aw}$ as in Lemma~\ref{lem:four-plane-averaging}).
Since $\mathfrak R_{Q_0}^d\backslash \{0\}\subseteq\T^d[Q_0!]\backslash \{0\}$, Lemma~\ref{lem:primary-frames-high}
shows that all but a proportion less than $\delta$ of the $\sigma$ satisfy
\[
 \operatorname{dist}(B_\sigma^T\theta,\mathfrak R_{Q_0}^d\setminus\{0\})
 \ge\gamma>\frac{C_0}{p}.
\]
For every such $\sigma$, taking real parts in
Lemma~\ref{lem:four-plane-averaging} and using $\|Aw\|=\ell\|w\|$ gives
\[
 \left|L_\sigma-F\!\left(\ell hp\operatorname{dist}(B_\sigma^T\theta,0)\right)\right|
 <\delta \Lambda,
\]
and \eqref{eq:high-transfer-radial} therefore gives $L_\sigma>-\delta \Lambda$.  For the remaining $\sigma$,
$|\overline{\mathcal K}_{p,Aw}(B_\sigma^T\theta)|\le1$ and $c_\infty\ge0$, so
$L_\sigma\ge-\Lambda$.  Since the exceptional proportion is less than
$\delta$,
\[
 \frac1{|\Sigma|}\sum_{\sigma\in\Sigma}L_\sigma
 >-(1-\delta)\delta \Lambda-\delta \Lambda
 >-2\delta \Lambda>-\varepsilon.
\]
Since the \(\xi\)-symmetrization in \(\mathcal G\) takes the real part, expanding the definitions gives
\[
F_\varepsilon(\theta)
=
\varepsilon+\frac1{|\Sigma|}\sum_{\sigma\in\Sigma}L_\sigma
>0,
\]
proving 
\eqref{eq:high-transfer-positive}, as desired.
\end{proof}

\section{The main theorem}\label{sec:main-theorem}

In this section we finally prove Theorem~\ref{thm:main}.
The two parts of the proof share a common transference construction; they differ in the cone and normalization conditions checked at the end.
The converses of Theorems~\ref{prop:scalar-certificate} and~\ref{prop:chromatic-certificate} provide certificates in the measurable setting.
Theorems~\ref{thm:radial-transfer-four} and~\ref{thm:radial-transfer} then transfer these certificates into the ordinary setting with Theorems~\ref{thm:lattice-bound} and~\ref{thm:lattice-colouring} certifying the new certificates.

\begin{proof}[Proof of Theorem~\ref{thm:main}]
As discussed, we have the inequalities
\[
\overline\alpha_D(\Q^d) \ge \overline\alpha_D(\R^d) \ge 
m_D(\R^d)
 \qquad\text{and}\qquad
 \chi_D(\Q^d)\le  \chi_D(\R^d) \le \chi_D^{(m)}(\R^d).
\]
The inequality $m_D(\R^d)\le \overline\alpha_D(\R^d)$ was observed by Larman and Rogers \cite{larmanrogers1972} with a simple averaging argument, while the others are immediate from their definitions.

It remains to prove the reverse inequalities.
We shall do this by transferring the measurable certificates provided by
Theorems~\ref{prop:scalar-certificate} and~\ref{prop:chromatic-certificate}.
In the independence case the common notation will be $c=z_0$ and $M_q=Z_q$; in the colouring case it will be $c=z$ and $M_q=S_{Z_q}$.
We first transfer the radial inequality shared by these certificates, and then check their respective cone and normalization conditions.
Thus, suppose that finite nonempty sets
$U_1,\ldots,U_\ell\subseteq\Q^d$, symmetric matrices $M_q$ indexed by
$U_q$, and real numbers $c,a_1,\ldots,a_t$ satisfy
\begin{equation}\label{eq:7.1}
 c-\sum_{q=1}^\ell\tr (M_q)+\sum_{s=1}^t a_s\Omega_d(d_s\tau)
 -\sum_{q=1}^\ell\sum_{\substack{u,v\in U_q\\u\ne v}}
    M_q(u,v)\Omega_d(\|u-v\|\tau)\ge0
 \qquad(\tau\ge0).
\end{equation}
By scaling, we may assume that $d_1^2,\ldots,d_t^2\in \Z$ and $U_1, \ldots , U_\ell \subseteq \Z^d$.

Let
\[
 c_\infty:=c-\sum_{q=1}^\ell\tr (M_q)\ge0,
\]
where the inequality follows from \eqref{eq:7.1} by letting $\tau\to\infty$.  
By Lagrange's four-square theorem choose
$w_s\in\Z^d$, supported on the first four coordinates, with
$\|w_s\|=d_s$.  
Let 
\[
W:=
\{w_s:s\in[t]\}
\cup
\{v-u:q\in[\ell],\ u,v\in U_q,\ u\ne v\}
\subseteq\mathbb Z^d\setminus\{0\}.
\]
Collecting equal nonzero vectors among the $w_s$ and $v-u$ gives the following equivalent form of~\eqref{eq:7.1}
\begin{equation}\label{eq:7.2}
 c_\infty+\sum_{w\in W}c_w\Omega_d(\|w\|\tau)\ge0,
 \qquad(\tau\ge0).
\end{equation}

Fix $\varepsilon>0$.
If $d=4$, then apply
Theorem~\ref{thm:radial-transfer-four} to \eqref{eq:7.2}
and put $r:=8$,
otherwise, if $d\ge5$, apply Theorem~\ref{thm:radial-transfer} to \eqref{eq:7.2} and put $r:=d$.
Let $E$ and $\mathcal G$ be the resulting injective homomorphism and finite symmetric family of injective homomorphisms given by either Theorem~\ref{thm:radial-transfer-four} or Theorem~\ref{thm:radial-transfer} 
depending on whether $d=4$.
Let
\[
C_\varepsilon := \{ \Gamma w_s : \Gamma \in \mathcal{G}, s\in [t] \} \subseteq \Z^r \backslash \{0\}.
\]
By symmetry of $\mathcal{G}$, we have that $C_\varepsilon=-C_\varepsilon$.

Since $\|E\Gamma w_s\|=d_s$ for every $\Gamma\in \mathcal{G}$, $s\in [t]$, we have that $E$ is an
injective graph homomorphism from
\[
 G_\varepsilon:=\Cay(\Z^r,C_\varepsilon)
\]
into the $D$-distance graph on $\R^d$, and therefore
\begin{equation}\label{eq:main-graph-monotonicity}
 \overline\alpha_D(\R^d)\le\overline\alpha(G_\varepsilon),
 \qquad
 \chi_D(\R^d) \ge \chi(G_\varepsilon).
\end{equation}
When $d\ge5$, Theorem~\ref{thm:radial-transfer} moreover gives
$E(\Z^d)\subseteq\Q^d$, so in this case
\begin{equation}\label{eq:main-graph-monotonicity-rational}
 \overline\alpha_D(\Q^d)\le\overline\alpha(G_\varepsilon),
 \qquad
 \chi_D(\Q^d) \ge \chi(G_\varepsilon).
\end{equation}
So, we aim to apply Theorems~\ref{thm:lattice-bound} and~\ref{thm:lattice-colouring} to $G_\varepsilon$.

First collect the coefficients of all pairs $(\Gamma,s)$ that give the same generator, so that their total contribution is preserved. Define
\[
 b_v:=\frac{1}{|\mathcal{G}|}
 \sum_{\substack{\Gamma \in \mathcal{G},\ s\in[t]\\ v=\Gamma w_s}}a_s
 \qquad(v\in C_\varepsilon).
\]
By symmetry of $\mathcal{G}$, we have that $b_v=b_{-v}$. 
Moreover,
\[
 \frac{1}{|\mathcal{G}|}
 \sum_{\Gamma \in \mathcal{G}}\sum_{s=1}^t
 a_se^{2\pi i\Gamma w_s\cdot\theta}
 =\sum_{v\in C_\varepsilon}b_ve^{2\pi i v\cdot\theta},
 \qquad
 \sum_{v\in C_\varepsilon}b_v=\sum_{s=1}^t a_s.
\]
Next relabel each finite constraint along the injective map $\Gamma$ and divide by $|\mathcal G|$, so that averaging preserves the sum of their traces.
For $q\in [\ell]$, $\Gamma \in \mathcal{G}$ put
\[
 U_{q,\Gamma}:=\Gamma(U_q),\qquad
 M_{q,\Gamma}(\Gamma(u),\Gamma(v)):=\frac{1}{|\mathcal{G}|} M_q(u,v).
\]
Then,
\[
c_\infty=c-\sum_{q=1}^\ell\tr (M_q) =  
c-
\sum_{\Gamma \in \mathcal{G}}\sum_{q=1}^\ell\tr (M_{q,\Gamma} ).
\]
The same relabelling and averaging will preserve the normalization quantities needed in the two final applications.
Moreover, by the definition of $M_{q,\Gamma}$,
\[
 \frac{1}{|\mathcal{G}|}
 \sum_{\Gamma \in \mathcal{G}}
 \sum_{q=1}^\ell
 \sum_{\substack{u,v\in U_q\\u\ne v}}
 M_q(u,v)e^{2\pi i\Gamma(v-u)\cdot\theta}
 =
 \sum_{\Gamma \in \mathcal{G}}
 \sum_{q=1}^\ell
 \sum_{\substack{u,v\in U_{q,\Gamma}\\u\ne v}}
 M_{q,\Gamma}(u,v)e^{2\pi i(v-u)\cdot\theta}.
\]
Combining the preceding identities with \eqref{eq:four-transfer-positive} of Theorem \ref{thm:radial-transfer-four}  when \(d=4\), or with \eqref{eq:high-transfer-positive} of Theorem~\ref{thm:radial-transfer} when \(d\ge5\), gives
\begin{equation}\label{eq:main-transfer-identity}
 c+\varepsilon-\sum_{\Gamma \in \mathcal{G}}\sum_{q=1}^\ell\tr (M_{q,\Gamma})
 +\sum_{v\in C_\varepsilon}b_ve^{2\pi i v\cdot\theta}
 -\sum_{\Gamma \in \mathcal{G}}\sum_{q=1}^\ell
      \sum_{\substack{u,v\in U_{q,\Gamma}\\u\ne v}}M_{q,\Gamma}(u,v)
       e^{2\pi i(v-u)\cdot\theta}>0
 \qquad(\theta\in\T^r).
\end{equation}
From here we complete the proof of the independence and chromatic parts of the theorem separately.

We begin with the independence half.
Fix $\varepsilon>0$ and use the converse of
Theorem~\ref{prop:scalar-certificate} to choose a certificate with
\[
 z_0<m_D(\R^d)+\varepsilon.
\]
Apply the preceding construction with $c=z_0$ and $M_q=Z_q$.
Since $\BQC^*$ is preserved by relabelling and positive scaling,
$M_{q,\Gamma}\in\BQC^*(U_{q,\Gamma})$, and we also have that
\[
 \sum_{\Gamma\in\mathcal G}
 \langle M_{q,\Gamma},J_{U_{q,\Gamma}}\rangle
 =\langle Z_q,J_{U_q}\rangle.
\]
Together with $\sum_{v\in C_\varepsilon}b_v=\sum_{s=1}^t a_s$
and \eqref{eq:main-transfer-identity}, the hypotheses of
Theorem~\ref{thm:lattice-bound} are satisfied with
$z_0+\varepsilon$ in place of $z_0$.  Hence
\[
 \overline\alpha(G_\varepsilon)
 \le z_0+\varepsilon
 <m_D(\R^d)+2\varepsilon.
\]
Using \eqref{eq:main-graph-monotonicity} when $d=4$ and
\eqref{eq:main-graph-monotonicity-rational} when $d\ge5$, then letting
$\varepsilon\downarrow0$ and using the inequalities at the beginning
of the proof, gives
\[
 \overline\alpha_D(\R^4)=m_D(\R^4)
\]
and, for $d\ge5$,
\[
 \overline\alpha_D(\Q^d)
 =\overline\alpha_D(\R^d)
 =m_D(\R^d).
\]

We now handle the colouring half similarly.
Choose $k:=\chi_D^{(m)}(\R^d)-1$. Clearly $k\ge 2$.
By the converse of
Theorem~\ref{prop:chromatic-certificate}, choose a certificate and then
choose $\varepsilon>0$ so that
\[
 z+\varepsilon-\frac1{k-1}\sum_{s=1}^t a_s
 +\sum_{q=1}^{\ell}A_{Z_q}<0.
\]
Apply the preceding construction with $c=z$ and $M_q=S_{Z_q}$.
For $q\in[\ell]$ and $\Gamma\in\mathcal G$, define the matrix
$Z_{q,\Gamma}$ indexed by $U_{q,\Gamma}\times[k]$ by
\[
 Z_{q,\Gamma}\bigl((\Gamma u,i),(\Gamma v,j)\bigr)
 :=\frac1{|\mathcal G|}Z_q\bigl((u,i),(v,j)\bigr).
\]
Since \(\BQC^*\) is preserved by relabelling and positive scaling, $
Z_{q,\Gamma}\in\BQC^*(U_{q,\Gamma}\times[k])$.
By the definitions of \(S_Z\) and \(A_Z\), we also have
\[
S_{Z_{q,\Gamma}}=M_{q,\Gamma},
\qquad
\sum_{\Gamma\in\mathcal G}A_{Z_{q,\Gamma}}=A_{Z_q}.
\]
Together with $\sum_{v\in C_\varepsilon}b_v=\sum_{s=1}^t a_s$
and \eqref{eq:main-transfer-identity}, the hypotheses of
Theorem~\ref{thm:lattice-colouring} are satisfied with
$z+\varepsilon$ in place of $z$. Hence
\[
 \chi(G_\varepsilon)>k, \qquad \text{and so} \qquad
 \chi(G_\varepsilon) \ge \chi_D^{(m)}(\R^d).
\]
The theorem now follows from \eqref{eq:main-graph-monotonicity} when $d=4$ and
\eqref{eq:main-graph-monotonicity-rational} when $d\ge5$.
\end{proof}

\section{High-dimensional Euclidean spaces}\label{sec:high-dimensional}

We prove Theorem~\ref{thm:intro-asymptotic} by establishing
\begin{equation}\label{eq:base-two-bound}
 m_1(\R^d)\le(2+o(1))^{-d}
 \qquad\text{and there}\qquad
 \chi^{(m)}(\R^d)\ge(2+o(1))^d.
\end{equation}
The ordinary chromatic bound then follows from Theorem~\ref{thm:main}.
The key observation is that a unit-distance-avoiding set has no overlap
with any of its translates through distance $1$. After compact
approximation, the same holds for an interval of distances around $1$.
The average overlap, and hence all its derivatives, then vanish there.
The following theorem turns these identities into a density bound. It will
also give the numerical bounds in Section~\ref{sec:mone-bounds}.
Recall that $\Omega_d$ averages a Fourier exponential over the unit sphere.
We write $\partial_r^j$ for differentiation $j$ times with respect to $r$,
with the other variables held fixed.

\begin{theorem}\label{lem:base-two-certificate}
Let $H(s)$ be a finite real linear combination of the functions
\[
 \left.\partial_r^j\Omega_d\!\left(\frac{rs}{2\pi}\right)\right|_{r=1},
 \qquad j=0,1,2,\ldots,
\]
normalized by $H(0)=1$. If $H(s)\ge-B$ for all $s\ge0$, where $B>0$,
then
\[
 m_1(\R^d)\le\frac{B}{1+B},\qquad
 \chi^{(m)}(\R^d)\ge\left\lceil1+\frac1B\right\rceil.
\]
\end{theorem}

\begin{proof}
Consider a measurable set avoiding distance $1$, with upper density
arbitrarily close to $m_1(\R^d)$. By the standard periodic approximation,
we may replace it by a periodic avoiding set with arbitrarily small loss
of density; see \cite[proof of Theorem~2.2]{bachoc2015density}.
Within one period cell, inner regularity of Lebesgue measure lets us choose
a compact subset losing arbitrarily little measure. Repeat this subset
periodically. Only finitely many neighbouring copies can contribute
distances near $1$, so their compact distance sets, which avoid $1$,
leave a uniform gap around $1$.

We now smooth the indicator of this periodic set by convolution with a
nonnegative smooth function of integral $1$ supported in a small ball.
This averages nearby translates, so the resulting periodic function
$0\le g\le1$ has the same average; if the ball is small enough, its
support still avoids a smaller interval around distance $1$.

Let $Q$ be a period cell and put
$\delta=\operatorname{vol}(Q)^{-1}\int_Qg(y)\,dy$.
The average overlap of $g$ with its translate by $x$, called its
autocorrelation, has the Fourier expansion
\[
 F(x):=\frac1{\operatorname{vol}(Q)}\int_Qg(y)g(y+x)\,dy
      =\sum_\xi a_\xi e^{2\pi i\xi\cdot x}.
\]
The sum runs over the Fourier frequencies of $g$, and the $a_\xi$ are
the squared absolute values of its Fourier coefficients. Hence
\[
 a_\xi\ge0,\qquad a_0=\delta^2,\qquad
 \sum_\xi a_\xi=F(0)\le\delta,
\]
where the last inequality uses $g^2\le g$.
If $\|x\|$ lies in the distance gap, the product $g(y)g(y+x)$ is zero
for every $y$. Thus $F$ vanishes on an annulus around the unit sphere.
Its spherical average therefore satisfies
\[
 \sum_\xi a_\xi\Omega_d(r\|\xi\|)=0
 \qquad\text{for $r$ near $1$}.
\]
Smoothness permits termwise differentiation of every order. Taking
the derivatives and linear combination that define $H$, we obtain
\[
 0=\sum_\xi a_\xi H(2\pi\|\xi\|)
   \ge\delta^2-B\sum_{\xi\ne0}a_\xi
   \ge\delta^2-B(\delta-\delta^2).
\]
For $\delta>0$, division gives $\delta\le B/(1+B)$.
Taking the density approximations proves the density bound; its
reciprocal gives the chromatic bound, since in a measurable $q$-colouring
some colour class has upper density at least $1/q$.
\end{proof}

In particular, we may take $H(s)=P(s^2)\Omega_d(s/(2\pi))$ for a real
polynomial $P$ with $P(0)=1$.
For a unit vector $\omega$, the negative sum of the second coordinate
derivatives of $e^{isx\cdot\omega}$ is $s^2e^{isx\cdot\omega}$.
Averaging over $\omega$ gives a radial function; writing that sum in
radial coordinates gives the identity
\[
 \left(-\partial_r^2-\frac{d-1}{r}\partial_r\right)
 \Omega_d\!\left(\frac{rs}{2\pi}\right)
 =s^2\Omega_d\!\left(\frac{rs}{2\pi}\right).
\]
By iteration and evaluation at $r=1$, this identity shows that each such $H$ is
a finite linear combination of the derivatives in the
theorem. It therefore suffices to construct such a polynomial with
$B\le2^{-d+o(d)}$.

Put $\nu=d/2-1$. In terms of the Bessel function $J_\nu$ of the first
kind, our Fourier normalization gives
\begin{equation}\label{eq:base-two-transform}
 \Omega_d\!\left(\frac{s}{2\pi}\right)
 =\Gamma(\nu+1)\left(\frac2s\right)^\nu J_\nu(s),
 \qquad \Omega_d(0)=1.
\end{equation}
We shall multiply this function by a polynomial in $s^2$. Placing a
root at each Bessel zero up to $2\nu$ will cancel its sign changes on
that interval. The next lemma counts these prescribed roots and
describes where they lie; both facts will be needed when choosing the
remaining roots.

\begin{lemma}\label{lem:base-two-bessel}
Let $j_1<j_2<\cdots$ be the positive zeros of $J_\nu$.
For every continuous $f:[0,1]\to\R$,
\begin{equation}\label{eq:base-two-zero-distribution}
 \frac1\nu\sum_{j_\ell\le2\nu}
       f\!\left(\frac{\nu^2}{j_\ell^2}\right)
 \longrightarrow
 \int_{1/4}^1 f(t)\frac{\sqrt{1-t}}{2\pi t^{3/2}}\,dt
 \qquad(\nu\to\infty).
\end{equation}
\end{lemma}

\begin{proof}
The standard inequality $j_1>\nu$ places all the points
$\nu^2/j_\ell^2$ in \eqref{eq:base-two-zero-distribution} in $[1/4,1]$;
see \cite[Section~10.21]{NIST:DLMF}.
The function $Y(s)=\sqrt{s}J_\nu(s)$ satisfies
\[
 Y''(s)+q(s)Y(s)=0,\qquad
 q(s)=1-\frac{\nu^2-1/4}{s^2}.
\]
For sufficiently large $\nu$, $q$ is positive and increasing on
$[\nu,2\nu]$. We count zeros on $[a\nu,b\nu]$, where
$1\le a<b\le2$, by following the angle $\theta(s)$ of the vector
$(Y'(s),pY(s))$, for a constant $p>0$. This vector never vanishes,
so its angle can be chosen continuously. Differentiation gives
\[
 \theta'(s)=p\cos^2\theta(s)+\frac{q(s)}p\sin^2\theta(s)>0.
\]
Zeros occur exactly when $\theta$ is a multiple of $\pi$, so each
increase of $\pi$ accounts for one zero, up to an endpoint error of
at most one. Taking $p=\sqrt{q(a\nu)}$ gives
$\theta'\ge p$, and taking $p=\sqrt{q(b\nu)}$ gives $\theta'\le p$.
Thus the number of zeros, divided by $\nu$, lies between
\[
 \frac{b-a}{\pi}\sqrt{1-a^{-2}}+o(1)
 \quad\text{and}\quad
 \frac{b-a}{\pi}\sqrt{1-b^{-2}}+o(1).
\]
Partition $[1,2]$, let $\nu\to\infty$, and then refine the partition.
The bounds give the zero density $\pi^{-1}\sqrt{1-u^{-2}}$ in the
variable $u=s/\nu$. Changing variables to $t=u^{-2}$ proves the claim.
\end{proof}

The factor $s^{-\nu}$ in \eqref{eq:base-two-transform} suggests a
polynomial in $s^2$ of degree $\lfloor\nu/2\rfloor$.
Lemma~\ref{lem:base-two-bessel} shows that cancelling the sign changes
leaves some roots free. We use them to control the product for $s\ge2\nu$.
The change of variable $y=\nu^2/s^2$ sends this range to $0<y\le1/4$;
we include $y=0$ to control the limit $s\to\infty$.
The next lemma places the free roots so that the corresponding polynomial
in $y$ is uniformly small on $[0,1/4]$.

\begin{lemma}\label{lem:base-two-root-placement}
Let $L$ count the zeros $j_\ell\le2\nu$ of $J_\nu$, and put
$k=\lfloor\nu/2\rfloor-L$. For all sufficiently large $\nu$, $k>0$
and there are points $t_1,\ldots,t_k\in(0,1/4)$ such that the monic
polynomial
\[
 R(y)=\prod_{\ell=1}^L\left(y-\frac{\nu^2}{j_\ell^2}\right)
       \prod_{i=1}^k(y-t_i)
\]
satisfies
\begin{equation}\label{eq:base-two-polynomial-norm}
 \max_{0\le y\le1/4}|R(y)|
 \le e^{(1-3\log2+o(1))\nu}\qquad(\nu\to\infty).
\end{equation}
\end{lemma}

\begin{proof}
By Lemma~\ref{lem:base-two-bessel},
$L/\nu$ tends to $\pi^{-1}\int_1^2\sqrt{1-u^{-2}}\,du<1/2$,
so $k>0$ for all sufficiently large $\nu$.
Taking logarithms turns the product $R(y)$ into the sum of
$\log|y-t|$ over its roots. Place mass $1/\nu$ at each root; the total
mass tends to $1/2$, because $\deg R=\lfloor\nu/2\rfloor$.
We shall add the remaining roots so that this sum, divided by $\nu$,
has the same limiting upper bound throughout $[0,1/4]$.

For $a>0$, the useful elementary identity is
\begin{equation}\label{eq:base-two-arcsine}
 \frac1\pi\int_0^a\frac{\log|y-t|}{\sqrt{t(a-t)}}\,dt
 =\log(a/4)\qquad(0\le y\le a).
\end{equation}
The measure with density $1/(\pi\sqrt{t(a-t)})$ has total mass $1$:
the substitution $t=a(1+\cos\phi)/2$ turns its integral into
$\pi^{-1}\int_0^\pi d\phi$.
The same substitution proves \eqref{eq:base-two-arcsine}. Write
$y=a(1+\cos\theta)/2$, factor $\cos\theta-\cos\phi$, and use
$\int_0^\pi\log\sin v\,dv=-\pi\log2$ to obtain $\log a-2\log2$.
The last integral identity follows by integrating
$\log\sin(2v)=\log2+\log\sin v+\log\cos v$ on $[0,\pi/2]$
and using symmetry.

We now average half of these measures over the intervals $[0,u^{-2}]$,
where $1\le u\le2$. More explicitly, define $\mu$ by
\[
 \int f(t)\,d\mu(t)
 =\frac1{2\pi}\int_1^2\int_0^{u^{-2}}
       \frac{f(t)}{\sqrt{t(u^{-2}-t)}}\,dt\,du
\]
for every continuous $f:[0,1]\to\R$. Taking $f=1$ shows that $\mu$
has total mass $1/2$. All the intervals contain $[0,1/4]$, so
\eqref{eq:base-two-arcsine} gives
\begin{equation}\label{eq:base-two-constant-potential}
 \int\log|y-t|\,d\mu(t)
 =\frac12\int_1^2\log\frac1{4u^2}\,du
 =1-3\log2\qquad(0\le y\le1/4).
\end{equation}
On $[1/4,1]$, the measure has density
\[
 \frac1{2\pi\sqrt t}\int_1^{t^{-1/2}}
       \frac{u\,du}{\sqrt{1-tu^2}}
 =\frac{\sqrt{1-t}}{2\pi t^{3/2}},
\]
exactly the density of the forced roots in
\eqref{eq:base-two-zero-distribution}.
The remaining roots must therefore supply the restriction of $\mu$
to $[0,1/4]$. Divide that mass into $k+1$ equal parts and place a root
at each of the $k$ internal division points:
\[
 \mu([0,t_i])=\frac{i}{k+1}\mu([0,1/4])
 \qquad(1\le i\le k).
\]
These points lie in $(0,1/4)$. The zero count and
$L+k=\lfloor\nu/2\rfloor$ give $k/\nu\to\mu([0,1/4])$.
Thus, as $\nu\to\infty$, the measure placing mass $1/\nu$ at each root
approaches $\mu$. Concretely, the sum of the values of any continuous
function at the roots, with weight $1/\nu$ for each root, tends to its
integral against $\mu$: for every continuous $f:[0,1]\to\R$,
\begin{equation}\label{eq:base-two-root-convergence}
 \frac1\nu\sum_{R(t)=0}f(t)\longrightarrow\int f(t)\,d\mu(t).
\end{equation}

To prove the stated bound, choose points $y_\nu$ attaining the maximum
and consider any subsequence on which $y_\nu\to y\in[0,1/4]$.
We cannot apply \eqref{eq:base-two-root-convergence} directly to the
logarithm, which is singular at a root. Instead, fix $M>0$ and replace
it by the larger function $\max\{\log|y-t|,-M\}$.
This function is continuous on the compact rectangle
$[0,1/4]\times[0,1]$, so replacing $y_\nu$ by $y$ does not change
the limit. Equation~\eqref{eq:base-two-root-convergence} gives
\[
 \limsup_{\nu\to\infty}\frac1\nu\log|R(y_\nu)|
 \le\int\max\{\log|y-t|,-M\}\,d\mu(t).
\]
Let $M\to\infty$ and use \eqref{eq:base-two-constant-potential}.
Every sequence of maximizing points has a convergent subsequence,
so this proves the uniform bound \eqref{eq:base-two-polynomial-norm}.
\end{proof}

We now multiply the spherical average by the polynomial determined
by these roots. This cancels its sign changes up to $2\nu$, while
Lemma~\ref{lem:base-two-root-placement} controls the product beyond
that point. Theorem~\ref{lem:base-two-certificate} will then give the
required density bound.

\begin{proof}[Proof of Theorem~\ref{thm:intro-asymptotic}]
Take $L$, $k$ and $t_1,\ldots,t_k$ from
Lemma~\ref{lem:base-two-root-placement}, and set
\begin{equation}\label{eq:base-two-polynomial}
 P(z)=\prod_{\ell=1}^L\left(1-\frac{z}{j_\ell^2}\right)
       \prod_{i=1}^k\left(1-\frac{t_i z}{\nu^2}\right).
\end{equation}
Its constant term is $1$ and its degree is $\lfloor\nu/2\rfloor$.
The positive zeros of $J_\nu$ are simple, and $J_\nu$ is positive
before its first zero \cite[Section~10.21]{NIST:DLMF}. Thus the first
factors change sign together with $J_\nu(s)$; the other
factors are positive for $s\le2\nu$. Hence
\[
 P(s^2)\Omega_d\!\left(\frac{s}{2\pi}\right)\ge0
 \qquad(0\le s\le2\nu).
\]

For $s\ge2\nu$, put $y=\nu^2/s^2\in(0,1/4]$.
Every factor of $P(s^2)$ is the corresponding factor of $R(y)$
divided by $y$. Thus
\[
 P(s^2)\Omega_d\!\left(\frac{s}{2\pi}\right)
 =\Gamma(\nu+1)(2/\nu)^\nu
   y^{\nu/2-\deg P}R(y)J_\nu(s).
\]
Since $\deg P\le\nu/2$, $y\le1/4$ and $|J_\nu(s)|\le1$
\cite[(10.14.1)]{NIST:DLMF}, the polynomial certificate following
Theorem~\ref{lem:base-two-certificate} applies with
\begin{equation}\label{eq:base-two-tail-bound}
 B=\Gamma(\nu+1)(2/\nu)^\nu
       \max_{0\le y\le1/4}|R(y)|.
\end{equation}
Finally, by Stirling's formula, the prefactor in
\eqref{eq:base-two-tail-bound} is $e^{(\log2-1+o(1))\nu}$.
Combining this with
Lemma~\ref{lem:base-two-root-placement} gives
\[
 B\le e^{(-2\log2+o(1))\nu}=2^{-d+o(d)}.
\]
Theorem~\ref{lem:base-two-certificate} now gives \eqref{eq:base-two-bound},
and Theorem~\ref{thm:main} gives the ordinary chromatic bound.
\end{proof}

\section{Low-dimensional Euclidean spaces}
\label{sec:mone-bounds}

We now use Theorem~\ref{lem:base-two-certificate} to obtain explicit
bounds in fixed dimensions. For $5\le d\le24$, we combine radial
derivatives of the unit-distance constraint, choose their coefficients
numerically, and verify the resulting inequalities on the whole
frequency range.

\begin{theorem}\label{thm:mone-lowdim}
For $5\le d\le24$, the upper bounds for $m_1(\R^d)$ and the lower bounds
for $\chi^{(m)}(\R^d)$ are those in the new columns of Table~\ref{tab:previous-new-lowdim}.
The same bounds hold for $\overline\alpha(\R^d)$ and
$\chi(\R^d)$, respectively.
\end{theorem}

The numerical certificates use the following explicit derivative family.
The lemma records the only additional calculation needed to apply the
common certificate theorem.

\begin{lemma}\label{lem:mone-derivative-certificate}
Let $d\ge3$, put $\nu=d/2-1$, and define
\begin{equation}\label{eq:mone-derivative-kernels}
 h_j(s):=\frac{\Gamma(\nu+1)2^\nu}{\nu^j}
             s^{j-\nu}J_{\nu+j}(s),
 \qquad 0\le j\le\left\lfloor\frac{d-1}{2}\right\rfloor.
\end{equation}
If real coefficients $c_j$ with $c_0=1$ satisfy
\[
 H(s):=\sum_jc_jh_j(s)\ge-B\qquad(s\ge0)
\]
for some $B>0$, then the density and chromatic bounds of
Theorem~\ref{lem:base-two-certificate} hold.
\end{lemma}

\begin{proof}
Formula~\eqref{eq:base-two-transform}, applied in dimension $d+2j$, gives
\[
 h_j(s)=\frac{\Gamma(\nu+1)s^{2j}}
              {(2\nu)^j\Gamma(\nu+j+1)}
              \Omega_{d+2j}\!\left(\frac{s}{2\pi}\right).
\]
Thus the functions extend continuously to $0$, with $h_0(s)=\Omega_d(s/(2\pi))$,
$h_0(0)=1$ and $h_j(0)=0$ for $j>0$.
The standard estimate $J_{\nu+j}(s)=O(s^{-1/2})$ for fixed order
\cite[Section~10.17]{NIST:DLMF} shows that the chosen range of $j$
also makes every $h_j$ bounded on $[0,\infty)$.
The Bessel differentiation identity
\cite[(10.6.6)]{NIST:DLMF} gives
\[
 \left.(r^{-1}\partial_r)^j
       \Omega_d\!\left(\frac{rs}{2\pi}\right)\right|_{r=1}
       =(-\nu)^jh_j(s).
\]
Expanding $(r^{-1}\partial_r)^j$ and evaluating at $r=1$ gives a
finite linear combination of ordinary radial derivatives. Thus $H$
is in the family of Theorem~\ref{lem:base-two-certificate}, and
$H(0)=c_0=1$, so the theorem applies.
\end{proof}

For each dimension $5\le d\le24$, the computational supplement
\cite{MoneData} gives exact rational coefficients $c_j$, with $c_0=1$,
and the positive rational parameter $B$ used in each certificate.
A standalone Python~3 verifier, using only the standard library, checks
that the function $H$ in Lemma~\ref{lem:mone-derivative-certificate}
satisfies $H(s)\ge-B$ for every $s\ge0$. It also checks the normalization
and the exact density and integer bounds.

The verification covers $[0,1]$, $[1,T]$ and $[T,\infty)$, with $T$
chosen for each certificate. A rational lower bound handles the first
interval; Taylor's formula, with bounds on all numerical errors and on
its remainder, covers the middle interval; and explicit estimates
control the entire infinite tail. The estimates, their derivations
and the verification code are given in the supplement.

The coefficients were found by a linear program imposing the inequality
at finitely many frequencies, with further constraints added at local
minima, and then rounded to exact rational decimals.
Their global verification establishes the hypothesis of
Lemma~\ref{lem:mone-derivative-certificate}, which gives the measurable
bounds in the new columns of Table~\ref{tab:previous-new-lowdim}. Theorem~\ref{thm:main} gives
the same bounds for the ordinary parameters, proving
Theorem~\ref{thm:mone-lowdim}.

\section{Colouring Euclidean spaces}\label{sec:euclidean-upper-bound}

We now prove Theorem~\ref{thm:intro-upper-bound} with a surprisingly simple
probabilistic construction. We remark that the resulting colouring is Borel.

\begin{proof}[Proof of Theorem~\ref{thm:intro-upper-bound}]
We first construct a random periodic set that avoids a small interval
around distance $1$. Independent copies will cover a finite grid;
slightly enlarging them will then cover all of $\R^d$.
For $d\ge2$, put
\[
\varepsilon=d^{-1/2},\qquad r=\frac{1-\varepsilon}{2}.
\]
Then $0<\varepsilon<1$ and $r>0$.

To construct the set, let $C$ consist of finitely many distinct centres
in $[0,2)^d$, repeated with period $2\Z^d$. Place an open ball of
radius $r$ around each centre. We trim these balls so that the retained
pieces are separated by more than $1+\varepsilon$. For each pair of
centres, the two cutting planes are parallel to their perpendicular
bisector and lie at distance $(1+\varepsilon)/2$ on opposite sides of it.
Around each $a\in C$, retain the piece
\begin{equation}\label{eq:euclidean-upper-cell}
A_a=\left\{a+x:
 \|x\|<r,\quad
 x\cdot\frac{b-a}{\|b-a\|}
 <\frac{\|b-a\|-1-\varepsilon}{2}
 \text{ for every }b\in C\setminus\{a\}
\right\}.
\end{equation}
Write $A=\bigcup_{a\in C}A_a$. Only centres within distance $2$ can
cut a radius-$r$ ball, since $1+\varepsilon+2r=2$.
Each piece is therefore an open ball intersected with finitely many
open half-spaces, so $A$ is open and periodic.

Two points in one piece have distance less than $2r=1-\varepsilon$.
For points $a+x\in A_a$ and $b+y\in A_b$ in different pieces, put
$s=\|b-a\|$ and $u=(b-a)/s$. The two opposing cuts give
\[
x\cdot u<\frac{s-1-\varepsilon}{2},\qquad
y\cdot u>-\frac{s-1-\varepsilon}{2},
\]
and hence
\[
((b+y)-(a+x))\cdot u>1+\varepsilon.
\]
Thus the pieces are disjoint, and $A$ avoids every distance in
$[1-\varepsilon,1+\varepsilon]$. In particular, its open
$\varepsilon/3$-neighbourhood avoids distance $1$, because moving two
points by less than $\varepsilon/3$ changes their distance by less than
$2\varepsilon/3$.

We now bound the region in which another centre would remove a given
point. Fix $a+x$ with $\|x\|<r$. A centre at $a+su$, where $s>0$ and
$\|u\|=1$, removes this point precisely when
\[
s\le1+\varepsilon+2x\cdot u.
\]
This describes a region of forbidden centre displacements contained in
the ball of radius $2$. We bound its volume by enclosing it in a cylinder.
Rotate coordinates so that $x=\|x\|e_1$, and put $t=u_1$.
The distance of $su$ from the $e_1$-axis is $s\sqrt{1-t^2}$.
Using $2\|x\|<1-\varepsilon$, we obtain
\[
\begin{aligned}
s\sqrt{1-t^2}
&\le(1+\varepsilon+2\|x\|t)\sqrt{1-t^2}\\
&\le(1+\varepsilon)(1+|t|)\sqrt{1-t^2}\\
&\le(1+\varepsilon)\frac{3\sqrt3}{4}.
\end{aligned}
\]
The last inequality follows by maximizing
$(1+t)\sqrt{1-t^2}$ on $[0,1]$; its maximum is attained at $t=1/2$.
The first coordinate of $su$ lies in $[-2,2]$, since $\|su\|<2$.
Thus the forbidden region lies in a cylinder about the $e_1$-axis of
height $4$ and radius $(1+\varepsilon)3\sqrt3/4$.
Writing $v_j$ for the volume of the unit ball in $\R^j$, we bound the
forbidden volume by
\begin{equation}\label{eq:euclidean-upper-volume}
W_d:=4v_{d-1}\left(\frac{3\sqrt3}{4}(1+\varepsilon)\right)^{d-1}
\le v_d\left(\frac{3\sqrt3}{4}+o(1)\right)^d.
\end{equation}
The last inequality uses $\varepsilon\to0$ and the standard estimate
$v_{d-1}/v_d=O(\sqrt d)$, obtained from
$v_j=\pi^{j/2}/\Gamma(1+j/2)$ by Stirling's formula.

We choose the number of random centres so that any fixed point of a
ball is retained with conditional probability at least $1/2$. Take
\[
N=\left\lceil\frac{2^d}{2W_d}\right\rceil
\]
independent uniform centres in $[0,2)^d$, and let $C$ consist of
these centres together with all their translates by $2\Z^d$.
The chosen centres are distinct with probability one. Conditional on one
centre $a$ and an offset $x$ in its ball, each other random centre has a
periodic copy that removes $a+x$ with probability at most $W_d/2^d$.
The centre's own nonzero periodic copies are at distance at least $2$
and do not cut its open ball. The union bound therefore shows that
the conditional probability that $a+x$ survives is at least
\[
1-(N-1)\frac{W_d}{2^d}\ge\frac12.
\]

Now fix $z\in\R^d$. For each of the $N$ chosen centres, the probability
that one of its periodic balls contains $z$ is $v_dr^d/2^d$, and there
is at most one such ball. Conditional on this ball containing $z$, the
preceding estimate gives a probability at least $1/2$ that $z$ is
retained. Since the retained pieces are disjoint, we can sum these $N$
contributions. Thus, for every fixed $z$,
\begin{equation}\label{eq:euclidean-upper-probability}
\mathbb P(z\in A)
\ge\frac{Nv_dr^d}{2\cdot2^d}
\ge\frac{v_dr^d}{4W_d}
=:p.
\end{equation}
Since $r\to1/2$, \eqref{eq:euclidean-upper-volume} gives
\begin{equation}\label{eq:euclidean-upper-probability-base}
\frac1p\le\left(\frac{3\sqrt3}{2}+o(1)\right)^d.
\end{equation}

To turn these sets into a colouring, let $G$ be the grid of $(4d)^d$
points in $[0,2)^d$ with spacing $1/(2d)$ in each coordinate.
If we take $k$ independent copies of the random set $A$, then, by
\eqref{eq:euclidean-upper-probability}, a fixed grid point is missed
by all $k$ copies with probability at most $(1-p)^k\le e^{-kp}$.
Another union bound gives
\[
\mathbb P(\text{some grid point is missed})
\le |G|e^{-kp}.
\]
Taking
\begin{equation}\label{eq:euclidean-upper-copies}
k=\left\lceil\frac{\log|G|+1}{p}\right\rceil
\end{equation}
makes this bound at most $e^{-1}<1$.
There is therefore a choice of $k$ sets covering the grid and all its
periodic copies, with each set avoiding every distance in
$[1-\varepsilon,1+\varepsilon]$. Every point of $\R^d$ is within distance
$\sqrt d/(4d)=\varepsilon/4$ of this repeated grid, so the open
$\varepsilon/3$-neighbourhoods of these sets cover all of $\R^d$.
They avoid distance $1$, as shown above. Colouring each point by the
first neighbourhood containing it gives a proper colouring with at
most $k$ colours.  The colour classes are Borel, since each is an open
neighbourhood with the preceding neighbourhoods removed.
Since $\log|G|=d\log(4d)=O(d\log d)$,
\eqref{eq:euclidean-upper-probability-base} and
\eqref{eq:euclidean-upper-copies} give
\[
\chi(\R^d)\le\left(\frac{3\sqrt3}{2}+o(1)\right)^d,
\]
as required.
\end{proof}

\subsection*{Acknowledgements}

We thank Igor Balla, Marthe Bonamy, Margaret Bilu, Rose McCarty, Luke Peachy, Micha{\l} Pilipczuk, and Julian Sahasrabudhe for helpful discussions over the years
In particular, we are grateful to Luke Peachy and Julian Sahasrabudhe for the very helpful encouragement that an estimate such as Theorem~\ref{lem:magyar} should exist in the literature.
I also thank my Mum and Dad for making me at least 100 cups of tea while working on this paper.

\subsection*{Statement of AI use}

We used ChatGPT 5.4 to find the crucial reference \cite[Lemma 1]{magyar2007distribution} (providing Theorem \ref{lem:magyar}) needed for our approach (after failing to find such a reference for some years!).

An early stage version of the presented manuscript contained a result very similar to Theorems~\ref{thm:radial-transfer-four} and~\ref{thm:radial-transfer} without obtaining Theorem~\ref{thm:main} and with only minor support from ChatGPT 5.5.
This draft had similar but weaker applications than the results presented here.
For instance, it proved that $\chi(\R^d) \ge (1.319+o(1))^d$ and $\sup_{|D|=k}\chi_D(\R^d) > (\sqrt{e/2})^{dk}$ for $d\ge 4$. 
It also improved the lower bounds for $\chi(\R^d)$ in low dimensions $4\le d \le 24$.

Since early September we started using ChatGPT more and were able to use ChatGPT 5.6 Sol to complete the proof of Theorem~\ref{thm:main}. Its contribution here was the initial proof of the full converse of Theorem~\ref{prop:scalar-certificate}, and the initial statements and proofs of Theorems~\ref{prop:chromatic-certificate} and~\ref{thm:lattice-colouring}.
With some other very minor adjustments, this completed the proof of Theorem~\ref{thm:main}.
More recently we then used ChatGPT 6.1 Sol to improve the $\chi(\R^d)$ lower bound applications in Sections \ref{sec:high-dimensional} and \ref{sec:mone-bounds}. The initial proof of Theorem \ref{thm:intro-upper-bound} was also found by ChatGPT 6.1 Sol.

ChatGPT 5.6 Sol and 6.1 Sol were used extensively for editing the paper and providing initial drafts of parts of the paper, either from parts it proved or from detailed proof sketches provided by the author.
It was also used to spot and correct numerous minor mistakes throughout the paper as well as finding small improvements to arguments.

Due to the very recent release of \cite{openai2026plane} by OpenAI, we have posted this paper slightly earlier than we would have liked.
We apologize for this.

\bibliographystyle{amsplain}
{\small
\bibliography{EuclideanRamsey}}

\appendix

\section{Finite-point extraction}
\label{app:scalar-certificate}

In this appendix we prove the converse assertion of Theorem~\ref{prop:scalar-certificate}.
The completely-positive framework of DeCorte, de Oliveira Filho and
Vallentin~\cite{decorte2022complete} gives the thick certificate needed
below; the remaining step is to extract from it finitely many point
constraints supported on rational configurations.

For a finite set
$U\subseteq\R^d$, pairwise disjoint balls $B_u=B(u,\delta)$, $u\in U$,
and $Z\in\COP(U)$, put
\begin{equation}\label{eq:thick-radial}
 R_{U,Z,\delta}(\tau):=
 \sum_{u,v\in U} Z(u,v)\int_{B_u}\int_{B_v}
 \Omega_d(\|x-y\| \tau )\,dy\,dx.
\end{equation}
We call such a $R_{U,Z,\delta}(\tau)$ a thick radial function.
By \cite[Lemma 10.2]{decorte2022complete}, we have that $R_{U,Z,\delta}\in C_0([0,\infty))$.

The functions above are precisely the thick radial functions arising from
the copositive constraints of DeCorte, de Oliveira Filho and Vallentin
\cite[Lemmas~10.1--10.2]{decorte2022complete}; in their notation they are the functions
$r$ in \cite[(43)]{decorte2022complete}. Their Theorem~6.3 gives the exact completely
positive formulation, and \cite[(45)--(46)]{decorte2022complete} gives the corresponding
primal--dual pair for an arbitrary finite set of forbidden distances.

In Sections~10.2--10.3 of \cite{decorte2022complete} the subsequent finite-hierarchy
argument is written out after specializing to the unit-distance case.
The proof of their Theorem~10.3 carries over verbatim to finite
$D=\{d_1,\ldots,d_t\}$, since the additional moment functions
$\Omega_d(d_s\,\cdot)$ all belong to $C_0([0,\infty))$ and hence pass
through the same weak-* limit.  Their Theorem~10.4 establishes strong
duality for the corresponding unit-distance hierarchy.
Retaining these finitely many moment
constraints in the corresponding dual argument gives the following
finite-$D$ form of their certificate.

\begin{lemma}[DeCorte, de Oliveira Filho,
Vallentin~\cite{decorte2022complete}]\label{lem:finite-thick-scalar-certificate}
For every $\varepsilon>0$ there are finitely many thick radial functions
$R_1,\ldots,R_N$ and real numbers
$\lambda,a_1,\ldots,a_t$, with
\[
 \lambda<m_D(\R^d)+\varepsilon,
\]
such that
\begin{align}
\lambda+\sum_{s=1}^t a_s-\sum_{q=1}^N R_q(0)&\ge1,\label{eq:thick-cert-normalization}\\
\lambda+\sum_{s=1}^t a_s\Omega_d(d_s\tau)
 -\sum_{q=1}^N R_q(\tau)&\ge0
 \qquad(\tau\ge0).\label{eq:thick-cert-radial}
\end{align}
\end{lemma}

For a finite $U\subseteq\R^d$ and a symmetric matrix $Z$ indexed by $U$,
write
\[
 P_{U,Z}(\tau):=
 \sum_{u,v\in U}Z(u,v)\Omega_d(\|u-v\|\tau).
\]
We now replace finitely many thick radial functions uniformly by finite point constraints supported on rational configurations.
This is an adaptation of the finite-sampling argument from the proof of \cite[Lemma~10.1]{decorte2022complete}.

\begin{lemma}\label{lem:thick-point-approximation}
Let $R_1,\ldots,R_N$ be thick radial functions. For every $\eta>0$ there are finite sets
$U_1,\ldots,U_L\subseteq\Q^d$ and matrices
$Z_\ell\in\COP(U_\ell)$ such that
\[
 \left\|
 \sum_{q=1}^N R_q-\sum_{\ell=1}^L P_{U_\ell,Z_\ell}
 \right\|_\infty<\eta,
 \qquad
 \sum_{\ell=1}^L P_{U_\ell,Z_\ell}(0)
 =\sum_{q=1}^N R_q(0).
\]
\end{lemma}

\begin{proof}
It suffices to treat one thick radial function, since the finitely many
approximations can then be chosen with total error less than $\eta$.
Write $R=R_{U,Z,\delta}$ and put
$V=\operatorname{vol}(B(0,\delta))$.

For each $u\in U$ choose distinct rational points
$x_{u,1},x_{u,2},\ldots\in B_u$ such that
\[
 \frac Vm\sum_{a=1}^m\delta_{x_{u,a}}
\]
converges weakly to Lebesgue measure on $B_u$. Put
\[
 U_m:=\{x_{u,a}:u\in U,\ 1\le a\le m\}, \qquad Z_m(x_{u,a},x_{v,b})
 :=\frac{V^2}{m^2}Z(u,v).
\]
If $w\in\R_{\ge0}^{U_m}$ and
\[
 s_u:=\sum_{a=1}^m w(x_{u,a}) \qquad (u\in U),
\]
then
\[
 w^TZ_mw=\frac{V^2}{m^2}s^TZs\ge0,
\]
and therefore $Z_m\in\COP(U_m)$.
Let $P_m=P_{U_m,Z_m}$.

For each fixed $\tau\ge0$ the function
\[
 (x,y)\longmapsto\Omega_d(\tau\|x-y\|)
\]
is bounded and continuous. Hence weak convergence of the discrete measures
above, and therefore of their product measures, gives for each fixed $\tau\ge0$,
\begin{equation}\label{eq:A2}
 \lim_{m \to \infty} P_m(\tau)= R(\tau).
\end{equation}
Also,
$P_m(0)=R(0)$ for every $m$.

For fixed $m$, since the points $x_{u,a}$ are all distinct and
$\Omega_d(t)\to0$ as $t\to\infty$, only the diagonal terms survive as
$\tau\to\infty$.
We consider these terms now.
Let
\[
 e_m:=\lim_{\tau\to\infty}P_m(\tau)
 =\tr(Z_m)
 =\frac{V^2}{m}\tr(Z).
\]
So, clearly
\begin{equation}\label{eq:A1}
 \lim_{m\to\infty}e_m=0.
\end{equation}
Thus $P_m-e_m\in C_0([0,\infty))$. 
Since $|\Omega_d|\le 1$, the definition of $Z_m$ gives 
$\sup_m\|P_m\|_\infty<\infty$, and hence $P_m-e_m$ is uniformly bounded.
Moreover, by \eqref{eq:A2} and \eqref{eq:A1}, we have for each fixed $\tau\ge0$,
\[
 \lim_{m\to \infty}(P_m-e_m)(\tau)= R(\tau).
\]
By the Riesz--Markov representation theorem
\cite[Theorem~7.17]{Folland}, every continuous linear functional on
$C_0([0,\infty))$ is integration against a finite signed regular Borel
measure. Hence the pointwise convergence and uniform boundedness above,
together with the dominated convergence theorem~\cite[Theorem~2.24]{Folland}, give that $P_m-e_m$ converges weakly to $R$ in
$C_0([0,\infty))$.

By Mazur's theorem \cite[Theorem~3.13]{RudinFA}, there are finite convex
combinations
\[
 F_n=\sum_{m\ge n}\alpha_{n,m}(P_m-e_m)
\]
converging uniformly to $R$, where $\alpha_{n,m}\ge0$ and
$\sum_m\alpha_{n,m}=1$. For the corresponding convex combinations
$G_n:=\sum_{m\ge n}\alpha_{n,m}P_m$, we therefore have that
\[
 \|G_n-R\|_\infty
 \le \|F_n-R\|_\infty+\sup_{m\ge n}|e_m|.
\]
Hence, by \eqref{eq:A1}, $G_n$ converges uniformly to $R$.
Absorbing the nonnegative weights into the matrices therefore gives the required uniform approximation. Moreover, $G_n(0)=R(0)$ for every $n$,
since $P_m(0)=R(0)$ for every $m$.
\end{proof}
The same construction of Lemma~\ref{lem:thick-point-approximation} applies with any fixed finite auxiliary index carried along unchanged. We use this for the colour index in Appendix~\ref{app:chromatic-certificate}.

With Lemma~\ref{lem:finite-thick-scalar-certificate} and Lemma~\ref{lem:thick-point-approximation}, we now quickly prove the full converse assertion of Theorem~\ref{prop:scalar-certificate}.

\begin{proof}[Proof of the converse assertion of
Theorem~\ref{prop:scalar-certificate}]
Apply Lemma~\ref{lem:finite-thick-scalar-certificate} with
$\varepsilon/2$ and Lemma~\ref{lem:thick-point-approximation} with
$\eta=\varepsilon/2$, and put
\[
 z_0:=\lambda+\frac{\varepsilon}{2}.
\]
By \eqref{eq:thick-cert-normalization} and the equality at $0$ in
Lemma~\ref{lem:thick-point-approximation},
\eqref{eq:scalar-cert-normalization} follows from
\[
 P_{U,Z}(0)=\langle Z,J_U\rangle.
\]
Likewise, \eqref{eq:thick-cert-radial} and the uniform approximation in
Lemma~\ref{lem:thick-point-approximation}, together with the added
$\varepsilon/2$, give \eqref{eq:scalar-cert-radial} after splitting off
the diagonal terms in $P_{U,Z}$. Finally,
\[
 z_0<m_D(\R^d)+\varepsilon,
\]
as required.
\end{proof}

\section{Converse for chromatic certificates}
\label{app:chromatic-certificate}

Throughout, $D=\{d_1,\ldots,d_t\}\subseteq(0,\infty)$ is finite and nonempty,
$d\ge2$, and $k\ge2$. We prove the converse assertion of
Theorem~\ref{prop:chromatic-certificate}. Our proof follows the completely
positive approach of DeCorte, de Oliveira Filho and
Vallentin~\cite{decorte2022complete}. We adapt their method of extracting
independent sets to obtain a measurable colouring from a labelled kernel.
Compactness and separation in their framework of thick constraints then
give a finite certificate, which the sampling
construction of Appendix~\ref{app:scalar-certificate} replaces by
rational point constraints.

We first show that a measurable proper colouring of a conull set can,
after changes on a null set, be extended to a measurable proper colouring
of $\R^d$. We suspect that this is a folklore result, but have not found
it in the literature.

\begin{lemma}\label{lem:conull-colouring-completion}
Suppose that $A_1,\ldots,A_k\subseteq\R^d$ are measurable $D$-avoiding sets
whose union is conull. Then $\R^d$ admits a measurable proper $k$-colouring
of its $D$-distance graph.
\end{lemma}

\begin{proof}
First make the sets disjoint by replacing $A_i$ by
$A_i\setminus\bigcup_{j<i}A_j$. A density-one point of a set is a point
where the proportion of the set in shrinking balls centred there tends
to one. Discard the other points of each $A_i$; the Lebesgue density
theorem~\cite[Theorem~3.21]{Folland} says that this changes each set
only on a null set.
The resulting sets are still $D$-avoiding, their union $A$ is conull, and every
point of $A_i$ is a density-one point of $A_i$. Let $c:A\to[k]$ give
$A_i$ colour $i$.

For finite $F\subseteq\R^d$, the intersection
\[
 \bigcap_{x\in F\cap A}(A_{c(x)}-x)
 \ \cap\ \bigcap_{x\in F\setminus A}(A-x)
\]
has positive measure near $0$: the first part has density one there
and the second is conull. A translation $\xi$ in this intersection
gives a proper colouring $x\mapsto c(x+\xi)$ of $F$ agreeing with $c$
on $F\cap A$.
The compactness principle for graph colourings~\cite{debruijnerdos1951}
then gives a proper colouring of $\R^d$ extending $c$. Each colour
class differs from $A_i$ only within the null set $\R^d\setminus A$,
so the colouring is Lebesgue measurable.
\end{proof}

We next show how completely positive correlation matrices give a
measurable colouring.

\Needspace{14\baselineskip}
\begin{lemma}
\label{lem:labelled-cp-extraction}
Let $h:[0,\infty)\to\R$ be continuous with
\[
 h(d_s)=-\frac1{k-1}\qquad(1\le s\le t),
\]
and put
\[
 C_{ij}(x,y):=\frac{1+(k\delta_{ij}-1)h(\|x-y\|)}{k^2}.
\]
If
\[
 C_U:=\bigl(C_{ij}(x,y)\bigr)_{(x,i),(y,j)\in U\times[k]}
 \in\CP(U\times[k])
\]
for every finite nonempty $U\subseteq\Q^d$, then
$\chi_D^{(m)}(\R^d)\le k$.
\end{lemma}

\begin{proof}
By Lemma~\ref{lem:conull-colouring-completion}, it suffices to construct
$k$ measurable $D$-avoiding sets whose union is conull.
For each pair $x,y$, the entries $C_{ij}(x,y)$ sum to one over $i,j$;
also, $C_{ii}(x,y)=0$ when $\|x-y\|\in D$. We will construct
nonnegative measurable functions $f_1,\ldots,f_k$ whose sum is one
and show that their positivity sets avoid these distances after taking
density points. We first factor $C$ on rational points.

For finite nonempty $U\subseteq\Q^d$, write
$C_U=\sum_a b_ab_a^T$ with $b_a\ge0$, and put
$w_a(x):=\sum_i b_a(x,i)$. The block sums of $C$ are one, so
\[
 \sum_a w_a(x)w_a(y)=1,
 \qquad
 \sum_a\bigl(w_a(x)-w_a(y)\bigr)^2=1+1-2=0.
\]
Thus $w_a(x)=w_a$ is independent of $x$. Discard zero factors and
put $p_i^{(a)}(x):=b_a(x,i)/w_a$. Then
\[
 \sum_i p_i^{(a)}(x)=1,\qquad \sum_a w_a^2=1,\qquad
 C_{ij}(x,y)=\sum_a w_a^2p_i^{(a)}(x)p_j^{(a)}(y).
\]

\Needspace{10\baselineskip}
The finite factorizations need not agree as $U$ grows. Compactness
lets us make their integral identities hold simultaneously.
Let $X$ be the space of maps
$p:\Q^d\to\R_{\ge0}^k$ satisfying $\sum_i p_i(q)=1$ for every $q$,
with the topology of pointwise convergence. The possible values at each
$q$ form a compact set, and $\Q^d$ is countable, so $X$ is compact
and metrizable. Choose finite nonempty sets
$U_1\subseteq U_2\subseteq\cdots$ with $\bigcup_n U_n=\Q^d$.
For each $U_n$, extend its factors $p^{(a)}$ to
$X$ by assigning $(1,0,\ldots,0)$ outside $U_n$, and put
$\nu_n:=\sum_a w_a^2\delta_{p^{(a)}}$. These are positive measures
of mass one on $X$. By compactness, after passing to a subsequence,
their integrals of every continuous function converge to those of a
positive measure $\nu$ of mass one; see
\cite[Theorem~7.17]{Folland} and \cite[Theorem~3.17]{RudinFA}.
For any fixed $q,r\in\Q^d$, eventually
$q,r\in U_n$, so
\begin{equation}\label{eq:rational-simplex-moments}
 \int_X p_i(q)p_j(r)\,d\nu(p)=C_{ij}(q,r).
\end{equation}

Continuity of $C$ now extends these rational-point factors to
measurable functions on $\R^d$. The positive weight
$\rho(x):=e^{-\|x\|^2}$ keeps our integrals finite over all of $\R^d$.
Define $\pi_n:\R^d\to\Q^d$ by rounding each coordinate down to a
multiple of $2^{-n}$. These maps are measurable and tend uniformly
to the identity. Put $f_n(p,x):=p(\pi_n(x))$.
As $n,m\to\infty$, \eqref{eq:rational-simplex-moments} gives
\[
\begin{aligned}
 &\int_X\int_{\R^d}\rho(x)^2
 \|f_n(p,x)-f_m(p,x)\|^2\,dx\,d\nu(p)\\
 &\qquad=\frac{2(k-1)}k\int_{\R^d}\rho^2
 \bigl(h(0)-h(\|\pi_n(x)-\pi_m(x)\|)\bigr)\,dx
 \longrightarrow0.
\end{aligned}
\]
The limit follows from continuity of $h$ at $0$ and uniform convergence
of the grid maps. Thus $f_n$ is Cauchy in the weighted $L^2$ space,
and a subsequence converges almost everywhere~\cite[Theorems~6.6 and~2.30]{Folland}.
Let $f(p,x)$ be the limit where it exists and $(1,0,\ldots,0)$
otherwise. It is jointly measurable, with $0\le f_i\le1$ and
$\sum_i f_i=1$. Since $\rho>0$, convergence holds
$\nu\otimes dx$-almost everywhere as well. For almost every $(x,y)$,
the subsequence converges at both points for $\nu$-almost every $p$.
Bounded convergence~\cite[Theorem~2.24]{Folland} in
\eqref{eq:rational-simplex-moments} and
continuity of $C$ therefore give
\begin{equation}\label{eq:labelled-integral-factorization}
 C_{ij}(x,y)=\int_X f_i(p,x)f_j(p,y)\,d\nu(p)
 \qquad\text{for almost every }(x,y).
\end{equation}

Pairs at forbidden distances form a null set in $\R^d\times\R^d$,
so \eqref{eq:labelled-integral-factorization} could fail on all of them.
We therefore integrate over the starting point
and use continuity in the displacement. For $p\in X$, put
\[
 \Gamma_i^p(v):=\int_{\R^d}\rho(x)\rho(x+v)
 f_i(p,x)f_i(p,x+v)\,dx.
\]
This nonnegative weighted overlap integral is positive exactly when
$f_i(p,x)$ and $f_i(p,x+v)$ are both positive on a set of positive
measure. It is continuous because
translations are continuous in $L^2(dx)$~\cite[Proposition~8.5]{Folland} and
$\rho f_i(p,\cdot)\in L^2(dx)$.
After the change of variables $y=x+v$, the Fubini--Tonelli
theorem~\cite[Theorem~2.37]{Folland} applied to
\eqref{eq:labelled-integral-factorization} gives
\begin{equation}\label{eq:labelled-autocorrelation}
 \int_X\Gamma_i^p(v)\,d\nu(p)
 =\frac{1+(k-1)h(\|v\|)}{k^2}
 \int_{\R^d}\rho(x)\rho(x+v)\,dx.
\end{equation}
Initially this holds for almost every $v$. Both sides are
continuous: on the left use
$0\le\Gamma_i^p(v)\le\|\rho\|_2^2$ and finiteness of $\nu$.
Thus it holds for every $v$.

At $v=d_s\omega$ the right side is zero, so the nonnegative
integrand vanishes for $\nu$-almost every $p$. Choose a countable
dense set of directions and fix $p$ for which all the corresponding
overlap integrals vanish, for every $i,s$. Continuity gives
$\Gamma_i^p(d_s\omega)=0$ for every $\omega\in S^{d-1}$.

\Needspace{8\baselineskip}
We now follow the density-point argument in
\cite[proof of Lemma~5.6]{decorte2022complete}.
Let $A_i$ be the density-one points of $\{x:f_i(p,x)>0\}$. If
$x,y\in A_i$ and $v:=y-x$ has length in $D$, both
$\{z:f_i(p,z)>0\}$ and $\{z:f_i(p,z+v)>0\}$ have density one at
$x$. Their intersection has positive measure near $x$, so
$\Gamma_i^p(v)>0$, a contradiction.
Hence each $A_i$ is $D$-avoiding. Since $\sum_i f_i(p,x)=1$, their
positivity sets cover $\R^d$. The Lebesgue density theorem makes
$\bigcup_i A_i$ conull, and
Lemma~\ref{lem:conull-colouring-completion} completes the colouring.
\end{proof}

\Needspace{11\baselineskip}
We now use the thick-constraint framework of DeCorte, de Oliveira
Filho and Vallentin~\cite[Lemmas~10.1--10.2]{decorte2022complete}.
For finite nonempty $U\subseteq\Q^d$, $Z\in\COP(U\times[k])$, and
pairwise disjoint balls $B_u=B(u,\delta)$ of common volume $V$,
associate the normalized thick pair $(A,R)$ defined by
\begin{align*}
 A&:=A_Z,\\
 R(\tau)&:=\frac1{V^2}\sum_{u,v\in U}S_Z(u,v)
 \int_{B_u}\int_{B_v}\Omega_d(\tau\|x-y\|)\,dy\,dx,
\end{align*}
where $A_Z$ and $S_Z$ are as in \eqref{eq:AZ}--\eqref{eq:SZ}.
This normalization gives an equivalent thick system: dividing by
$V^2>0$ preserves each inequality and copositivity.
The ball averages ensure $R\in C_0([0,\infty))$, the space of
continuous functions vanishing at infinity: dominated
convergence applies since $|\Omega_d|\le1$, and
$\Omega_d(\tau\|x-y\|)\to0$ as $\tau\to\infty$ whenever $x\ne y$.
This decay will make the thick inequalities weak-* closed.

\begin{proof}[Proof of the converse assertion of Theorem~\ref{prop:chromatic-certificate}]
Assume $\chi_D^{(m)}(\R^d)>k$. Suppose that a nonnegative measure
$\mu$ on $[0,\infty)$ of mass at most one satisfies
\[
 \int_0^\infty\Omega_d(d_s\tau)\,d\mu(\tau)=-\frac1{k-1}
 \quad(1\le s\le t),
 \qquad
 A+\int_0^\infty R(\tau)\,d\mu(\tau)\ge0
\]
for every thick pair $(A,R)$. Put
$h(r):=\int_0^\infty\Omega_d(r\tau)\,d\mu(\tau)$ and define $C_U$
as in Lemma~\ref{lem:labelled-cp-extraction}. The definitions give
\[
 \langle Z,C_U\rangle
 =A_Z+\sum_{u,v\in U}S_Z(u,v)h(\|u-v\|).
\]
The thick expression $A+\int R\,d\mu$ replaces each $h(\|u-v\|)$
by its average over $B_u\times B_v$.
For the pair arising from $U,Z,\delta$, Fubini and continuity of $h$
give, as $\delta\downarrow0$,
\[
 A+\int R\,d\mu
 \longrightarrow\langle Z,C_U\rangle.
\]
Thus $\langle Z,C_U\rangle\ge0$ for every
$Z\in\COP(U\times[k])$. For finite symmetric matrices, complete
positivity is equivalent to having nonnegative Frobenius pairing with
every copositive matrix~\cite[Section~4]{decorte2022complete}. Hence
$C_U\in\CP(U\times[k])$ for every finite nonempty $U\subseteq\Q^d$,
contradicting Lemma~\ref{lem:labelled-cp-extraction}. The thick
system is therefore infeasible.

We now extract a finite certificate from this infeasibility, following
the compactness and duality strategy of
\cite[Sections~10.2--10.3]{decorte2022complete}. The set
\[
 \mathcal M:=\{\mu\ge0:\mu([0,\infty))\le1\}
\]
is weak-* compact in $C_0([0,\infty))^*$~\cite[Theorem~3.15]{RudinFA}.
Here convergence means
convergence of integrals against functions in $C_0([0,\infty))$.
Mass may escape to infinity: unit point masses moving to infinity
converge to the zero measure in this sense. We therefore allow mass
at most one; Lemma~\ref{lem:labelled-cp-extraction} does not require
$h(0)=1$. Since $d\ge2$, both the moment functions
$\Omega_d(d_s\,\cdot)$ and the thick radial functions belong to
$C_0([0,\infty))$. All constraints are therefore weak-* closed.
If every finite collection could be satisfied, compactness would give
a measure satisfying all of them. Hence there are finitely many thick pairs $(A_q,R_q)$,
$1\le q\le N$, already incompatible with the prescribed moments.

The compact convex set of vectors
\[
 \left\{
 \left(\left(\int\Omega_d(d_s\tau)\,d\mu(\tau)\right)_{s=1}^t,
       \left(\int R_q\,d\mu\right)_{q=1}^N\right):\mu\in\mathcal M
 \right\}
\]
is disjoint from the closed convex set whose first $t$ coordinates
are $-1/(k-1)$ and whose remaining coordinates are at least $-A_q$.

\Needspace{8\baselineskip}
Strict separation~\cite[Theorem~3.4(b)]{RudinFA} gives a weighted
combination of the constraints,
expressed by an affine functional
$z+\sum_s a_sx_s+\sum_q b_qy_q$ that is nonnegative on the first
set and negative on the second. Since the second set is unbounded
in each positive $y_q$-direction, $b_q\le0$. Multiply each defining
copositive matrix by $-b_q$, scaling both $A_q$ and $R_q$ by that
factor, and discard zero multipliers. With $A_q,R_q$ now denoting
the scaled pairs, evaluation at unit point masses $\delta_\tau$ and
at the point with $y_q=-A_q$ gives
\begin{align*}
 z+\sum_{s=1}^t a_s\Omega_d(d_s\tau)-\sum_q R_q(\tau)
 &\ge0\qquad(\tau\ge0),\\
 g:=z-\frac1{k-1}\sum_{s=1}^t a_s+\sum_q A_q&<0.
\end{align*}

It remains to replace the thick pairs by rational point constraints.
We use the sampling argument of~\cite[proof of Lemma~10.1]{decorte2022complete},
adapted for uniform rational approximation in
Lemma~\ref{lem:thick-point-approximation}, with the colour index retained.
For $m$ rational sample points $x_{u,a}$ in each ball, its
labelled matrix is
\[
 Z_m\bigl((x_{u,a},i),(x_{v,b},j)\bigr)
 :=\frac1{m^2}Z_{ui,vj}.
\]
Aggregating nonnegative coordinates over $a$ proves copositivity
of $Z_m$. The same scalar approximation argument applies to the
coefficients $S_Z$, while $A_{Z_m}=A_Z=A$ exactly.
Consequently, for any $\eta>0$ we obtain finitely many rational
labelled configurations $(U_\ell,Z_\ell)$ such that, writing
\[
 P(\tau):=\sum_\ell\sum_{u,v\in U_\ell}S_{Z_\ell}(u,v)
 \Omega_d(\|u-v\|\tau),
\]
the constants are preserved and the error is less than $\eta$ at every
$\tau\ge0$:
\[
 \sum_\ell A_{Z_\ell}=\sum_q A_q,
 \qquad \left|\sum_q R_q(\tau)-P(\tau)\right|<\eta.
\]
Take $\eta:=-g/2>0$. Increasing $z$ by $\eta$ compensates for the
approximation error at every $\tau$, so the radial inequality remains
nonnegative. The inequality~\eqref{eq:chromatic-cert-gap} remains strict,
since its left-hand side is now $g+\eta=g/2<0$.
Finally, $\Omega_d(0)=1$ gives
\[
 P(\tau)=\sum_\ell\tr(S_{Z_\ell})+
 \sum_\ell\sum_{\substack{u,v\in U_\ell\\u\ne v}}
 S_{Z_\ell}(u,v)\Omega_d(\|u-v\|\tau)
\]
so the radial inequality is \eqref{eq:chromatic-cert-radial}.
\end{proof}

\end{document}


\maketitle
\vspace{-2.5em}

This supplement supplies the exact certificates and the global verification
used in Section~9 of the accompanying paper. It contains all twenty
coefficient lists and explains the estimates implemented by the standalone
Python~3 verifier. The verification uses only the Python standard library;
no optimization software is needed to check the certificates.

\section{Certificates and bounds}

For each dimension, put $\nu=d/2-1$ and
$K=\lfloor(d-1)/2\rfloor$. Use the functions from Lemma~9.2 of the paper,
\begin{equation}\label{eq:kernels}
 h_j(s)=\frac{\Gamma(\nu+1)2^\nu}{\nu^j}
              s^{j-\nu}J_{\nu+j}(s),\qquad 0\le j\le K,
 \qquad H(s)=\sum_{j=0}^K c_jh_j(s).
\end{equation}
Here $s$ is the Bessel argument: $h_0(s)=\Omega_d(s/(2\pi))$ in the
paper's Fourier convention. The functions are defined at zero by
continuity, with $h_0(0)=1$ and $h_j(0)=0$ for $j>0$.
Every coefficient and every $B$ below is an exact rational decimal.
All certificates have $c_0=1$, so $H(0)=1$.

For the coefficients in Section~\ref{sec:coefficients}, the verifier
establishes
\begin{equation}\label{eq:certificate}
 H(s)\ge-B\qquad\text{for every }s\ge0.
\end{equation}
Theorem~8.1 and Lemma~9.2 of the paper therefore give
\[
 m_1(\R^d)\le\frac{B}{1+B},\qquad
 \chi_m(\R^d)\ge\left\lceil1+\frac1B\right\rceil.
\]
Table~\ref{tab:bounds} records these consequences and the cutoff $T$
used in the verification. The ordinary chromatic and independence-density
conclusions follow from Theorem~1.1 of the paper.

\begin{table}[tbp]
\centering
\small
\setlength{\tabcolsep}{7pt}
\renewcommand{\arraystretch}{1.04}
\begin{tabular}{@{}rrrrrr@{}}
\toprule
$d$ & $K$ & $T$ & $B$ & $m_1(\R^d)\le$ & $\chi_m(\R^d)\ge$\\
\midrule
5 & 2 & 50 & $0.0440974$ & $0.04223494858$ & 24 \\
6 & 2 & 50 & $0.0196765$ & $0.01929680639$ & 52 \\
7 & 3 & 50 & $0.0103676$ & $0.01026121582$ & 98 \\
8 & 3 & 50 & $0.00580985$ & $0.005776290618$ & 174 \\
9 & 4 & 50 & $0.00280977$ & $0.002801897313$ & 357 \\
10 & 4 & 110 & $0.00137375$ & $0.0013718654$ & 729 \\
11 & 5 & 50 & $0.000723561$ & $0.0007230378381$ & 1,384 \\
12 & 5 & 60 & $0.000397746$ & $0.0003975878611$ & 2,516 \\
13 & 6 & 50 & $0.000178119$ & $0.0001780872793$ & 5,616 \\
14 & 6 & 230 & $7.88264\times10^{-5}$ & $7.882018689\times10^{-5}$ & 12,688 \\
15 & 7 & 70 & $4.27521\times10^{-5}$ & $4.275027234\times10^{-5}$ & 23,392 \\
16 & 7 & 110 & $2.40671\times10^{-5}$ & $2.406652079\times10^{-5}$ & 41,552 \\
17 & 8 & 50 & $1.15532\times10^{-5}$ & $1.155306653\times10^{-5}$ & 86,558 \\
18 & 8 & 370 & $5.61038\times10^{-6}$ & $5.610348524\times10^{-6}$ & 178,243 \\
19 & 9 & 100 & $2.94066\times10^{-6}$ & $2.940651353\times10^{-6}$ & 340,061 \\
20 & 9 & 170 & $1.53933\times10^{-6}$ & $1.539327631\times10^{-6}$ & 649,635 \\
21 & 10 & 60 & $6.73174\times10^{-7}$ & $6.731735469\times10^{-7}$ & 1,485,502 \\
22 & 10 & 560 & $3.31407\times10^{-7}$ & $3.314068902\times10^{-7}$ & 3,017,439 \\
23 & 11 & 140 & $1.76559\times10^{-7}$ & $1.765589689\times10^{-7}$ & 5,663,831 \\
24 & 11 & 250 & $9.51021\times10^{-8}$ & $9.510209096\times10^{-8}$ & 10,515,016 \\
\bottomrule
\end{tabular}
\caption{Exact certificate parameters and their consequences. The density
bound $B/(1+B)$ is rounded upwards to ten significant figures; the integer
bound is computed from the exact reciprocal.}
\label{tab:bounds}
\end{table}

\section{Verification on the whole half-line}

The verification covers $[0,1]$, $[1,T]$ and $[T,\infty)$ separately.
The first interval has a direct rational bound. On the finite interval,
the program evaluates intervals containing the exact values of $H$ and
its first six derivatives, then applies Taylor's formula on each
subinterval. The infinite tail is controlled by an explicit Bessel
bound in even dimensions and by exact trigonometric formulas in odd
dimensions. The standard Bessel identities used below are recorded
in~\cite{NIST:DLMF}.

For the interval near zero, set
\[
 \beta_j=\frac1{(2\nu)^j\prod_{\ell=1}^j(\nu+\ell)},
 \qquad h_j(s)=\beta_js^{2j}
              \Omega_{d+2j}\!\left(\frac{s}{2\pi}\right),
\]
where an empty product is $1$. The spherical average satisfies
$|\Omega_n|\le1$. Averaging $\cos x\ge1-x^2/2$ also gives
$h_0(s)\ge1-s^2/(2d)$, since the average squared coordinate on the
unit sphere is $1/d$. Consequently,
\begin{equation}\label{eq:near-zero}
 H(s)\ge1-\frac1{2d}-\sum_{j=1}^K|c_j|\beta_j
 \qquad(0\le s\le1).
\end{equation}
The expression on the right is rational and positive for every certificate.

For the finite interval, the program uses the normalized Bessel series
\begin{equation}\label{eq:series}
 h_m(s)=\beta_ms^{2m}
  \sum_{\ell=0}^{\infty}
    \frac{(-1)^\ell(s^2/4)^\ell}
         {\ell!\prod_{q=1}^{\ell}(\nu+m+q)}.
\end{equation}
The ratio of consecutive term magnitudes is
$(s^2/4)/((\ell+1)(\nu+m+\ell+1))$.
Once this ratio is less than $1$, it continues to decrease, so the
alternating-series remainder is bounded by the next term.
The program sums with decimal arithmetic rounded down for lower
endpoints and up for upper endpoints, and includes this remainder in
the resulting interval. Thus cancellation in the series is fully covered.

Derivatives are obtained from
$h_m'(s)=s^{-1}(2mh_m(s)-\nu h_{m+1}(s))$.
More explicitly, write
\begin{equation}\label{eq:derivatives}
 h_k^{(r)}(s)=s^{-r}\sum_{j=0}^r a_{k,r,j}\nu^jh_{k+j}(s).
\end{equation}
The coefficients are integers, starting with $a_{k,0,0}=1$ and satisfying
\[
 a_{k,r+1,j}=(2(k+j)-r)a_{k,r,j}-a_{k,r,j-1},
\]
with missing indices interpreted as zero. The verifier uses this
recurrence to obtain the derivative intervals.
All coefficients multiplying the intervals in \eqref{eq:derivatives}
are evaluated as exact rational numbers before their outward-rounded
conversion to decimal intervals.

Consider a subinterval $[a,b]\subseteq[1,T]$, with midpoint $m$ and
half-width $\rho$. These midpoint calculations, including the squared
midpoint used in \eqref{eq:series}, are checked for exact representability.
Let $L_j\le H^{(j)}(m)\le U_j$ be the computed intervals, and put
$M_j=\max\{|L_j|,|U_j|\}$. Taylor's formula gives the lower bound
\begin{equation}\label{eq:taylor}
 H(m+x)\ge L_0-M_1|x|+\tfrac12L_2x^2
     -\sum_{j=3}^6\frac{M_j\rho^j}{j!}
     -\frac{V\rho^7}{7!},\qquad |x|\le\rho,
\end{equation}
where $V$ is a rational upper bound for $|H^{(7)}|$ on $[a,b]$.
The minimum of the quadratic part is computed exactly: if $L_2>0$,
it occurs at $|x|=\min\{\rho,M_1/L_2\}$; if $L_2\le0$, it occurs at
$|x|=\rho$.

To obtain $V$, the program combines \eqref{eq:derivatives} with two
bounds for each $s^{-7}h_m(s)$.
Writing $C=\Gamma(\nu+1)2^\nu$, these are
\[
 |s^{-7}h_m(s)|\le
 \min\left\{C\nu^{-m}s^{m-\nu-7},\,
                 \beta_ms^{2m-7}\right\}.
\]
They follow respectively from $|J_{\nu+m}(s)|\le1$ and
$|\Omega_{d+2m}|\le1$.
The maximum of each power on $[a,b]$ occurs at an endpoint.
For rational computation, $s^{m-\nu-7}$ is replaced by the larger
$s^{\lceil m-\nu-7\rceil}$, which is valid because $s\ge1$.
When $\nu$ is an integer, $C$ is an integer; when it is a half-integer,
$C^2/\pi$ is rational and $\pi<22/7$ supplies a rational square-root
upper bound for $C$. Substituting these bounds in
\eqref{eq:derivatives} with $r=7$ gives $V$ by exact rational arithmetic.

If the minimum in \eqref{eq:taylor} is at least $-B$, the whole
subinterval passes. Otherwise the program bisects it and checks both
halves. A successful run exhausts the list of subintervals, so it covers
every point of $[1,T]$. Decimal precision is at least $180$ digits and
is increased for the larger cutoffs. The proof uses the computed interval
endpoints and rational comparisons, rather than a presumed accuracy of
the decimal approximations.

For the even-dimensional tail, put $y(s)=\sqrt{s}J_\alpha(s)$, where
$\alpha\ge1/2$. Bessel's equation becomes
\[
 y''+qy=0,\qquad q(s)=1-\frac{\alpha^2-1/4}{s^2}.
\]
The quantity $E=y'^2+qy^2$ satisfies $E'=q'y^2\ge0$.
The large-argument Bessel formulas give $E(s)\to2/\pi$; hence
$E(s)\le2/\pi$. Where $q>0$, this proves
\begin{equation}\label{eq:bessel-tail}
 |J_\alpha(s)|\le
    \left(\frac{2s}{\pi(s^2-\alpha^2+1/4)}\right)^{1/2}.
\end{equation}
Applying this to $\alpha=\nu+j$ gives
\[
 |c_jh_j(s)|\le
 \frac{|c_j|C}{\nu^j}
 \left(\frac{2s^{2j-2\nu-1}}
  {\pi(1-((\nu+j)^2-1/4)/s^2)}\right)^{1/2}.
\]
Because $j\le K$ gives $2j-2\nu-1\le0$, each bound is
nonincreasing once its denominator is positive.
The program checks $T^2>(\nu+K)^2-1/4$ and that the sum of the
bounds at $T$ is at most $B$. This proves $|H(s)|\le B$ for every
$s\ge T$. Rational square-root upper bounds are used throughout.
The lower bound $3.14159265358979<\pi$ is checked by the alternating
arctangent series and Machin's identity
$\pi=16\arctan(1/5)-4\arctan(1/239)$.

For odd $d$, write $\nu=n+1/2$. The exact half-integer formulas take
the form
\[
 J_{\ell+1/2}(s)=\sqrt{\frac2{\pi s}}
       \bigl(p_\ell(1/s)\sin s+q_\ell(1/s)\cos s\bigr).
\]
The rational polynomials start with
$p_0=1$, $q_0=0$, $p_1=x$ and $q_1=-1$, and both follow
the recurrence $f_{\ell+1}(x)=(2\ell+1)xf_\ell(x)-f_{\ell-1}(x)$.
In \eqref{eq:kernels}, the square-root and gamma factors cancel to
give the rational factor $(2n+1)!!\nu^{-j}s^{j-n-1}$.
Since $j\le K=n+1$, this leaves only nonnegative powers of $1/s$.
Thus the program computes rational polynomials $A$ and $D$ with
\[
 H(s)=A(1/s)\sin s+D(1/s)\cos s.
\]
Their squared amplitude has no odd powers of $1/s$, as also checked
exactly by the program. Therefore
\begin{equation}\label{eq:amplitude}
 |H(s)|^2\le A(1/s)^2+D(1/s)^2=Q(1/s^2)
\end{equation}
for a rational polynomial $Q$.
The entire tail corresponds to $0\le x=1/s^2\le T^{-2}$, including
the limiting value $x=0$.

To bound $Q$, write $Q(x)=\sum_{j=0}^N q_jx^j$ and expand
\[
 Q(T^{-2}t)=\sum_{i=0}^N b_i\binom Ni t^i(1-t)^{N-i},
 \qquad
 b_i=\sum_{j=0}^i q_jT^{-2j}\frac{\binom ij}{\binom Nj}.
\]
For $0\le t\le1$, the basis functions are nonnegative and sum to
$1$, so $Q(T^{-2}t)\le\max_i b_i$.
The program checks $\max_i b_i\le B^2$ by exact rational arithmetic.
Together with \eqref{eq:amplitude}, this proves $|H(s)|\le B$
throughout the tail, including any nonzero limiting oscillation amplitude.

All three ranges pass for every certificate in
Table~\ref{tab:bounds}. This establishes \eqref{eq:certificate}.
The program also checks the exact fraction $B/(1+B)$, its upward-rounded
decimal display, and the integer $\lceil1+1/B\rceil$.

\section{Reproducing the verification}

Keep \texttt{section9\_derivative\_certificates.json} beside
\texttt{verify\_section9\_derivatives.py} and run
\begin{verbatim}
python3 verify_section9_derivatives.py > section9_verification.jsonl
\end{verbatim}
Run standard Python~3 without its \texttt{-O} or \texttt{-OO} options,
which disable assertions. Optional dimension arguments select a subset,
for example \texttt{python3 verify\_section9\_derivatives.py 5 12 24}.
The default run checks all twenty certificates. A successful process
exits with status zero and prints one JSON record per dimension.

The accompanying log was produced by a complete successful run.
Its approximate summary values and timings are for inspection;
all acceptance decisions use rational comparisons or intervals computed
with directed rounding. The coefficient data and verifier suffice to
repeat the checks independently. The search optimizer is not part of
the verification.

The archive also includes \texttt{section9\_bounds.csv}, containing the
twenty table rows, and \texttt{SHA256SUMS.txt}, which identifies the
distributed file versions. The printed coefficients in the next section
are exactly the rational decimals in the JSON file.

\clearpage
\section{Exact coefficient lists}\label{sec:coefficients}

In every dimension $c_0=1$. The following tables give all the remaining
coefficients; $B$, $K$ and $T$ are in Table~\ref{tab:bounds}.
No decimal in these lists denotes an approximation.

\noindent
\begin{minipage}[t]{0.47\textwidth}
\centering\fontsize{9}{11}\selectfont\renewcommand{\arraystretch}{0.94}
\textbf{Dimension $d=5$}\par\vspace{3pt}
\begin{tabular}{@{}rr@{}}
$j$ & $c_j$\\\hline
1 & \texttt{-0.066144184854403612}\\
2 & \texttt{0.033072670895074228}\\
\end{tabular}
\end{minipage}
\hfill
\begin{minipage}[t]{0.47\textwidth}
\centering\fontsize{9}{11}\selectfont\renewcommand{\arraystretch}{0.94}
\textbf{Dimension $d=6$}\par\vspace{3pt}
\begin{tabular}{@{}rr@{}}
$j$ & $c_j$\\\hline
1 & \texttt{-0.130564833102734951}\\
2 & \texttt{0.065280859385139334}\\
\end{tabular}
\end{minipage}
\par\vspace{6pt}
\noindent
\begin{minipage}[t]{0.47\textwidth}
\centering\fontsize{9}{11}\selectfont\renewcommand{\arraystretch}{0.94}
\textbf{Dimension $d=7$}\par\vspace{3pt}
\begin{tabular}{@{}rr@{}}
$j$ & $c_j$\\\hline
1 & \texttt{-0.191402895212792135}\\
2 & \texttt{0.095708212875307230}\\
3 & \texttt{0.000000486783660435}\\
\end{tabular}
\end{minipage}
\hfill
\begin{minipage}[t]{0.47\textwidth}
\centering\fontsize{9}{11}\selectfont\renewcommand{\arraystretch}{0.94}
\textbf{Dimension $d=8$}\par\vspace{3pt}
\begin{tabular}{@{}rr@{}}
$j$ & $c_j$\\\hline
1 & \texttt{-0.250964975520935996}\\
2 & \texttt{0.125483317171346254}\\
3 & \texttt{0.000003164075371536}\\
\end{tabular}
\end{minipage}
\par\vspace{6pt}
\noindent
\begin{minipage}[t]{0.47\textwidth}
\centering\fontsize{9}{11}\selectfont\renewcommand{\arraystretch}{0.94}
\textbf{Dimension $d=9$}\par\vspace{3pt}
\begin{tabular}{@{}rr@{}}
$j$ & $c_j$\\\hline
1 & \texttt{-0.362017742401780807}\\
2 & \texttt{0.199038483006752892}\\
3 & \texttt{-0.020650157848473918}\\
4 & \texttt{0.004015587653468221}\\
\end{tabular}
\end{minipage}
\hfill
\begin{minipage}[t]{0.47\textwidth}
\centering\fontsize{9}{11}\selectfont\renewcommand{\arraystretch}{0.94}
\textbf{Dimension $d=10$}\par\vspace{3pt}
\begin{tabular}{@{}rr@{}}
$j$ & $c_j$\\\hline
1 & \texttt{-0.471646755847647536}\\
2 & \texttt{0.278701422375559782}\\
3 & \texttt{-0.047636823244562337}\\
4 & \texttt{0.009527300708383315}\\
\end{tabular}
\end{minipage}
\par\vspace{6pt}
\noindent
\begin{minipage}[t]{0.47\textwidth}
\centering\fontsize{9}{11}\selectfont\renewcommand{\arraystretch}{0.94}
\textbf{Dimension $d=11$}\par\vspace{3pt}
\begin{tabular}{@{}rr@{}}
$j$ & $c_j$\\\hline
1 & \texttt{-0.567901686579266207}\\
2 & \texttt{0.353247174659388707}\\
3 & \texttt{-0.075386234541086866}\\
4 & \texttt{0.015419786377334148}\\
5 & \texttt{0.000000126371535913}\\
\end{tabular}
\end{minipage}
\hfill
\begin{minipage}[t]{0.47\textwidth}
\centering\fontsize{9}{11}\selectfont\renewcommand{\arraystretch}{0.94}
\textbf{Dimension $d=12$}\par\vspace{3pt}
\begin{tabular}{@{}rr@{}}
$j$ & $c_j$\\\hline
1 & \texttt{-0.656408012141682096}\\
2 & \texttt{0.425949933617937759}\\
3 & \texttt{-0.104726426426660954}\\
4 & \texttt{0.021817257616005257}\\
5 & \texttt{0.000000104600839289}\\
\end{tabular}
\end{minipage}
\par\vspace{6pt}
\noindent
\begin{minipage}[t]{0.47\textwidth}
\centering\fontsize{9}{11}\selectfont\renewcommand{\arraystretch}{0.94}
\textbf{Dimension $d=13$}\par\vspace{3pt}
\begin{tabular}{@{}rr@{}}
$j$ & $c_j$\\\hline
1 & \texttt{-0.784716293839608392}\\
2 & \texttt{0.553707526598557598}\\
3 & \texttt{-0.175999077673894083}\\
4 & \texttt{0.044951075871507200}\\
5 & \texttt{-0.003880892695589451}\\
6 & \texttt{0.000474302732300296}\\
\end{tabular}
\end{minipage}
\hfill
\begin{minipage}[t]{0.47\textwidth}
\centering\fontsize{9}{11}\selectfont\renewcommand{\arraystretch}{0.94}
\textbf{Dimension $d=14$}\par\vspace{3pt}
\begin{tabular}{@{}rr@{}}
$j$ & $c_j$\\\hline
1 & \texttt{-0.903698324846007317}\\
2 & \texttt{0.682235291916104325}\\
3 & \texttt{-0.254397246224883189}\\
4 & \texttt{0.072913789422781664}\\
5 & \texttt{-0.009367498920941438}\\
6 & \texttt{0.001170922078261791}\\
\end{tabular}
\end{minipage}
\par\vspace{6pt}
\newpage
\noindent
\begin{minipage}[t]{0.47\textwidth}
\centering\fontsize{9}{11}\selectfont\renewcommand{\arraystretch}{0.94}
\textbf{Dimension $d=15$}\par\vspace{3pt}
\begin{tabular}{@{}rr@{}}
$j$ & $c_j$\\\hline
1 & \texttt{-1.006820977802063588}\\
2 & \texttt{0.803653953819001732}\\
3 & \texttt{-0.336029969816147855}\\
4 & \texttt{0.105845324480242003}\\
5 & \texttt{-0.017003943162352684}\\
6 & \texttt{0.002166987823016599}\\
7 & \texttt{0.000000027328550885}\\
\end{tabular}
\end{minipage}
\hfill
\begin{minipage}[t]{0.47\textwidth}
\centering\fontsize{9}{11}\selectfont\renewcommand{\arraystretch}{0.94}
\textbf{Dimension $d=16$}\par\vspace{3pt}
\begin{tabular}{@{}rr@{}}
$j$ & $c_j$\\\hline
1 & \texttt{-1.104975919339296064}\\
2 & \texttt{0.926577936775460076}\\
3 & \texttt{-0.423725031536075714}\\
4 & \texttt{0.143301196637574818}\\
5 & \texttt{-0.026297130403147694}\\
6 & \texttt{0.003408825991131504}\\
7 & \texttt{0.000000024658646875}\\
\end{tabular}
\end{minipage}
\par\vspace{6pt}
\noindent
\begin{minipage}[t]{0.47\textwidth}
\centering\fontsize{9}{11}\selectfont\renewcommand{\arraystretch}{0.94}
\textbf{Dimension $d=17$}\par\vspace{3pt}
\begin{tabular}{@{}rr@{}}
$j$ & $c_j$\\\hline
1 & \texttt{-1.237723724775380951}\\
2 & \texttt{1.117249702351442275}\\
3 & \texttt{-0.581994214295437295}\\
4 & \texttt{0.223932499048137279}\\
5 & \texttt{-0.052680776297886221}\\
6 & \texttt{0.008864258011452238}\\
7 & \texttt{-0.000639024501353064}\\
8 & \texttt{0.000057059595936094}\\
\end{tabular}
\end{minipage}
\hfill
\begin{minipage}[t]{0.47\textwidth}
\centering\fontsize{9}{11}\selectfont\renewcommand{\arraystretch}{0.94}
\textbf{Dimension $d=18$}\par\vspace{3pt}
\begin{tabular}{@{}rr@{}}
$j$ & $c_j$\\\hline
1 & \texttt{-1.363089411037665455}\\
2 & \texttt{1.310154641526465102}\\
3 & \texttt{-0.753188385651569869}\\
4 & \texttt{0.318590501332819409}\\
5 & \texttt{-0.087214584768572412}\\
6 & \texttt{0.017122988693912971}\\
7 & \texttt{-0.001842640407911840}\\
8 & \texttt{0.000167505650662529}\\
\end{tabular}
\end{minipage}
\par\vspace{6pt}
\noindent
\begin{minipage}[t]{0.47\textwidth}
\centering\fontsize{9}{11}\selectfont\renewcommand{\arraystretch}{0.94}
\textbf{Dimension $d=19$}\par\vspace{3pt}
\begin{tabular}{@{}rr@{}}
$j$ & $c_j$\\\hline
1 & \texttt{-1.474194325773747982}\\
2 & \texttt{1.490523877762282901}\\
3 & \texttt{-0.921639869382412380}\\
4 & \texttt{0.417850615380876478}\\
5 & \texttt{-0.126433201115922045}\\
6 & \texttt{0.027475547666212328}\\
7 & \texttt{-0.003538232209301139}\\
8 & \texttt{0.000326927673992952}\\
9 & \texttt{-0.000000006500570853}\\
\end{tabular}
\end{minipage}
\hfill
\begin{minipage}[t]{0.47\textwidth}
\centering\fontsize{9}{11}\selectfont\renewcommand{\arraystretch}{0.94}
\textbf{Dimension $d=20$}\par\vspace{3pt}
\begin{tabular}{@{}rr@{}}
$j$ & $c_j$\\\hline
1 & \texttt{-1.579945474809915984}\\
2 & \texttt{1.670544725872577629}\\
3 & \texttt{-1.096740261161917251}\\
4 & \texttt{0.525461623853863191}\\
5 & \texttt{-0.170906953294798852}\\
6 & \texttt{0.039766945626827054}\\
7 & \texttt{-0.005655438627376280}\\
8 & \texttt{0.000530237000008302}\\
9 & \texttt{-0.000000007058310060}\\
\end{tabular}
\end{minipage}
\par\vspace{6pt}
\noindent
\begin{minipage}[t]{0.47\textwidth}
\centering\fontsize{9}{11}\selectfont\renewcommand{\arraystretch}{0.94}
\textbf{Dimension $d=21$}\par\vspace{3pt}
\begin{tabular}{@{}rr@{}}
$j$ & $c_j$\\\hline
1 & \texttt{-1.710888986722522764}\\
2 & \texttt{1.915840980982849828}\\
3 & \texttt{-1.358974851720881505}\\
4 & \texttt{0.705094047699715509}\\
5 & \texttt{-0.255936726616214749}\\
6 & \texttt{0.067805905436844796}\\
7 & \texttt{-0.012007024130289834}\\
8 & \texttt{0.001494166175592522}\\
9 & \texttt{-0.000087483458378013}\\
10 & \texttt{0.000006155975513212}\\
\end{tabular}
\end{minipage}
\hfill
\begin{minipage}[t]{0.47\textwidth}
\centering\fontsize{9}{11}\selectfont\renewcommand{\arraystretch}{0.94}
\textbf{Dimension $d=22$}\par\vspace{3pt}
\begin{tabular}{@{}rr@{}}
$j$ & $c_j$\\\hline
1 & \texttt{-1.832604677997432319}\\
2 & \texttt{2.158606485315783363}\\
3 & \texttt{-1.634973847658425017}\\
4 & \texttt{0.908659885468182438}\\
5 & \texttt{-0.361036469275248073}\\
6 & \texttt{0.106238823990947043}\\
7 & \texttt{-0.021856998384023118}\\
8 & \texttt{0.003220769629351691}\\
9 & \texttt{-0.000280549435446234}\\
10 & \texttt{0.000020038883914574}\\
\end{tabular}
\end{minipage}
\par\vspace{6pt}
\noindent
\begin{minipage}[t]{0.47\textwidth}
\centering\fontsize{9}{11}\selectfont\renewcommand{\arraystretch}{0.94}
\textbf{Dimension $d=23$}\par\vspace{3pt}
\begin{tabular}{@{}rr@{}}
$j$ & $c_j$\\\hline
1 & \texttt{-1.947858626197330301}\\
2 & \texttt{2.400359832064227206}\\
3 & \texttt{-1.922938413463691765}\\
4 & \texttt{1.132479341164536413}\\
5 & \texttt{-0.483562893094953761}\\
6 & \texttt{0.154151064019892037}\\
7 & \texttt{-0.035133563154735807}\\
8 & \texttt{0.005763978416708618}\\
9 & \texttt{-0.000595670794165965}\\
10 & \texttt{0.000043136442214113}\\
11 & \texttt{-0.000000000504942887}\\
\end{tabular}
\end{minipage}
\hfill
\begin{minipage}[t]{0.47\textwidth}
\centering\fontsize{9}{11}\selectfont\renewcommand{\arraystretch}{0.94}
\textbf{Dimension $d=24$}\par\vspace{3pt}
\begin{tabular}{@{}rr@{}}
$j$ & $c_j$\\\hline
1 & \texttt{-2.058035013443602512}\\
2 & \texttt{2.641322752468249213}\\
3 & \texttt{-2.220766702884943822}\\
4 & \texttt{1.373394535733419408}\\
5 & \texttt{-0.621363707651894304}\\
6 & \texttt{0.210809307189059353}\\
7 & \texttt{-0.051785471916799525}\\
8 & \texttt{0.009172399088840927}\\
9 & \texttt{-0.001047904852589911}\\
10 & \texttt{0.000076849427739887}\\
11 & \texttt{-0.000000000336875450}\\
\end{tabular}
\end{minipage}
\par